\documentclass[12pt]{amsart}
\usepackage[utf8]{inputenc}
\usepackage{graphicx}
\graphicspath{{Figures/}}
\usepackage{amssymb}
\usepackage{bbm}
\usepackage{bm}
\usepackage{mathtools}
\usepackage{amsthm}
\usepackage{amsbsy}
\usepackage{amstext}
\usepackage{amscd}
\usepackage{nccmath}
\usepackage[usenames]{xcolor}
\usepackage{enumerate}
\usepackage{float}
\allowdisplaybreaks

\newcommand{\ab}{\allowbreak}

\newcommand{\cP}{\mathcal{P}}
\newcommand{\cPS}{\mathcal{P\!S}}
\newcommand{\cU}{\mathcal{U}}
\newcommand{\cV}{\mathcal{V}}
\newcommand{\cW}{\mathcal{W}}

\newcommand{\C}{k}
\newcommand{\E}{\mathrm{E}}
\newcommand{\cov}{\mathrm{cov}}
\newcommand{\NC}{\mathit{NC}}
\newcommand{\Tr}{\mathrm{Tr}}
\newcommand{\tr}{\mathrm{tr}}
\newcommand{\mult}{\mathrm{mult}}
\newcommand{\Mob}{\textrm{M\"ob}}              
\newcommand{\bZ}{\mathbb{Z}}
\newcommand{\diag}{\mathrm{diag}}
\newcommand{\rdg}{\setbox0\hbox{\ \,%
\smash{\raise-0.755em\hbox{$\rightharpoonup$}}}%
\setbox1=\hbox{\rm{deg}}%
\leavevmode\vbox{\hsize\wd1\box0\box1}}
\newcommand{\ldg}{\setbox0\hbox{\ \,%
\smash{\raise-0.755em\hbox{$\leftharpoonup$}}}%
\setbox1=\hbox{\rm{deg}}%
\leavevmode\vbox{\hsize\wd1\box0\box1}}

\newcommand{\bs}{\kern-0.5em}

\allowdisplaybreaks

\newcommand{\ds}{\displaystyle}
\newcommand{\thebottomline}{\renewcommand{\thefootnote}{}
  \renewcommand{\footnoterule}{}
  \phantom{M}\footnotetext{\tiny{}\hfill
    \textit{\noindent\romannumeral\day.%
\romannumeral\month.\romannumeral\year}}}
\allowdisplaybreaks

\newtheorem{theorem}{Theorem}[section]
\newtheorem{proposition}[theorem]{Proposition}
\newtheorem{corollary}[theorem]{Corollary}
\newtheorem{lemma}[theorem]{Lemma}

\theoremstyle{definition}
\newtheorem{definition}[theorem]{Definition}
\newtheorem{notation}[theorem]{Notation}
\newtheorem{example}[theorem]{Example}
\newtheorem{remark}[theorem]{Remark}

\newcounter{jmpnumber}

\newcounter{dmpnumber}

\let\phi=\varphi
\newcommand{\ols}[1]{\overline{#1}}

\newcommand{\psN}{{\scriptscriptstyle (N)}}
\numberwithin{equation}{section}
\newcommand{\K}{\kappa}
\newcommand{\VG}{limit graph}
\newcommand{\VP}{limit partition}
\newcommand{\VPSet}{\mathcal{LP}}
\newcommand{\VGSet}{\mathcal{LG}}
\newcommand{\cycle}{\text{circuit}}
\newcommand{\cycles}{\text{circuit }}
\newcommand{\rp}{\textsc{r}}

\newcommand{\PSG}{\mathcal{PS}_{NC}^{(1,1,1)}(m_1,m_2,m_3)}
\newcommand{\PSSG}{\mathcal{PS}_{NC}^{(2,1,1)}(m_1,m_2,m_3)}
\newcommand{\PSSSG}{\mathcal{PS}_{NC}^{(1,1)}(m_1,m_2,m_3)}

\newcommand{\PS}{\mathcal{PS}_{NC_2}^{(1,1,1)}(m_1,m_2,m_3)}
\newcommand{\PSS}{\mathcal{PS}_{NC_2}^{(2,1,1)}(m_1,m_2,m_3)}
\newcommand{\PSSS}{\mathcal{PS}_{NC_{2,1,1}}^{(1,1,1)}(m_1,m_2,m_3)}
\newcommand{\PSSSS}{\mathcal{PS}_{NC_2}^{(1,1)(t)}(m_1,m_2,m_3)}
\newcommand{\PSSSSS}{\mathcal{PS}_{NC_2}^{(1,1)}(m_1,m_2,m_3)}
\newcommand{\PSSSSSS}{\mathcal{PS}_{NC_2^{(2)}}^{(1,1)}(m_1,m_2,m_3)}

\newcommand{\PSDUL}{\mathcal{PS}_{NC_2}^{(1,1)(t)}(m_1,m_2,m_3)}
\newcommand{\PSDUC}{\mathcal{PS}_{NC_2}^{(1,1)}(m_1,m_2,m_3)}
\newcommand{\PSDUCTwoThroughStrings}{\mathcal{PS}_{NC_2^{(2)}}^{(1,1)}(m_1,m_2,m_3)}

\newcommand{\TSTTs}{\text{$2$-$6$ tree type }}
\newcommand{\TSTTSet}{\mathcal{T}_{2,6}}

\newcommand{\TFFTTs}{\text{$2$-$6$ tree type }}
\newcommand{\TFFTTSet}{\mathcal{T}_{2,4,4}}

\newcommand{\TFUL}{\text{$2$-$4$ uniloop type}}
\newcommand{\TFULs}{\text{$2$-$4$ uniloop type }}
\newcommand{\TFULSet}{\mathcal{UL}_{2,4}}

\newcommand{\TFUC}{\text{$2$-$4$ unicircuit type}}
\newcommand{\TFUCs}{\text{$2$-$4$ unicircuit type }}
\newcommand{\TFUCSet}{\mathcal{UC}_{2,4}}

\newcommand{\DB}{double bicircuit type}
\newcommand{\DBs}{\text{double bicircuit type }}
\newcommand{\DBSet}{\mathcal{DB}}

\newcommand{\DUs}{\text{double unicircuit }}

\newcommand{\FirstCycleNotation}{[\![m_1 ]\!]}
\newcommand{\SecondCycleNotation}{[\![m_2 ]\!]}
\newcommand{\ThirdCycleNotation}{[\![m_3 ]\!]}
\newcommand{\CycleithNotation}{[\![m_i ]\!]}
\newcommand{\CyclejthNotation}{[\![m_j ]\!]}

\title[{Third order moments of the Wigner ensemble}] {Third order moments of\\ complex Wigner matrices$^{(*)}$}

\author[mingo]{James A. Mingo} \address{Department
  of Mathematics and Statistics, Queen's University, Jeffery
  Hall, Kingston, Ontario, K7L 3N6, Canada}

\email{mingo@mast.queensu.ca, 18dmg1@queensu.ca}

\author[munoz]{Daniel Munoz George}

\thanks{$^{(*)}$ Research supported by a Discovery Grant from
  the Natural Sciences and Engineering Research Council of
  Canada}


\begin{document}\thispagestyle{empty}

\begin{abstract}
We compute the third order moments of a complex Wigner matrix. We provide a formula for the third order moments $\alpha_{m_1,m_2,m_3}$ in terms of quotient graphs $T_{m_1,m_2,m_3}^{\pi}$ where $\pi$ is the Kreweras complement of a non-crossing pairing on the annulus. We prove that these graphs can be counted using the set of partitioned permutations, this permits us to write the third order moments in terms of the first, second and third order free cumulants which have a simple expression.

\end{abstract}

\maketitle

\section{Introduction}

In a series papers concluding with \cite{W2}, E.~Wigner showed that the
limit eigenvalue distribution of various ensembles of
self-adjoint random matrices was given by the semi-circle
law, now known as \textit{Wigner's semi-circle law}. The
matrices he considered are frequently referred to as Wigner
matrices. Much later, the global fluctuation moments of a
Wigner matrix were found by Khorunzhy, Khoruzhenko, and
Pastur in \cite{KKP}.  In \cite{MMPS} these global
fluctuation moments and their interactions with
deterministic matrices were described by the non-crossing
annular permutations introduced in \cite{MN}. In this paper
we shall show that the third order moments of a Wigner
matrix can be simply described by higher order free
cumulants using planar objects.

Let us describe the moments we consider. Let $\{ X_N \}_N$
be an ensemble of $N \times N$ random matrices. As in
Wigner's original work, we are interested in the large $N$
limit of the eigenvalue distribution, as measured by
$\{\Tr(X_N^n)\}_n$ where $\Tr(X_N)$ denotes the
un-normalized trace of $X_N$. Wigner's result was that as $N
\rightarrow \infty$, the moments $\E(\tr(X_N^n))$ converged
to the moments of the semi-circle law where $\tr =
N^{-1}\Tr$ is the normalized trace. In \cite{KKP} the second
cumulant or covariances $\{\cov\big(\Tr(X_N^m),
\Tr(X_N^n)\big)\}_{m,n}$ were studied and the two variable
Cauchy transform was computed, thus giving the moment
generating function for the covariances or fluctuation
moments.

For unitarily invariant ensembles the analysis of higher
order moments was achieved by the introduction higher order
free cumulants in \cite{CMSS}. This analysis did not directly apply
to Wigner matrices, however in, \cite[Ch.~5]{AGZ} and
\cite[\S4.4]{MS}, it was shown that the mixed moments of
Wigner matrices and constant matrices could be described by
free probability, using the language of non-crossing
partitions.  Recently, in \cite{MMPS}, the global
fluctuations of Wigner and constant matrices was analysed
and described by annular non-crossing pairings.  In this
paper we adapt the techniques of \cite{MMPS} to analyse the
third order moments using non-crossing annular pairings and
show that they have a simple description in terms of
free cumulants of order one, two, and three.

Indeed, recall that the semi-circular law can be described by
requiring that all free cumulants $\kappa_n = 0$ for $n >
2$, this is the free analogue of the description of the
Gauss law that is given by requiring that all classical
cumulants, $k_n$, vanish for $n > 2$. In the case of the
fluctuation moments, it was shown in \cite{MMPS} that the
fluctuation moments can be described by requiring that
$\kappa_2 = k_2$ and $\kappa_n = 0$ for $n > 2$ where $k_2$ is the second classical cumulant of an off-diagonal entry and
$\kappa_{p,q} = 0$ for $(p, q) \not = (2, 2)$, where
$\{\kappa_{p,q}\}_{p,q}$ are the second order free cumulants and
$\kappa_{2,2} = 2 k_4$ where $k_4$ is the fourth classical
cumulant of an off-diagonal entry.

In this paper we shall describe the third order moments
\[
\alpha_{m_1, m_2, m_3}
= \lim_N N \C_3(\Tr(X_N^{m_1}), \Tr(X_N^{m_2}), \Tr(X_N^{m_3}))
\]
where for random variables $A_1$, $A_2$, and $A_3$ we have $\C_3(A_1, A_2, A_3) = \E(A_1A_2A_3) - \{
\E(A_1A_2)\E(A_3) + \E(A_1A_3)\E(A_2) + \E(A_1)\E(A_2A_3) \}
+ 2 \E(A_1)\E(A_2)\E(A_3) $ is the third classical cumulant
of the random variables $A_1$, $A_2$, and $A_3$.  Our main
theorem is that the third order moment sequence $\{ \alpha_{m_1, m_2, m_2} \}_{m_1, m_2, m_3}$ can be expressed very simply in terms of third order cumulants as in the case of the moments of first and second order. 

\begin{theorem}[Main theorem]\label{thm:main}
The free cumulants of a Wigner matrix, up to order three, are given by $\K_2=1, \K_{2,2}=2\C_4, \K_{2,2,2}=4\C_6,
\K_{2,1,1}\ab = \mathring{\C}_4-2\C_4$ and $0$ otherwise. Here by $\C_6$ we mean
$\C_6(x_{1,2},x_{1,2}, x_{1,2}, x_{2,1}, \ab x_{2,1},x_{2,1})$, by
$\C_4$ we mean $\C_4(x_{1,2},x_{1,2},x_{2,1},x_{2,1})$, and by
$\mathring{\C}_4$ we mean $\C_4(x_{1,1},x_{1,1},x_{1,1},\ab x_{1,1})$.
\end{theorem}

The higher order free cumulants are symmetric so $\kappa_{m_1, m_2} = \kappa_{m_2, m_1}$ and $\kappa_{m_1, m_2, m_3} = \kappa_{m_{\sigma(1)}, m_{\sigma(2)}, m_{\sigma(3)}}$, and $\sigma$ is any permutation on $\{1, 2, 3\}$. Note that in a significant departure from the first and second order case, the distribution of the diagonal entries of the matrix $X_N$ appears in the limit distribution. 

Our result depends on the moment-cumulant relation in \cite[\S7.2]{CMSS}
\begin{equation}\label{ThirdOrdercaseCumulants}
\alpha_{m_1,m_2,m_3}=\sum_{(\cU,\pi)\in \mathcal{PS}_{NC}(m_1,m_2,m_3)}\K_{(\cU,\pi)}.
\end{equation}
In \S 2 and \S3 we shall describe the set $\cPS_{\NC}(m_1, m_2, m_3)$ of third order non-crossing partitioned permutations and the higher order cumulants $\K_{(\cU,\pi)}$.

Recently it has been shown that higher order free cumulants have an interpretation in terms of a phenomenon, known as topological recurrence, which is a universal recurrence on the genus of a map on a surface and the number of its boundaries. In \cite{bcgf}, Borot, Charbonnier, Garcia-Failde, Leid, and Shadrin showed that this relation between higher order moments and higher order cumulants can be expressed in the form of power series connecting fluctuation moments to higher order free cumulants. This was already done in \cite{CMSS} in the second order case. In this paper we compute, for the first time in a non-trivial concrete example, the third order cumulants.

The basic method is to write a fluctuation moment as a sum indexed by some graphs. Then we determine which graphs survive in the large $N$ limit. Then we count these graphs and determine their corresponding weight in terms of the cumulants of the entries of the Wigner matrix. Finally we show in Theorem \ref{thm:main}, that this sum has a very simple expression in  terms of higher order free cumulants.

After this introduction, the organization of the paper is as  follows.
In Section \ref{sec:non-crossin-permutations} we review some basic material on non-crossing permutations on a multi-annulus. In Section \ref{sec:partitioned-permutations} we describe the exact situation in the case of three circles. In Section \ref{sec:wigner-matrices} we specify the random matrix model we will be working with. In Section \ref{Section:GraphTheory} we prove some basic properties of our graphs that we will need to establish our main result. In Section \ref{sec:asymptotic-expressions} we show that limiting third order moments exit and in  Section \ref{sec:leading-order} we identify the graphs that correspond to these moments. In Section \ref{sec:counting-the-graphs} we count the limit graphs and in Section \ref{sec:free-cumulants} we express this in terms of partitioned permutations and higher order free cumulants, which will conclude the proof. 

\section{Non-crossing permutations on a multi-annulus}
\label{sec:non-crossin-permutations}
Higher order free cumulants are described in terms of
partitioned permutations, so we shall start with these.  We
begin by recalling from \cite{NS} some basic facts about
non-crossing partitions and free cumulants. We let $[n] =
\{1, 2, \dots, n\}$ and $\cP(n)$ be the set of partitions of
$[n]$.  Our usual notation will be to denote a
\textit{partition} by $\pi=\{V_1, \dots, V_k\}$ with $V_1$,
\dots, $V_k$, the \textit{blocks} of $\pi$. By this we mean
that $V_1 \cup \cdots \cup V_k = [n]$ and $V_i \cap V_j =
\emptyset$ for $i \not= j$. $0_n$ denotes the partition of
$[n]$ with all blocks singletons and $1_n$ denotes the
partition of $[n]$ with only one block. We say $\pi$ is
\textit{non-crossing} if we cannot find $i < j < k <l \in
       [n]$ such that $i$ and $k$ are in one block of $\pi$
       and $j$ and $l$ are in a different block of $\pi$. We
       denote by $\NC(n)$ the subset of $\cP(n)$ consisting
       of non-crossing partitions.
       
If $\pi \in \cP(n)$ and $A \subseteq [n]$ we can form $\pi|_A$, the \textit{restriction} of $\pi$ to $A$ as follows. If the blocks of $\pi$ are $V_1, \dots, V_k$, then the blocks of $\pi|_A$ are the non-empty elements of $V_1 \cap A, \dots, V_k \cap A$.        

As one can see from the definition, the order in $[n]$ plays
an important role for elements of $\NC(n)$, even though the
blocks of a non-crossing partition are an ordered set.  In
this paper it will be very important to switch our focus to
non-crossing permutations. We let $S_n$ be the group of
permutations of $[n]$. We say that $\pi \in S_n$ is
\textit{non-crossing}, if the partition obtained from its
cycle decomposition is non-crossing.  Using a theorem of
Biane we can describe easily the permutations that give
non-crossing partitions.

For a partition or a permutation, $\pi$, we let $\#(\pi)$
denote the number of blocks of $\pi$ if it is a partition
and the number of cycles of $\pi$ if it is a permutation. We
let $\gamma_n = (1, 2, \dots, n) \in S_n$ be the permutation
with one cycle and the elements in increasing order. We
compose our permutations from right to left: $\pi\gamma_n$
means perform $\gamma_n$ first then $\pi$. When there is no
risk of confusion we shall often write $\gamma$ for
$\gamma_n$ to give us a more compact notation.  Biane's rule
is that for $\pi \in S_n$ we have $\#(\pi) +
\#(\pi^{-1}\gamma_n) \leq n + 1$ with equality only if $\pi$
is non-crossing. See \cite[Ch.~5]{MS} for more details and
references. Another way to write Biane's rule is to use a
\textit{length} function on $S_n$.

\begin{figure}
  \begin{center}\includegraphics{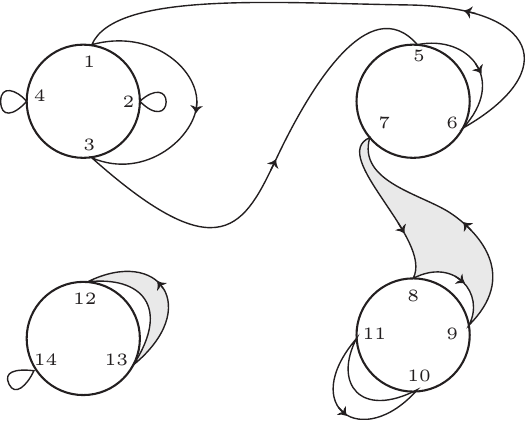}\hfill\includegraphics{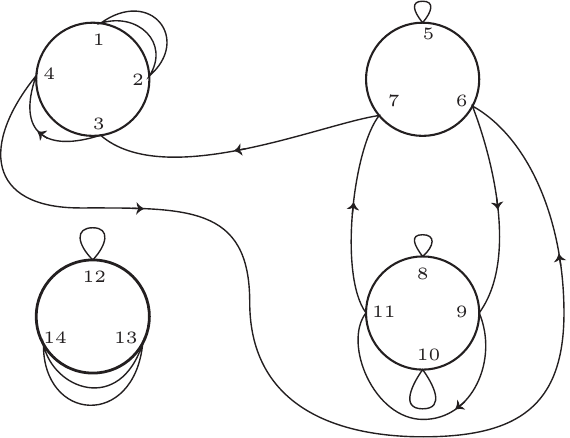}\end{center}
  
\caption{\label{fig:exact_factorization}\small In this
  example \begin{center}$\gamma$ = (1, 2, 3, 4) (5, 6, 7)(8,
    9, 10, 11)(12, 13, 14),
    and\end{center} \begin{center}$\pi$ = (1,3,5,6)
      (2)(4)(7,8,9)(10,11)(12,13)(14)\end{center} is shown
    on the left and
    \begin{center}
      $\pi^{-1}\gamma$ =
      (1,2)(3,4,6,9,11,7)(5)(8)(10)(12)(13,14)
    \end{center} is shown on the
  right. We have drawn the 4 cycles of $\gamma$ as circles
  and the 7 cycles of $\pi$ in a non-crossing way on these 4
  circles. Note that $|\pi| = 7$, $|\pi^{-1}\gamma| = 7$,
  $|\gamma| = 10$, and $|\pi \vee \gamma | = 12$. Thus
  $|\pi| + |\pi^{-1}\gamma| + |\gamma| = 2 |\pi \vee
  \gamma|$, as claimed in $(i)$ of Proposition
  \ref{prop:equality_in_triangle}. Note also that $\pi$
  connects the first three circles of $\gamma$, so we let
  $\cV= \{ (1,3,5,6),(2),(4),(7,8,9,
  12,\ab13),(10,11),(14)\}$. Each block of $\cV$ is a single
  cycle of $\pi$, except the block $(7,8,\ab9, 12,13)$,
  which is the union of the 2 cycles of $\pi$ shown in grey.
  $\pi \vee \gamma = (1,2,3,4,5,6,7,8,9,10,11),(12,13,14)$
  and $\cV \vee \gamma = 1_{14}$. So $|\cV \vee \gamma| =
  13$, $|\pi \vee \gamma | = 12$, and $|\cV \vee \gamma| -
  |\pi \vee \gamma | = 1$ . On the other hand $|\cV| = 8$,
  $|\pi| = 7$, so $|\cV| - |\pi|= 1$. Thus $|\cV \vee \gamma
  | - |\pi \vee \gamma| = |\cV| - |\pi|$ as claimed in
  $(ii)$ of Proposition \ref{prop:equality_in_triangle}. We
  let $\cW = 0_{\pi^{-1}\gamma} =
  \{(1,2),(3,4,6,9,11,7),(5),(8),(10),(12)(13, 14)\}$. Then
  $\cW \vee \gamma = \pi \vee \gamma$, so $|\cW \vee \gamma|
  - |\pi \vee \gamma| = 0 = |\cW| - |\pi^{-1}\gamma|$, as
  claimed in $(iii)$ of Proposition
  \ref{prop:equality_in_triangle}. Finally $\cW \vee \gamma
  = \cV \vee \gamma$ and $\cV \vee \cW = 1_{14}$. Thus $|\cV
  \vee \cW| - |\pi \vee \gamma| = 1$. On the other hand
  $|\cV| - |\pi| + |\cW| - |\pi^{-1}\gamma| = 1 = |\cV \vee
  \cW| - |\pi \vee \gamma|$, as claimed in $(iv)$ of
  Proposition \ref{prop:equality_in_triangle}.}
\end{figure}

For a permutation $\pi \in S_n$ we let $|\pi|$ be the
minimal number of transpositions required to write $\pi$ as
a product of transpositions. One can then check that
$\#(\pi) + |\pi| = n$. One immediately has the triangle
inequality $|\pi\sigma| \leq |\pi| + |\sigma|$ for any $\pi,
\sigma \in S_n$. Then Biane's inequality becomes $|\gamma_n|
\leq |\pi| + |\pi^{-1} \gamma_n|$ with equality only if
$\pi$ is non-crossing. Because of this we shall say that
$\pi$ is non-crossing \textit{relative} to $\gamma_n$.

In this paper we will be concerned with permutations on a
\textit{multi-annulus}. In terms of permutations this means
we consider a permutation $\gamma$ with one or more cycles. If $m_1,
m_2, \dots, m_r$ are positive integers, $n = m_1 + m_2 +
m_3$, we let $\gamma_{m_1,m_2, \dots, m_r}$ be the
permutation in $S_n$ with $r$ cycles, the $i^{th}$ cycle
being $[\![m_i ]\!] := (m_1\ab + \cdots + m_{i-1} +1, \dots,
m_1 + \dots + m_i)$. When $r = 1$ we are in the case of
$\NC(n)$. When $r > 1$ we are in the case of a
\textit{r-annulus}. In \cite[Def.~3.5]{MN}, conditions were
given, similar to the ones for non-crossing partitions
above, that ensured that when $r = 2$ we can draw two
circles representing the cycles of $\gamma_{m_1, m_2}$ and
then arrange the cycles of $\pi$ so that the cycles do not
cross, assuming that there is at least one cycle of $\pi$
that meets both cycles of $\gamma_{m_1, m_2}$, i.e $\pi \vee
\gamma_{m_1, m_2} = 1_{m_1 + m_2}$. For an illustration of
what we mean by non-crossing when $r= 4$, see Figure
\ref{fig:exact_factorization}.  In \cite{MN}, it was shown
that the non-crossing condition was equivalent to the metric
condition: $|\pi| + |\pi^{-1}\gamma_{m_1, m_2}| =
|\gamma_{m_1, m_2}| + 2$.

\begin{notation}
We let for $r \geq 1$ and $n = m_1 + \cdots + m_r$ and
$\gamma = \gamma_{m_1, \dots, m_r}$
\begin{multline*}
S_{\NC}(m_1, m_2, \dots, m_r)\\ = \{ \pi \in S_n \mid |\pi|
+ |\pi^{-1} \gamma| = |\gamma| + 2(r -1) \textrm{ and } \pi
\vee \gamma = 1_n\};
\end{multline*}
\begin{multline*}
\NC_2(m_1, \dots, m_r) \\ = \{ \pi \in S_\NC(m_1, \dots,
m_r) \mid \textrm{\ all cycles of\ } \pi \textrm{\ have 2
  elements\ }\},
\end{multline*}
A cycle of $\pi\in \NC_2(m_1,\dots,m_r)$ consisting of two elements
from distinct cycles of $\gamma$ is called a \textit{through string}.
The set of all permutations $\pi\in \NC_2(m_1,\dots,m_r)$ with exactly $k$
through strings is denoted by $\NC_2^{(k)}(m_1, \dots, m_r)$.
In addition we set
\[
\NC(m_1) \times \cdots \times \NC(m_2) = \{ \pi \in S_n \mid
|\pi| + |\pi^{-1} \gamma| = |\gamma| \textrm{\ and\ } \pi
\leq \gamma\}.
\]
In the last definition we mean by $\pi \leq \gamma$ that all
cycles of $\pi$ are contained in some cycles of
$\gamma$. Finally there are the non-crossing permutations
that lie in between these two. For $(i_1, i_2, i_3)$ some
permutation of $(1, 2, 3)$. Let
\begin{multline*}
S_\NC(m_{i_1}, m_{i_2}) \times \NC(m_{i_3}) = \{\pi \in S_n
\mid \#(\pi) + \#(\pi^{-1}\gamma) = n+1\\ \textrm{\ and\ }
\pi \vee \gamma = \{[\![ m_{i_1} ]\!] \cup [\![ m_{i_2}
  ]\!], [\![ m_{i_3} ]\!] \}.
\end{multline*}
These are non-crossing, but $\pi$ only connects the cycle
$[\![ m_{i_1} ]\!]$ to $[\![ m_{i_2} ]\!]$.
\end{notation}

\section{Partitioned permutations on 3 cycles}
\label{sec:partitioned-permutations}
\begin{notation}
As above, let $\cP(n)$ be the partitions of $[n] = \{1 ,2 ,
\dots,\ab n\}$ and $S_n$ be the permutations of $[n]$.  For
$\cU \in \cP(n)$ we let $\#(\cU)$ denote the number of
blocks of $\cU$ and $|\cU| = n - \#(\cU)$, the
\textit{length} of $\cU$. $|\cU|$ is the number of joins one
has to perform on $0_n$ to get $\cU$; here performing a
\textit{join} on a partition means joining two blocks. We
have $|0_n| = 0$ and $|1_n| = n-1$.

For $\pi \in S_n$ recall that $|\pi| = n - \#(\pi)$ and $|\pi|$ is
the minimal number of transpositions needed to write $\pi$
as a product of transpositions. Since multiplying a
permutation by the transposition, $(i, j)$, reduces the
number of cycles by 1, if $i$ and $j$ were in different
cycles of the permutation, and increases it by $1$ if $i$
and $j$ are in the same cycle of the permutation, we see
that the length of a permutation and the length of the
partition obtained from it cycles are the same.

When appropriate, we shall, for a permutation $\pi$, use the
same symbol to denote the partition $\pi$ whose blocks are
the cycles of the permutation $\pi$.

If $\pi \in S_n$ and $\cU \in \cP(n)$ we write $\pi \leq
\cU$ to mean that every cycle of $\pi$ is contained in some
block of $\cU$.  We call the pair $(\cU, \pi)$ a
\textit{partitioned permutation}. We let $|(\cU, \pi)| =
2|\cU| - |\pi|$ and call this the \textit{length} of $(\cU,
\pi)$.

If we let $p = |\cU| - |\pi| = \#(\pi) - \#(\cU)$, then $p$
is the number of cycles of $\pi$ joined by $\cU$. For
example, when $p = 1$, this means that one block of $\cU$
contains two cycles of $\pi$ and all other blocks of $\cU$
contain only one cycle of $\pi$.  If $p = 2$ this means that
either one block of $\cU$ contains three cycles of $\pi$ or
two blocks of $\cU$ each contain 2 cycles of $\pi$.

\end{notation}

Two crucial properties of this length function are
\cite[Lemma 2.2]{MSS} and \cite[Prop. 5.6]{CMSS}.

\begin{proposition}[Triangle Inequality]
Given partitioned permutations $(\cV, \pi)$ and $(\cU,
\sigma)$ we have
\[
|(\cV \vee \cU, \pi\sigma)| \leq |(\cV, \pi)| + |(\cU, \sigma)|.
\]

\end{proposition}

\begin{proposition}[Equality in the triangle inequality] 
  \label{prop:equality_in_triangle}
Let $(\cV, \pi)$ and $(\cW, \ab\pi^{-1} \gamma)$ be
partitioned permutations of $[n]$ such that we have equality
in the triangle inequality:
\[
|(\cV, \pi)| + |(\cW, \pi^{-1}\gamma)| = |(\cV \vee \cW,
\gamma)|.
\]
Then 
\begin{enumerate}

\item
$|\pi| + |\pi^{-1}\gamma| + |\gamma| = 2 |\pi \vee \gamma|$

\item
$|\cV \vee \gamma| - |\pi \vee \gamma | = |\cV| - |\pi|$

\item
$|\cW \vee \gamma| - |\pi \vee \gamma| = |\cW| - |\pi^{-1}\gamma|$

\item
$|\cV| - |\pi| + |\cW| - |\pi^{-1}\gamma| =
|\cV \vee \cW| - |\pi \vee \gamma|$

\end{enumerate}

\end{proposition}

\begin{figure}
\begin{center}\includegraphics{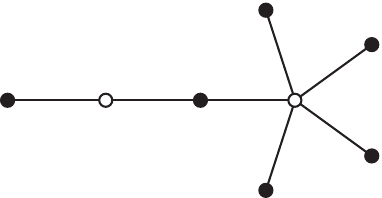}\end{center}
\caption{\small\label{fig:partition_graph} We have shown the
  graph described in Remark \ref{rem:partition_graph} for
  the partitioned permutation $(\cV, \pi)$ illustrated in
  Figure \ref{fig:exact_factorization}. The 7 edges
  correspond to the 7 cycles of $\pi$. The 2 white vertices
  correspond to the 2 blocks of $\pi \vee \gamma$. The graph
  does not take into the account of the planarity of the
  pair $\{\pi, \gamma\}$, but rather the minimality of the
  blocks of $\cV$. The fact that we get a tree corresponds
  to the equality in $(ii)$ of Proposition
  \ref{prop:equality_in_triangle}, See Lemma
  \ref{lemma:tree_lemma}.}
\end{figure}

\begin{remark}
Let us recall from \cite[Notation 5.13]{CMSS} what each of
these four conditions means.  First, $(i)$ means that if we
make the cycles of $\gamma$ into circles (see Figure
\ref{fig:exact_factorization}) then we can draw the cycles
of $\pi$ in a planar way.  We call this the \textit{planar} condition. Condition $(ii)$ describes the
way that $\cV$ can be a union of cycles of $\pi$.

Since $\pi \leq \cV$, we must have that each block of $\cV$
is a union of cycles of $\pi$; however $(ii)$ says that if
$c_1$ and $c_2$ are cycles of $\pi$ and are in the same block of
$\cV$ then these cycles may not meet the same cycle of
$\gamma$, i.e. they lie in different blocks of $\pi \vee
\gamma$. For example in Figure
\ref{fig:exact_factorization}, the cycles $(7,8,9)$ and
$(12, 13)$ are in the same block of $\cV$, but there is no
cycle of $\gamma$ touched by both of these cycles of
$\pi$. Condition $(iii)$ says that the same holds for the
pair $(\cW, \pi^{-1}\gamma)$. Condition $(iv)$ says that the
sum of the number of joins of the cycles of $\pi$ made by
$\cV$ and the number of joins of the cycles of
$\pi^{-1}\gamma$ made by $\cW$ equals the number of joins of
the blocks of $\pi \vee \gamma$ made by $\cV \vee \cW$.
\end{remark}

\begin{remark}
Now let us consider what happens in Proposition
\ref{prop:equality_in_triangle} when $\cW =
0_{\pi^{-1}\gamma}$ and $\cV \vee \gamma = 1_n$. First $(i)$
does not involve the partitions $\cV$ and $\cW$, so nothing changes
here. Next, $(ii)$ becomes $|1_n| - |\pi \vee \gamma| =
|\cV| - |\pi|$. Finally, $(iii)$ and $(iv)$ are tautologies
when $\cW = 0_{\pi^{-1}\gamma}$.
\end{remark}

\begin{remark}\label{rem:partition_graph}
Suppose $(\cV, \pi)$ is a partitioned permutation.  Let us
create an unoriented bipartite graph, $\Gamma(\cV, \pi,
\gamma)$. We let the edge set $E = \{ e_c \mid c$ is a cycle
of $\pi\}$ be the cycles of $\pi$ and the vertex set $V$ be
the union of the blocks of $\cV$ (the \textit{black}
vertices) and the blocks of $\pi \vee \gamma$ (the
\textit{white} vertices). Each edge $e_c$ connects the block
of $\cV$ containing $c$ to the block of $\pi \vee \gamma$
containing $c$. See Figure \ref{fig:partition_graph}. If
$\cV \vee \gamma = 1_n$ then item $(ii)$ of Proposition
\ref{prop:equality_in_triangle} is satisfied if and only if
$(V, E)$ is a tree. If $\cV \vee \gamma < 1_n$ then item
$(ii)$ is satisfied if and only if $(V, E)$ is disjoint
union of trees.
\end{remark}

In this paper we will deal with several types of graphs. For
the sake of clarity we will be specific about our
terminology.

By a \textit{graph} we mean a pair of sets $(V, E)$. $V$
denotes the set of \textit{vertices} of our graphs and $E$
denotes the set of edges. For us, an \textit{edge} $e$ is an
unordered set of two vertices $\{u, v\}$. So with this
convention $\{u, v\} = \{v, u\}$. If $u = v$, we will
say that $\{u, v\}$ is a \textit{loop}. Since we have not
oriented the edges we will have an unoriented graph.

\begin{lemma}\label{lemma:tree_lemma}
The graph $\Gamma(\cV, \pi, \gamma)$ is a union of trees if
and only if $(ii)$ of Proposition
\ref{prop:equality_in_triangle} holds. If condition $(ii)$
holds, the number of trees is $\#(\pi \vee \gamma)$.
\end{lemma}

\begin{proof}
We just have to check that
\[
|V| - |E| = \#(\cV \vee \gamma) + (|\cV \vee \gamma| - |\pi
\vee \gamma|) - (|\cV| - |\pi|)
\]
and this follows easily from the definitions. When $\#(\cV
\vee \gamma) = 1$, we have that $\Gamma(\cV, \pi, \gamma)$
is a tree. If $\#(\pi \vee \gamma) > 1$, we apply the
previous result to the restriction of $(\cV, \pi)$ to each
block of $\cV \vee \gamma$, thus obtaining that the
restriction is a tree.
\end{proof}

\setbox1=\hbox{\includegraphics{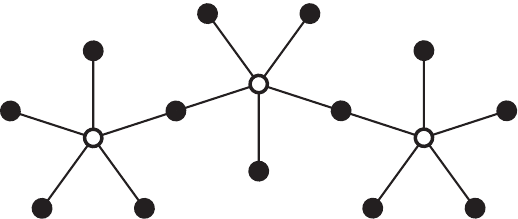}}
\setbox2=\hbox{\includegraphics{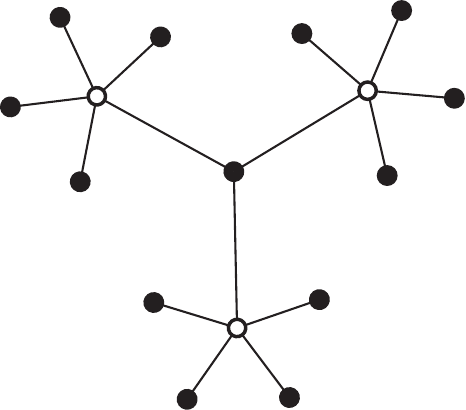}}

\begin{figure}
  \begin{center}
    \hfill\includegraphics{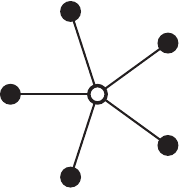}\hfill\includegraphics{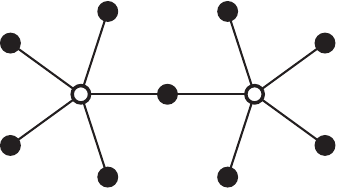}\hfill\hbox{}
  \end{center}

\medskip

\begin{center}\hfill
  $\vcenter{\hsize=\wd1\box1}$\hfill$\vcenter{\hsize=\wd2\box2}$\hfill\hbox{}
\end{center}

\caption{\small\label{fig:three_circles} When $\gamma
  =\gamma_{m_1, m_2, m_3}$ there are 4 possibilities for
  $\Gamma(\cV, \pi, \gamma)$. Let $n = m_1 + m_2 + m_3$. In
  Lemma \ref{lemma:factorization_lemma} we have presented
  the 4 cases where each of these graphs arises. When $\pi
  \vee \gamma = 1_n$ we have $\cV = 0_\pi$; there is only 1
  white vertex and the number of black vertices is
  $\#(\pi)$. This is illustrated in the upper left, $(i)$ in
  Lemma \ref{lemma:factorization_lemma}. When $\#(\pi \vee
  \gamma) = 2$ we have 2 white vertices, one black vertex of
  degree 2 and all other black vertices of degree 1. In this
  case one block of $\cV$ contains 2 cycles of $\pi$ and all
  other contain 1 cycle of $\pi$.  This is illustrated on
  the upper right, $(ii)$ in Lemma
  \ref{lemma:factorization_lemma}. In the bottom row we have
  the 2 cases when $\#(\pi \vee \gamma) = 3$. On the left,
  $(iii)$ in Lemma \ref{lemma:factorization_lemma}, we have
  two blocks of $\cV$ which each contain 2 cycles of $\pi$
  and on the right, $(iv)$ in Lemma
  \ref{lemma:factorization_lemma}, one block of $\cV$
  contains 3 cycles of $\pi$. }\end{figure}

\begin{lemma}\label{lemma:factorization_lemma}
Let $n = m_1 + m_2 + m_3$, $\gamma = \gamma_{m_1, m_2, m_3}$
and $(\cV, \pi)$ be a partitioned permutation such that $\cV
\vee \gamma = 1_n$ and
\[
|(\cV, \pi)| + |(0_{\pi^{-1}\gamma}, \pi^{-1} \gamma)| =
|(1_n, \gamma)|.
\]
Then, either
\begin{enumerate}

\item
$\cV = 0_\pi$, and $\pi \vee \gamma = 1_n$ and $\pi \in
  S_{NC}(m_1, m_2, m_3)$;

\item
$\pi =\pi_1 \times \pi_2 \in S_{NC}(m_{i_1}, m_{i_2}) \times
  \NC(m_{i_3})$ for some permutation $i_1, i_2, i_3$ of $1,
  2, 3$, $\#(\cV) = \#(\pi) - 1$, and $\cV$ joins a cycle of
  $\pi_1$ with a cycle of $\pi_2$;

\item
$\pi = \pi_1 \times \pi_2 \times \pi_3 \in NC(m_1) \times
  NC(m_2) \times NC(m_3)$, $\#(\cV) = \#(\pi) - 2$ and,
  $\cV$ joins a cycle of $\pi_{i_1}$ with a cycle of
  $\pi_{i_2}$ in one block and joins a cycle $\pi_{i_2}$
  with a cycle of $\pi_{i_3}$ into another block of $\cV$,
  with $i_1, i_2, i_3$ some permutation of $1, 2, 3$;

\item
$\pi = \pi_1 \times \pi_2 \times \pi_3 \in NC(m_1) \times
  NC(m_2) \times NC(m_3)$, $\#(\cV) = \#(\pi) - 2$ and,
  $\cV$ joins a cycle of $\pi_{1}$, a cycle of $\pi_{2}$ and
  a cycle of $\pi_3$ into a single block of $\cV$.
\end{enumerate}
\end{lemma}

\begin{remark}\label{remark:four_cases}
Before presenting the proof, note that the four cases in
Lemma \ref{lemma:factorization_lemma} are illustrated in
Figure \ref{fig:three_circles}, in the order presented in
Lemma \ref{lemma:factorization_lemma}, starting at the upper
hand corner.
\end{remark}
\begin{proof}
Let us begin by labelling the three cycles of $\gamma =
\gamma_{m_1, m_2, m_3}$
\[
=(1, 2, \dots, m_1)(m_1 + 1, \dots, m_1 + m_2)
 (m_1 + m_2 + 1,\dots, m_1+m_2+m_3)
\]
by $\gamma_{m_1}$, $\gamma_{m_2}$, and $\gamma_{m_3}$
respectively.  Note that $1 \leq \#(\pi \vee \gamma) \leq
\#(\gamma) = 3$, so there are three possibilities for
$\#(\pi \vee \gamma)$.

\smallskip\noindent\textit{Case}$(a)$: $\#(\pi \vee \gamma)
= 1$. In this case property $(i)$ of Proposition
\ref{prop:equality_in_triangle} means that $\pi$ is
non-crossing with respect to $\gamma$ and the cycles of
$\pi$ connect the cycles of $\gamma$. Thus $\pi \in
S_{NC}(m_1, m_2, m_3)$. Property $(ii)$ of Proposition
\ref{prop:equality_in_triangle} means that $|\cV| - |\pi| =
1_n - |\pi \vee \gamma| = 0$. So $\cV = 0_\pi$ and we are in
case $(i)$ of the Lemma.

\medskip\noindent\textit{Case}$(b)$: $\#(\pi \vee \gamma) =
2$. Then $\#(\pi \vee \gamma) = \#(\gamma) -1$ so the cycles
of $\pi$ connect two cycles of $\gamma$, say
$\gamma_{m_{i_1}}$ and $\gamma_{m_{i_2}}$. We consider the
cycles of $\pi$ that are contained in their union,
$\gamma_{m_{i_1},m_{i_2}}$, and call this permutation
$\pi_1$. We let $\pi_2$ be the cycles of $\pi$ that are in
contained in the remaining cycle $\gamma_{m_{i_3}}$. Then
property $(i)$ of Proposition
\ref{prop:equality_in_triangle} means that $\#(\pi_1) +
\#(\pi_1^{-1} \gamma_{m_{i_1},m_{i_2}}) = m_{i_1} + m_{i_2}$
and $\#(\pi_2) + \#(\pi_2^{-1}\gamma_{m_{i_3}}) = m_{i_3} +
1$. So $\pi_1 \in S_{NC}(m_1, m_2)$ and $\pi_2 \in
NC(m_{i_3})$. Thus $\pi \in S_{NC}(m_{i_1}, m_{i_2}) \times
NC(m_{i_3})$. Next, condition $(ii)$ of Proposition
\ref{prop:equality_in_triangle} means that $\#(\cV) =
\#(\pi) - 1$; so one block of $\cV$ contains two cycles of
$\pi$ and all other blocks contain one cycle of $\pi$. Since
$\cV \vee \gamma = 1_n$, we have that $\cV$ joins a cycle of
$\pi$ with a cycle of $\pi_2$.

\medskip\noindent\textit{Case}$(c)$: $\#(\pi \vee \gamma) =
3$ and no block of $\cV$ contains more that two cycles of
$\pi$. Then $|\cV| - |\pi| = 2$ and no block of $\cV$
contains more than two cycles of $\pi$. Hence two blocks of
$\cV$ each contain 2 cycles of $\pi$. Since $\#(\pi \vee
\gamma) = \#(\gamma)$ we have that each cycle of $\pi$ is
contained in some cycle of $\gamma$, so let us label the
restrictions of $\pi$ to the cycles $\gamma_{ m_1 }$,
$\gamma_{ m_2 }$, $\gamma_{ m_ 3}$ as $\pi_1$, $\pi_2$, and
$\pi_3$ respectively. By condition $(i)$ of Proposition
\ref{prop:equality_in_triangle} we have
\begin{multline*}
\#(\pi_1) + \#(\pi_1^{-1}\gamma_{m_1})
+
\#(\pi_2) + \#(\pi_2^{-1}\gamma_{m_2})\\
+
\#(\pi_3) + \#(\pi_3^{-1}\gamma_{m_3}) 
=
\#(\pi) + \#(\pi^{-1}\gamma) = n + 3.
\end{multline*}
Since $\#(\pi_i) + \#(\pi_i^{-1} \gamma_{m_i}) \leq m_i +
1$, we must have equality for $i = 1, 2, 3$. Thus $\pi_i \in
NC(m_i)$, for $i = 1, 2, 3$.

\medskip\noindent\textit{Case}$(d)$: $\#(\pi \vee \gamma) =
3$ and there is a block of $\cV$ which contains more that
two cycles of $\pi$. As above we have each cycle of $\pi$ is
contained in a cycle of $\gamma$ and thus $\pi \in NC(m_1)
\times NC(m_2) \times NC(m_3)$. Since $\cV \vee \gamma =
1_n$ we must have that the block of $\cV$ that contains 3
cycles of $\pi$ must contain one cycle contained in each
cycle of $\gamma$.
\end{proof}

\begin{notation}
We let denote by $\cPS_n$ the set of partitioned
permutations $(\cV, \pi)$ with $\pi \leq \cV$. We make
$\cPS_n$ a poset as follows. Given $(\cV, \pi)$ and $(\cU,
\gamma)$ in $\cPS_n$ we say $(\cV, \pi) \leq (\cU, \gamma)$
if
\begin{enumerate}
\setcounter{enumi}{3}

\item
$\cV \leq \cU$,

\item
$\Gamma(\cV, \pi, \gamma)$ is a union of trees, and 

\item
$|\pi| + |\pi^{-1}\gamma| +|\gamma| = 2 |\pi \vee \gamma|$.
\end{enumerate}
We let $\cPS_\NC(\cU, \gamma) = \{ (\cV, \pi) \in \cPS_n
\mid (\cV, \pi) \leq (\cU, \gamma)$ and $\cV \vee \gamma =
\cU\}$.

If $\gamma = \gamma_n$ and $\cU = 1_n$ then $\cPS_\NC(\cU,
\gamma) = \{ (0_\pi, \pi) \mid \pi \in\NC(n)\}$ because
$\Gamma(\cV, \pi, \gamma)$ will have 1 white vertex, one
edge for each cycle of $\pi$ and one black vertex for each
block of $\cV$. If a block of $\cV$ were to contain more
than one cycle of $\pi$ then $\Gamma(\cV, \pi, \gamma)$
would have a circuit and thus fail to be a tree; hence $\cV
= 0_\pi$. Since $\gamma$ has only one cycle, condition
$(vi)$ above forces $|\pi| + |\pi^{-1}\gamma| = |\gamma|$,
which is the non-crossing condition; thus $\pi \in
\NC(n)$. For convenience we write $\cPS_\NC(n) = \NC(n)
=\cPS_\NC(1_n, \gamma_n)$.

If $\gamma = \gamma_{m_1,m_2}$, $n = m_1 + m_2$, and $\cU =
1_n$ then $\cPS_\NC(\cU, \gamma) = \{ (0_\pi, \pi) \mid \pi
\in S_\NC(m_1, m_2)\} \cup \{(\cV, \pi) \mid \pi \in
\NC(m_1) \times \NC(m_2), \cV \vee \gamma = 1_n$ and $|\cV|
= |\pi| + 1 \}$. In the first part we have $\pi \vee \gamma
= 1_n$ and in the second part $\pi \leq\gamma$. In this case
we shall write $\cPS_\NC(m_1, m_2) = \cPS_\NC(1_n,
\gamma_{m_1, m_2})$.

Finally suppose $n = m_1 + m_2 + m_3$, $\gamma =
\gamma_{m_1, m_2, m_3}$, and $\cU = 1_n$. Then
$\cPS_\NC(\cU, \gamma)$, which we shall write 
as $\cPS_\NC(m_1, m_2,m_3)$, is the union of 4 subsets
corresponding to the four graphs in Figure
\ref{fig:three_circles}. The first is $S_\NC(m_1, m_2,
m_3)$. This is on the upper left in Figure
\ref{fig:three_circles}. Next, there is the set
\begin{multline*}
\{ (\cV, \pi) \mid \pi \in S_\NC(m_{i_1}, m_{i_2}) \times
\NC(m_{i_3}), |\cV| = |\pi| + 1\\ \textrm{\ and\ } \cV \vee
\{ [\![ m_{i_1}]\!] \cup [\![ m_{i_2}]\!], [\![
    m_{i_3}]\!]\} = 1_n \\
    \textrm{with } (i_1,i_2,i_3) \textrm{ a permutation of } (1,2,3)\},
\end{multline*}
which we denote by $\PSSSG$. These partitioned permutations are all represented by the graph on the upper
right in Figure \ref{fig:three_circles}. Next, we have
\begin{multline*}
\{ (\cV, \pi) \mid \pi \in \NC(m_1) \times \NC(m_2) \times
\NC(m_3), |\cV| = |\pi| + 2, \cV \vee \gamma = 1_n
\\ \textrm{\ and 2 blocks of\ }\cV \textrm{\ each contain
  two cycles of\ }\pi\},
\end{multline*}
which we denote by $\PSSG$. These partitioned permutations all have the graph on the
lower left of Figure \ref{fig:three_circles}. Finally we
have
\begin{multline*}
\{ (\cV, \pi) \mid \pi \in \NC(m_1) \times \NC(m_2) \times
\NC(m_3), |\cV| = |\pi| + 2, \cV \vee \gamma = 1_n
\\ \textrm{\ and 1 block of\ }\cV \textrm{\ contain three
  cycles of\ }\pi\},
\end{multline*}
which we denote by $\PSG$. These partitioned permutations all have the graph on the
lower right of Figure \ref{fig:three_circles}.
\end{notation}

\begin{remark}
The set $\cPS_\NC(m_1, m_2,m_3)$ is the disjoint union of $4$ sets, namely,
\begin{multline*}
\cPS_\NC(m_1, m_2,m_3)= S_\NC(m_1, m_2,
m_3) \cup \PSG \\
 \cup \PSSG \cup \PSSSG.
\end{multline*}
An example of each set can be see in Figure \ref{Figure:Partitioned permutations 3-annulus}.
\end{remark}

\begin{figure}
\begin{center}
\includegraphics[width=150pt]{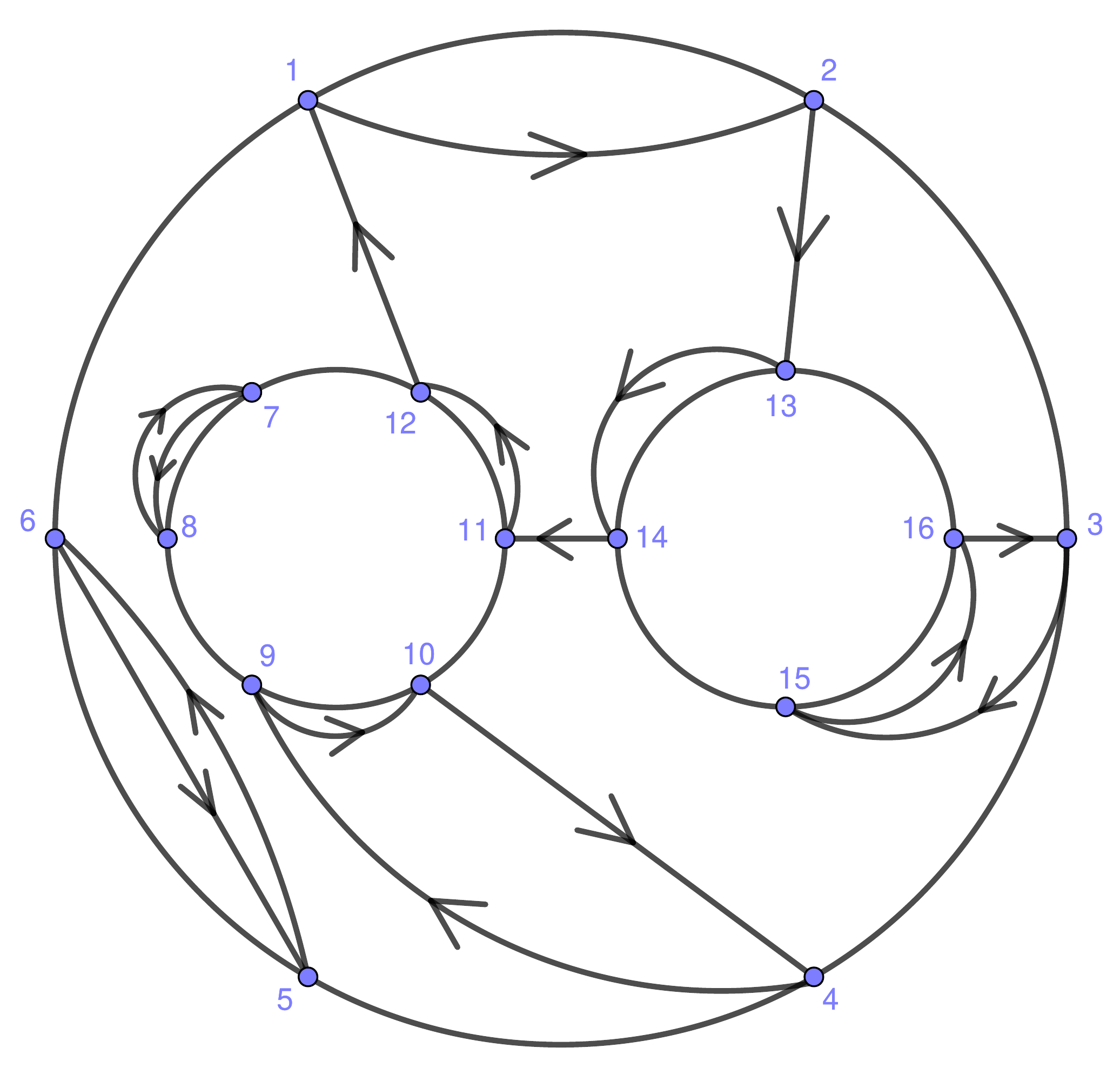} 
\includegraphics[width=150pt]{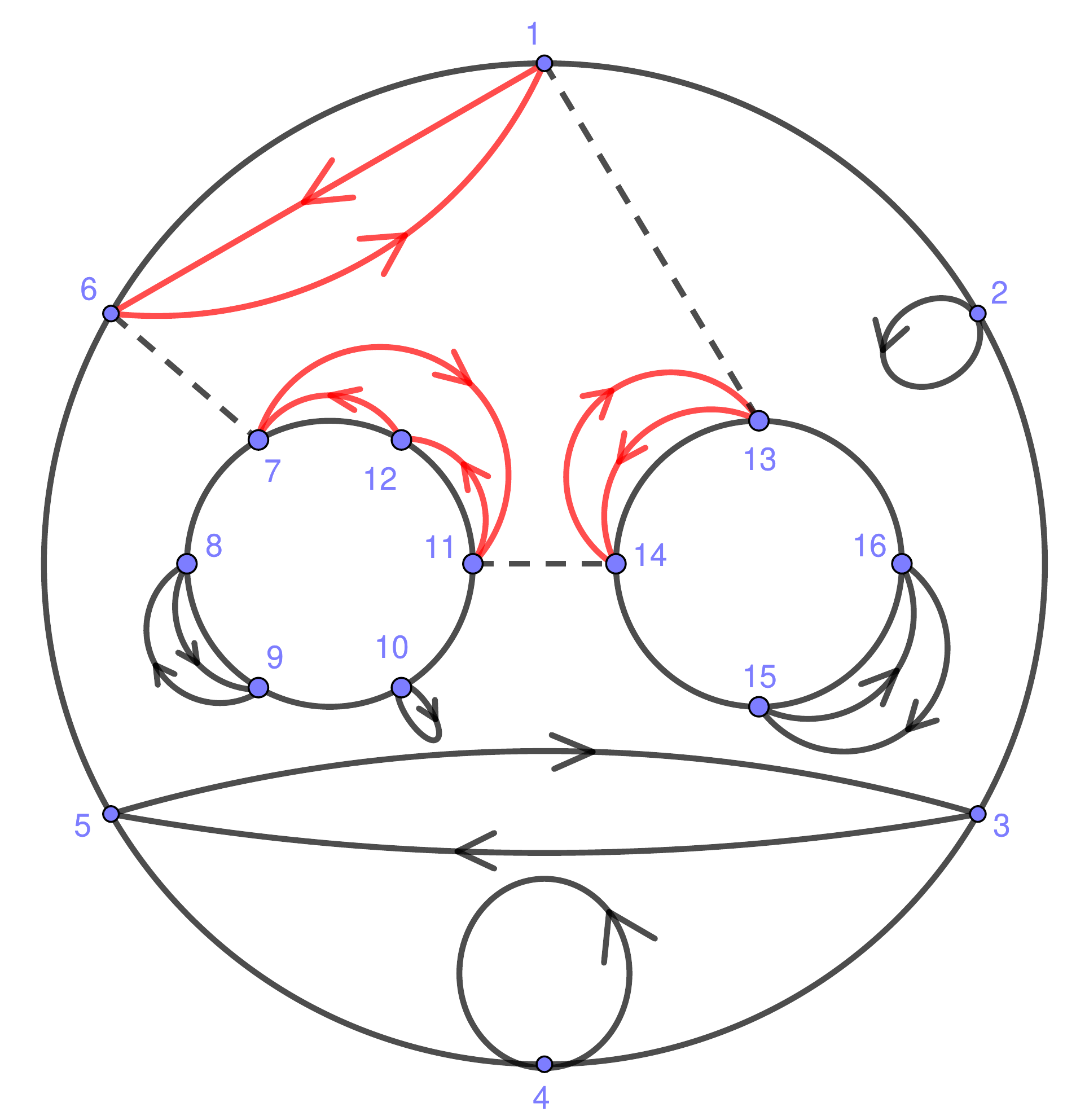}
\end{center}
\begin{center}
\includegraphics[width=150pt]{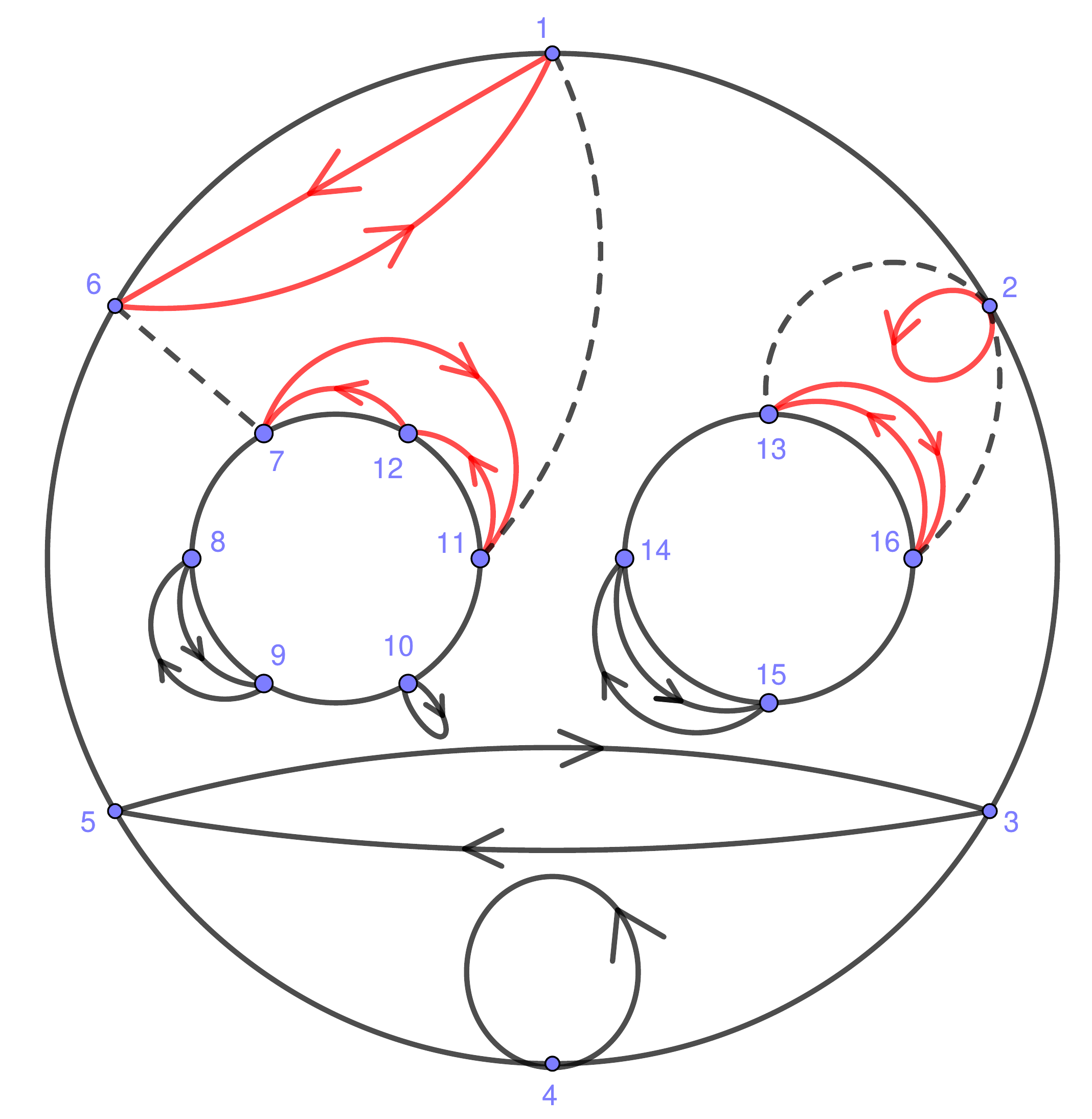} 
\includegraphics[width=150pt]{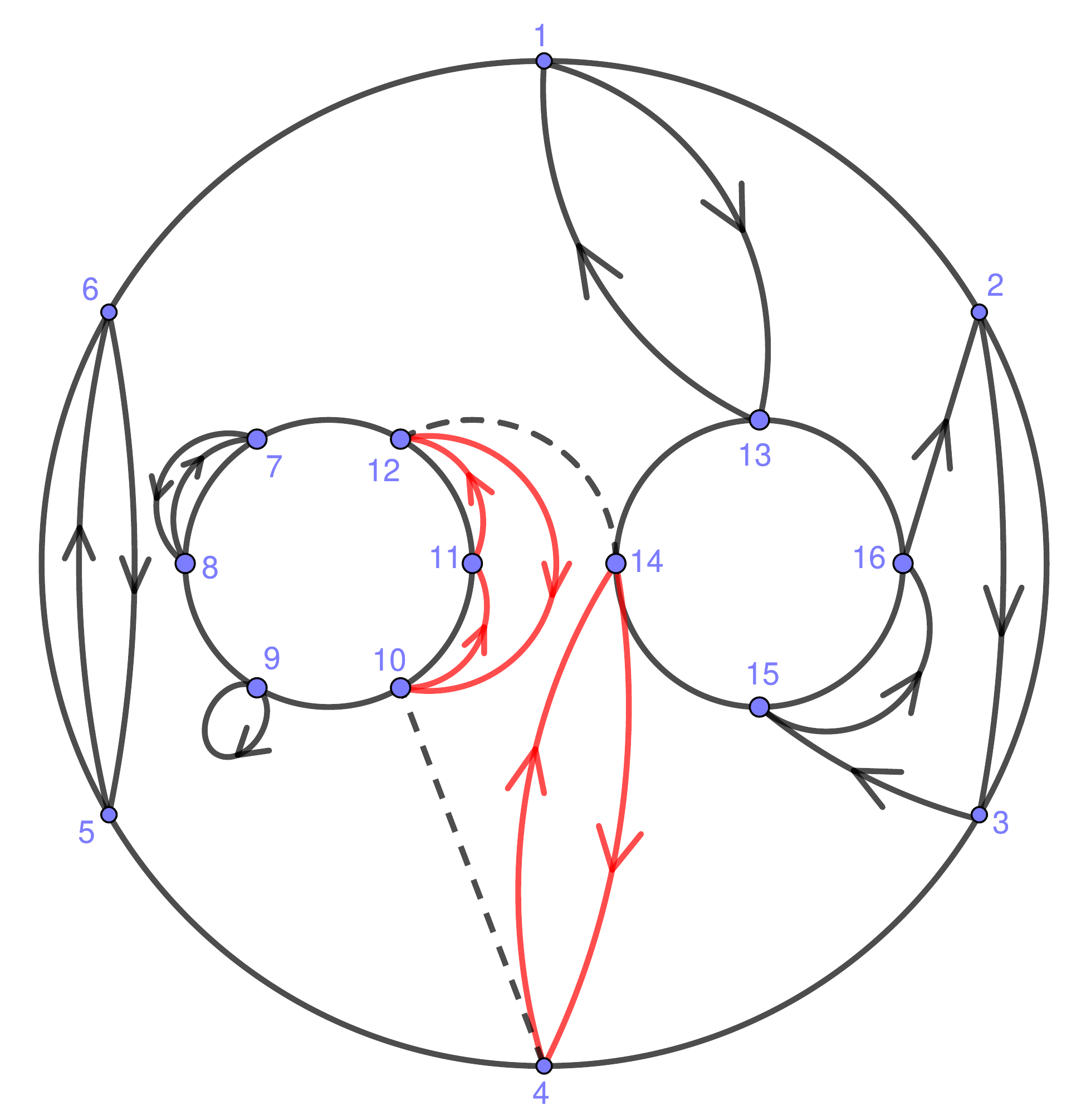}
\end{center}
\caption{\tiny\label{Figure:Partitioned permutations 3-annulus} We give examples of each of the four types making up the set $\cPS_\NC(m_1, m_2,m_3)$ with $m_1=m_2=6$ and $m_3=4$. Upper left: A non-crossing permutation $\pi\in S_\NC(m_1, m_2,
m_3)$ corresponding to \hfill\break
\centerline{$\pi=(1,2,13,14,11,12)(3,15,16)(4,9,10)(5,6)(7,8),$}
in this case $\pi$ connects all three circles. Upper right: A partitioned permutation $(\cV,\pi)\in \PSG$ corresponding to \begin{center}$\pi=(1,6)(2)(3,5)(4)(7,11,12)(8,9)(10)(13,14)(15,16),$\end{center} note that $\pi$ connects no circles but $\cV$ does by letting $\cV=\{\{1,6,7,11,\ab 12,13,14\},\{2\},\{3,5\},\{4\},\{8,9\},\{10\},\{15,16\}\}$. Each block of $\cV$ is a cycle of $\pi$ except the block $\{1,6,7,11,12,13,14\}$ which is the union of three cycles of $\pi$ shown in red. Bottom left: A partitioned permutation $(\cV,\pi)\in \PSSG$ corresponding to \hfill\break
\centerline{$\pi=(1,6)(2)(3,5)(4)(7,11,12)(10)(8,9)(13,16)(14,15),$}
$\pi$ connects no circles but $\cV$ does by letting $\cV=\{\{1,6,7,11,12\},\{2,13,16\},\ab\{3,5\},\{4\},\{10\},\{8,9\},\{14,15\},\}$. Each block of $\cV$ is a cycle of $\pi$ except the blocks $\{1,6,7,11,12\}$ and $\{2,13,16\}$ each one is the union of two cycles of $\pi$ shown in red. Bottom right: A partitioned permutation $(\cV,\pi)\in \PSSSG$ corresponding to \hfill\break
\centerline{$\pi=(1,13)(2,3,15,16)(5,6)(7,8)(9)(10,11,12).$}
Here $\pi$ connects only two circles, so we let $\cV=\{\{1,13\},\{2,3,15,16\},\{4,14\},\{5,6\},\ab\{7,8\},\ab\{9\},\{10,11,12,4,14\}\}$ so that $\cV$ connects all circles. Each block of $\cV$ is a cycle of $\pi$ except the block $\{10,11,12,4,14\}$ which is the union of two cycles of $\pi$ shown in red.}
\end{figure}

\section{Wigner ensemble and higher order\\ moments and free cumulants}
\label{sec:wigner-matrices}
A complex Wigner matrix is a self-adjoint random matrix with
independent identically distributed entries on and above the
main diagonal. In order to keep the formulas simple we shall
make some simplifying assumptions.

\begin{definition}\label{def:wigner_matrix}
By a complex Wigner matrix $X_N$ we mean a self-adjoint
$N\times N$ random matrix of the form
$X_N=\frac{1}{\sqrt{N}}(x_{i,j})$ such that,

\begin{itemize}

\item[$\diamond$]
the entries are complex random variables;

\item[$\diamond$]
 the matrix is self-adjoint: $x_{i,j}=\overline{x_{j,i}}$;
 
\item[$\diamond$] 
all entries on and above the diagonal are independent: 
$\{x_{i,j}\}_{i<j}\ab \cup \{x_{i,i}\}_i$ are independent;

\item[$\diamond$]
the entries above the diagonal, $\{x_{i,j}\}_{i<j}$, are
identically distributed;

\item[$\diamond$] 
the diagonal entries, $\{x_{i,i}\}_{i}$, are identically
distributed;

\item[$\diamond$] 
$\E(x_{i,j})=0$ for all $i,j$,

\item[$\diamond$] 
$\E(x_{i,j}^2)=0$ for all $i\neq j$,

\item[$\diamond$] 
$\E(|x_{i,j}|^2)=1$ for all $i,j$,

\item[$\diamond$]
$\C_3(x_{i,i},x_{i,i},x_{i,i})=0$ for all $i$,

\item[$\diamond$] 
$\E(|X_{i,j}|^k)<\infty$ for all $i,j,k.$    

\end{itemize}

A collection $X=(X_N)_N$ where $X_N$ is a Wigner matrix for
each $N$ is called a Wigner ensemble.
\end{definition}

Note that elements below the diagonal are all described
because the matrix is self-adjoint. Next is defining the
higher order moments and free cumulants associated to this Wigner
matrix.

\begin{definition}
We want to consider for each $r\geq 1$ and integers
$m_1,\dots, m_r\geq 1$ the quantities,
\[\E(\Tr(X_X^{m_1}) \cdots \Tr(X_N^{m_r}))\] where
$X=(X_N)_N$ is a Wigner ensemble and $Tr(\cdot)$ is the
trace. We don't expect these to have a large $N$ limit so we
consider the moment of order $r$ defined as,
\[
\alpha_{m_1,\dots,m_r}=\lim_{N\rightarrow\infty}\alpha_{m_1,\dots,m_r}^\psN=\lim_{N\rightarrow\infty}N^{r-2}\C_r(\Tr(X_X^{m_1}),\dots
,\Tr(X_N^{m_r}))
\]
provided the limit exists, where $\C_r$ is the $r^{th}$
classical cumulant as defined in \cite[Lect. 11,
  Appendix]{NS}.
\end{definition}

\begin{remark}
If we ask for the stronger conditions,
\begin{enumerate}
\item $\C_{2n+1}(x_{i,j}^{(\epsilon_1)},\dots, x_{i,j}^{(\epsilon_1)})=0$ for any $n$, $\epsilon_i\in \{-1,1\}$ and $i\neq j$
\item $\C_n(x_{i,i},\dots, x_{i,i})=0$ for all $n\neq 2$,
\end{enumerate}
where $x_{i,j}^{(1)}=x_{i,j}$ and $x_{i,j}^{(-1)}=x_{j,i}$, then $X_N$ turns out to be invariant under signed permutations, that is $Y_N=PX_NP^{-1}$ and $X_N$ have the same distribution where $P=UD$ is a signed permutation, i.e $D=\diag(\pm 1,\dots, \pm 1)$ and $U$ is a matrix permutation. The latest means that the condition over our Wigner matrix is a weaker condition, and it suffices to guarantee the existence of all moments up to order $3$. Our guess is that the moments of any order will exist only when $X_N$ is invariant under signed permutations.
\end{remark}

The case $r=1,2$ has been studied in \cite{MMPS}, namely,
\begin{equation}\label{FirstOrdercase}
\alpha_{m_1}=|NC_2(m_1)|
\end{equation}
\begin{equation}\label{SecondOrdercase}
\alpha_{m_1,m_2}=|NC_2(m_1,m_2)|+k_4|NC_2^{(2)}(m_1,m_2)|,
\end{equation}
where $k_4=k_4(x_{1,2},x_{1,2},x_{2,1},x_{2,1})$ is the fourth
classical cumulant of an off-diagonal element.
Equations (\ref{FirstOrdercase}) and (\ref{SecondOrdercase})
provides an expression for the first and second order
moments, however they also determine the first and second
order cumulants, in order to make sense of this we define
the first, second and third order cumulants below.

\begin{definition}
Given
\[
\{\alpha_m\}_{m=1}^{\infty}, \{ \alpha_{m_1,m_2} \}_{m_1,m_2 = 1}^{\infty},
\mbox{\  and\ } \{\alpha_{m_1,m_2,m_3}\}_{m_1,m_2,m_3 = 1}^\infty
\]
a sequence of first, second and third order moments,
we define the first $\{\K_{m}\}_m$, second
$\{\K_{m_1,m_2}\}_{m_1,m_2}$ and third
$(\K_{m_1,m_2,m_3})_{m_1,m_2,m_3}$ order cumulants as the
sequences given by the recursive formulas,
\begin{equation}\label{mc1}
\alpha_{m}=\sum_{\pi\in \NC(m)}\K_{\pi}
\end{equation}
\begin{equation}\label{mc2}
\alpha_{m_1,m_2}=\sum_{(\cU,\pi)\in \cPS_{\NC}(m_1,m_2)}\K_{(\cU,\pi)}
\end{equation}
\begin{equation}\label{mc3}
\alpha_{m_1,m_2,m_3}=\sum_{(\cU, \pi)\in \cPS_{\NC}(m_1,m_2,m_3)}\K_{(\cU,\pi)}
\end{equation}
with $\K_{\pi}$ and $\K_{(\cU,\pi)}$ defined as follows,
$$\K_{\pi}=\prod_{B\text{ cycles of }\pi}\K_{|B|}$$
$$\K_{(\cU,\pi)}=\prod_{\substack{D\text{ blocks of }\cU \\ B_1,\dots,B_l\text{ cycles of }\pi \\ B_i\subset D}}\K_{|B_1|,\dots,|B_l|}$$
\end{definition}

In general, it is possible to define the cumulants of any
higher order via the set of non-crossing partitioned
permutations on the $(m_1,m_2, \dots, \ab m_r)$-annulus, this
is done in \cite[Def. 7.4]{CMSS}, however in this work we will
not go further than $r = 3$. Equations (\ref{mc1}), (\ref{mc2}) and (\ref{mc3})
define the cumulant sequences uniquely, let us do some examples of this.

\begin{example}
$NC(1)$ has a unique element $\pi=(1)$, thus,
$$\alpha_1=\K_1$$
$NC(2)$ has two elements, namely, $\pi_1=(1,2)$ and $\pi_2=(1)(2)$, thus, $
\alpha_2=\K_2+\K_1\K_1=\K_2+\alpha_1^2$,
therefore,
$$\K_2=\alpha_2-\alpha_1^2.$$
$NC(3)$ has 5 elements, namely, $\pi_1=(1,2,3)$, $\pi_2=(1)(2,3)$, $\pi_3=(2)(1,3)$, $\pi_4=(3)(1,2)$ and $\pi_5=(1)(2)(3)$, thus, $\alpha_{3}=\K_1^3+3\K_2\K_1+\K_3$, therefore,
$$\K_3=\alpha_3-3\alpha_1(\alpha_2-\alpha_1^2)-\alpha_1^3=\alpha_3-3\alpha_1\alpha_2+2\alpha_1^3.$$
Let us do some examples of second order, the simplest case is $m_1=m_2=1$, in this case $\mathcal{PS}_{NC}(1,1)$ has the elements, $(\cV_1,\pi_1)=(\{1,2\},\ab(1,2))$ and $(\cV_2,\pi_2)=(\{1,2\},(1)(2))$, thus, $\alpha_{1,1}=\K_{2}+\K_{1,1}$, therefore,
$$\K_{1,1}=\alpha_{1,1}-\K_2=\alpha_{1,1}-\alpha_2+\alpha_1^2.$$
The next case is $m_1=1$ and $m_2=2$. In this case $\mathcal{PS}_{NC}(1,2)$ has the elements,
\begin{eqnarray*}
(\cV_1,\pi_1) &=& (\{1,2,3\},(1,2,3)) \\
(\cV_2,\pi_2) &=&(\{1,2,3\},(1,3,2)) \\
(\cV_3,\pi_3) &=& (\{\{1,2\},\{3\}\},(1,2)(3)) \\
(\cV_4,\pi_4) &=& (\{\{1,3\},\{2\}\},(1,3)(2)) \\
(\cV_5,\pi_5) &=& (\{1,2,3\},(1)(2,3)) \\
(\cV_6,\pi_6) &=& (\{\{1,2\},\{3\}\},(1)(2)(3)) \\
(\cV_7,\pi_7) &=& (\{\{1,3\},\{2\}\},(1)(2)(3))
\end{eqnarray*}
thus,
$$\alpha_{1,2}=2\K_3+2\K_2\K_1+\K_{1,2}+2\K_1\K_{1,1},$$
therefore,
\begin{eqnarray*}
\K_{1,2}&=&\alpha_{1,2}-2(\alpha_3-3\alpha_1\alpha_2+2\alpha_1^3)-2\alpha_1(\alpha_2-\alpha_1^2)  \\
&& \mbox{} -
2\alpha_1(\alpha_{1,1}-\alpha_2+\alpha_1^2) \\
&=& \alpha_{1,2}-2\alpha_1\alpha_{1,1}-2\alpha_3+6\alpha_1\alpha_2-4\alpha_1^3.
\end{eqnarray*}
Finally let us do some examples of third order case, the easiest case is $m_1=m_2=m_3=1$, in this case $\mathcal{PS}_{NC}(1,1,1)$ has the six elements,
\begin{eqnarray*}
(\cV_1,\pi_1) &=& (\{1,2,3\},(1,2,3)) \\
(\cV_2,\pi_2) &=& (\{1,2,3\},(1,3,2)) \\
(\cV_3,\pi_3) &=& (\{1,2,3\},(1,2)(3)) \\
(\cV_4,\pi_4) &=& (\{1,2,3\},(1,3)(2)) \\
(\cV_5,\pi_5) &=& (\{1,2,3\},(2,3)(1)) \\
(\cV_6,\pi_6) &=& (\{1,2,3\},(1)(2)(3))
\end{eqnarray*}
thus,
$$\alpha_{1,1,1}=2\K_3+3\K_{1,2}+\K_{1,1,1}.$$
Therefore,
\begin{eqnarray*}
\K_{1,1,1}=\alpha_{1,1,1}-3\alpha_{1,2}+6\alpha_1\alpha_{1,1}+4\alpha_3-12\alpha_1\alpha_2+8\alpha_1^3.
\end{eqnarray*}
Similarly one can check,
\begin{eqnarray*}
\alpha_{1,1,2} &=& 2 \K_2^2 + 4 \K_1 \K_3 + 6 \K_4 + 
 4 \K_2 \K_{1, 1} + 2 \K_{1, 1}^2 + 6 \K_1 \K_{1, 2} + 
 4 \K_{1, 3} + \K_{2, 2} \\
 &+& 2 \K_1 \K_{1, 1, 1} + 
 \K_{1, 1, 2}
\end{eqnarray*}
\begin{eqnarray*}
\alpha_{1,2,2} &=& 8 \K_1 \K_2^2 + 8 \K_1^2 \K_3 + 
 16 \K_2 \K_3 + 24 \K_1 \K_4 + 16 \K_5 + 
 14 \K_1 \K_2 \K_{1, 1}\\
 &  + & 8 \K_3 \K_{1, 1} 
 + 6 \K_1 \K_{1, 1}^2 + 12 \K_1^2 \K_{1, 2} + 
 8 \K_2 \K_{1, 2} + 4 \K_{1, 1} \K_{1, 2} + 
 16 \K_1 \K_{1, 3} \\
 &+ & 4 \K_{1, 4} 
 + 4 \K_1 \K_{2, 2} + 
 4 \K_{2, 3} + 4 \K_1^2 \K_{1, 1, 1} + 
 4 \K_1 \K_{1, 1, 2} + \K_{1, 2, 2}
\end{eqnarray*}
\end{example}

The results studied
in \cite{MMPS} can be written in terms of cumulants, namely,
\begin{equation}\label{FirstOrdercaseCumulants}
\alpha_{m}=\sum_{\pi\in NC(m)}\K_{\pi} = |NC_2(n)|
\end{equation}
because $\K_2=1$ and $\K_n=0$ for $n\neq 2$.
\begin{align}\label{SecondOrdercaseCumulants}
\lefteqn{
\alpha_{m_1,m_2}  = \kern-1em
\sum_{(\cU,\pi)\in \cPS_{\NC}(m_1,m_2)} \kern-1em
\K_{(\cU,\pi)} }\\
&= 
|NC_2(m_1, m_2)| + \K_{2,2} \frac{m_1 m_2}{4} 
|NC_2(m_1)| |NC_2(m_2)|\notag
\end{align}
because $\K_2=1$, $\K_n=0$ for $n\neq 2$, $\K_{2,2}=2\C_4(x_{1,2},x_{1,2},x_{2,1},x_{2,1})$
and $\K_{p,q}=0$ for $(p,q)\neq (2,2)$. This determines
the first and second order cumulants. The main result
of the present work is the equivalent formula for the third
order case.

\section{Graph theory}
\label{Section:GraphTheory}


Inspired by the traffic techniques introduced by Male
\cite{C}, in this section we will define some specific graphs
familiar from traffic theory which will be used to
achieve our objective. In traffic theory one relates random matrices to graphs
so that the entries of the matrices can
be indexed by the edges of the graphs; this notion permits
us to translate properties of the cumulants of the entries to
topological properties of the graphs.

In the first part of this section we review some general
graph theory. We begin by recovering some basic facts from
\cite{NC}. By a \textit{graph} we mean an unoriented graph 
where loops and multiple edges are permitted. In detail,
by a graph, $G$, we  mean a pair $(V,E)$, where $V$ is the set of vertices and $E$ is the set of edges. For an unoriented graph an edge is a subset $\{a, b\}$ of $V$, where we allow $a = b$. For $\{a,b\}\in E$, we say that
$a$ and $b$ are \textit{adjacent} vertices of the graph. We shall say that
$G$ is \textit{planar} if it can be drawn on the $2$-sphere without the edges crossing. If $S$ denotes the
$2$-sphere then $S\setminus G$ is a disjoint union of connected
components called the \textit{faces} of the graph. We said
that a graph is \textit{connected} if for any two vertices there is a
path on the edges travelling from one to the other.

\begin{definition}\label{Definition:Circui_Cycle_OfAGraph}
A \textit{\cycle} in a graph $G=(V,E)$ is a connected sub-graph $(V^{\prime},E^{\prime})$ with each vertex of degree 2. $|V^\prime|$ is called the \textit{length} of the \cycle{}. Here by degree of a vertex we mean the number of adjacent edges to the vertex where by convention a loop on a vertex contributes two to its degree. In Figure \ref{Chapter1_Figure_Unoriented graph with 1 circuit} we can see an example of a graph with a \cycle.
\end{definition}

\begin{figure}[h]
\begin{center}
\includegraphics[width=170pt,height=100pt]{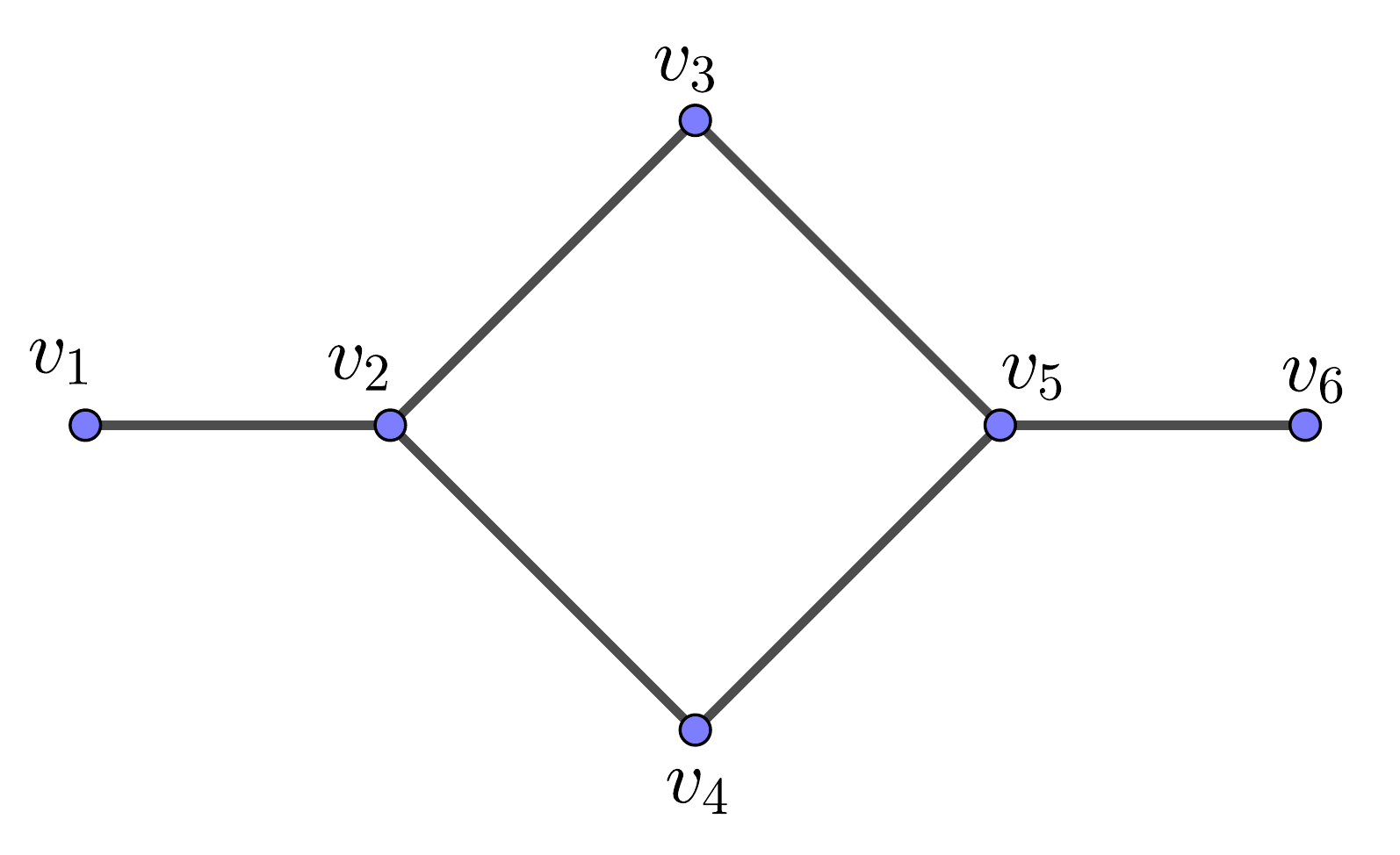}
\end{center}
\caption{\small\label{Chapter1_Figure_Unoriented graph with 1 circuit} A graph $(V,E)$ with vertex set $V
= \{v_1, v_2,\dots,\ab v_6\}$ and edge set
  $E=\{(v_1,v_2),(v_2,v_3),(v_3,v_5),(v_2,v_4),(v_4,\ab v_5),(v_5,v_6)\}$. This graph has a unique \cycles of length $4$: $(V^\prime,E^\prime)$ with $V^\prime=\{v_2,v_3,v_5,v_4\}$ and $E^\prime=\{(v_3,v_5),(v_5,v_4),(v_4,v_2),(v_2,v_3)\}$.}
\end{figure}

\begin{remark}\label{Theorem:UnicyclicCaracterization}
Suppose $G = (V_G, E_G)$ is a connected graph and $T =(V_T, E_T)$ is a subgraph which is a spanning tree, i.e. $T$ is a tree and  $V_G = V_T$. Let $m = |E_G| - |E_T|$. Let $H_1(G, \bZ)$ be the first homology group of $G$, then the rank of $H_1(G, \bZ)$ is $m$ and each edge in $E_G \setminus E_T$ is an edge of a circuit; together these circuits linearly generate $H_1(G, \bZ)$, see e.g. \cite{massey}. Moreover 
$
|V_G| - |E_G| = |V_T| - |E_T| - m = 1 - m. 
$

When $m=1$ (equivalently $|V_G|-|E_G|=0$) then we say $G$ is a \textit{unicircuit graph}. In this case $G$ has a unique \cycle.
\end{remark}

An example of a uni\cycles graph can be found in Figure \ref{Chapter1_Figure_Unoriented graph with 1 circuit}.

\begin{definition}
By an \textit{oriented graph} we mean a pair $(V,E)$ where
$V$ is the set of vertices and $E$ is a set of ordered pairs
of vertices, $E \subset \{(v_1,v_2):v_1,v_2\in V\}$. All our
graphs will be finite. Given $e=(v_1,v_2)\in E$ we let
$s(e)=v_1$ and $t(e)=v_2$ be the source and target of $e$
respectively and we say that $e$ is an outgoing edge of
$v_1$ and an incoming edge of $v_2$. In an oriented graph loops and multiple edges will be permitted.
\end{definition}

\begin{notation}\label{def:edge_equivalence}
Let $T = (V, E)$ be an oriented graph and let $e_1, e_2 \in E$ be two edges,
\begin{enumerate}
\item
We say $e_1 \sim e_2$ if $e_1$ and $e_2$ connect the same
pair of vertices, ignoring orientation.

\item
For $e_1 \sim e_2$ we have either $s(e_1)=s(e_2)$ and
$t(e_1)=t(e_2)$, or $s(e_1)=t(e_2)$ and $t(e_1)=s(e_2)$, in
the first case we say that they have the same orientation
and in the second case we say that they have the opposite
orientation.
\end{enumerate}
For each edge $e$ we let $\overline{e}$ be its equivalence
class and $\overline{E}$ the set of equivalence
classes. Observe that $\overline{T} = (V, \overline{E})$ is
itself a graph, called the \textit{elementarization} of
$T$. The elementarization of an oriented graph is a
unoriented graph with all edges simple, but with
the possibility of loops. Given an edge $\overline{e} \in
\overline{E}$ we say that the multiplicity of $e$ is $k$ if
the cardinality of the equivalence class $\overline{e}$ is
$k$ and we denote it as $\mult(\bar{e})$. If removing an
edge $\overline{e}$ from the graph $\overline{T}$ increases the number
of connected components of the graph we call $\overline{e}$ a \textit{cutting edge}.
\end{notation}

\begin{remark}\label{Remark:Loops have same and opposite orientation}
Let $T = (V, E)$ be an oriented graph and let $e_1, e_2 \in E$ be two edges such that both are loops and $e_1 \sim e_2$. Then by definition, $e_1$ and $e_2$ have both, the same and the opposite orientation. Loops are the only edges allowing this.
\end{remark}

\begin{definition}\label{def:quotient_graph}
Given $T=(V,E)$ an oriented graph and $\pi\in \cP(V)$, a
partition of the set of vertices, we define the \textit{quotient}
graph $T^{\pi}$ as the oriented graph obtained from $T$
after identifying vertices in the same block of $\pi$ and we write
$T^{\pi}=(V^{\pi},E^{\pi})$. In this convention the
blocks of $\pi$ are the vertices of $T^\pi$, and the edges of
$T^\pi$ are the edges of $T$. $T^{\pi}$ will be called the quotient of $T$ under $\pi$.
\end{definition}

\subsection{The graphs $T_{m_1,\dots,m_r}, T_{m_1,\dots,m_r}^{\pi}$ and $\overline{T^\pi_{m_1,\dots,m_r}}$}

This section is devoted to define a special graph which will
be the principal connection between graph theory and the
high order moments, for this section we set
$r\in\mathbb{N}$, $m_1,\dots,m_r\in\mathbb{N}$ and
$m=m_0+m_1+\cdots +m_r$ where by convention we let $m_0=0$. Most of the time we will work with
the case $r=3$, however some of the results are given for
the case $r=1,2$, this is the motivation of giving the more general
definition for any $r$.

\begin{definition}\label{def_T_k1k2kr}
Let 
\begin{multline*}
\gamma_{m_1,m_2,\dots,m_r}
=
(1,2,\ab\dots,m_1)
(m_1+1,m_1+2,\ab\dots,m_1+m_2) \\
\cdots (m_1+m_2+\cdots + m_{r-1}+1,\ab\dots,m)\in S_m.
\end{multline*}
\begin{enumerate}
\item 
For $1\leq j\leq r$ let $T_{m_j}=(V_i,E_j)$ be the oriented graph given by,
$$V_j=\{m_0 + m_1 + \cdots + m_{j-1} + 1,\ab \dots, m_0 + m_1 + \cdots + m_{j}\}$$
and edge set,
\begin{multline*}
E_j=\{(\gamma_{m_1,\dots,m_r}(i),i)
:\\ m_0 + m_1 + \cdots
+ m_{j-1} + 1\leq i \leq m_0 + m_1 + \cdots + m_{j}\}
\end{multline*}
We let $e_i \vcentcolon = (\gamma_{m_1,\dots,m_r}(i),i)$.
This graph is a regular $m_j$-gon, i.e. it has $m_j$ vertices
arranged in a circle and $m_j$ edges. See Figure
\ref{fig:t_6}. These oriented graphs are closed paths and we
will refer to them as \textit{basic cycles}.

\item 
For any $\{m_{j_1},\dots,m_{j_l}\}\subset\{m_1,\dots,m_r\}$
we let $T_{m_{j_1},\dots,m_{j_l}}$ be the graph
$\bigcup_{u=1}^l T_{m_{j_u}}$, in particular,
$T_{m_1,\dots,m_r}$ is the graph consisting of the $r$ basic
cycles, $T_{m_1}, \dots, T_{m_r}$. We write
$T_{m_{j_1},\dots,m_{j_l}}=(V_{j_1,\dots,j_l},E_{j_1,\dots,j_l})$
with
\[
V_{j_1,\dots,j_l}=\bigcup_{u=1}^ l V_{j_u} \mbox{\ and\ }
E_{j_1,\dots,j_l}\ab =\bigcup_{u=1}^ l E_{j_u}.
\]
When $\{m_{j_1},\dots,m_{j_l}\}=\{m_1,\dots,m_r\}$ we just
write $T_{m_1,\dots,m_r}=(V,E)$.  With the orientation from
$(1)$, we have that $T_{m_1, \dots, m_r}$ is an oriented
graph.

\end{enumerate}
\end{definition}

\begin{figure}\begin{center}
\quad\includegraphics[scale=0.75]{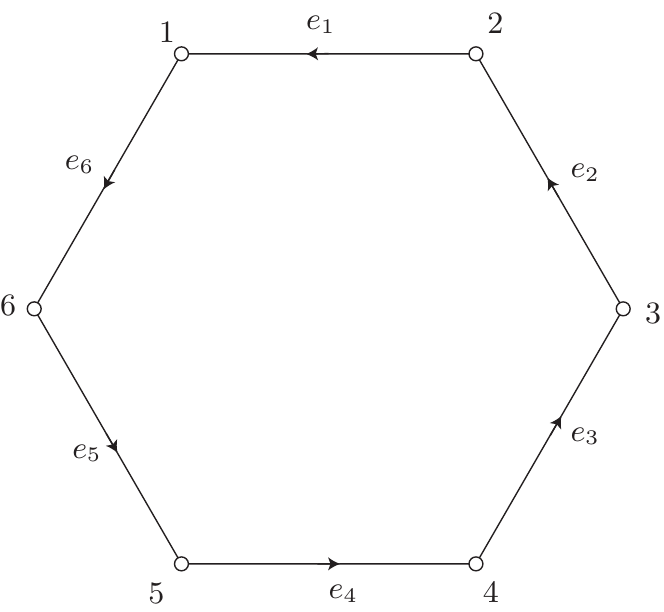}\hfill
\includegraphics[scale=0.75]{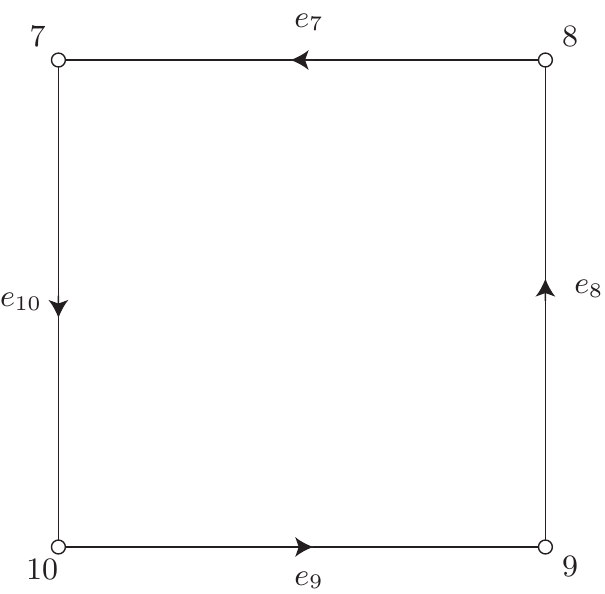}\quad\hbox{}
\end{center}
\caption{\small\label{fig:t_6} The graphs $T_{m_1}$, left, and
  $T_{m_2}$, right, with $m_1 = 6$ and $m_2 = 4$. Together
  they form the graph $T_{m_1,m_2}$.}
\end{figure}

\begin{definition}
Given $T_{m_1,\dots,m_r}=(V,E)$ as in Definition
\ref{def_T_k1k2kr} and $\pi \in \cP(m)$ we define the following graphs.
\begin{enumerate}

\item 
$T_{m_1,\dots,m_r}^{\pi}$ is the quotient graph of
$T_{m_1,\dots,m_r}$ under $\pi$. When convenient we shall write $(V^{\pi},E^\pi)$ for the vertices and edges of $T_{m_1,\dots,m_r}^{\pi}$. If $e_i \in E$ is
the edge $(\gamma_{m_1,\dots,m_r}(i), i)$, then we shall also write $e_i$ to denote the corresponding edge in
$E^\pi$, which it will be $([\gamma_{m_1,\dots,_r}(i)]_\pi, [i]_\pi)$ where $[i]_\pi$ denotes the block of $\pi$ containing $i$. For any
$\{m_{j_1},\dots,m_{j_l}\}\subset\{m_1,\dots,m_r\}$ we
denote  by $T_{m_{j_1},\dots,m_{j_l}}^{\pi}=(V_{j_1,\dots,j_l}^{\pi},E_{j_1,\dots,j_l}^{\pi})$, the graph of the restriction of $T_{m_1,\dots,m_r}^{\pi}$ to
$(V_{j_1,\dots,j_l},E_{j_1,\dots,j_l})$.

\item 
$\overline{T_{m_1,\dots,m_r}^{\pi}}$ is the elementarization
of the graph $T_{m_1,\dots,m_r}^{\pi}$. When convenient we shall write $(V^{\pi},\overline{E^\pi})$ for the vertices and edges of $\overline{T_{m_1,\dots,m_r}^{\pi}}$. For $e\in E^{\pi}$
we shall write $\overline{e}$ for the corresponding element of $\overline{E^\pi}$. For any
$\{m_{j_1},\dots,m_{j_l}\}\subset\{m_1,\dots,m_r\}$,
$\overline{T_{m_{j_1},\dots,m_{j_l}}^{\pi}}$ means the
elementarization of the graph
$T_{m_{j_1},\dots,m_{j_l}}^{\pi}$ and we write
$\overline{T_{m_{j_1},\dots,m_{j_l}}^{\pi}}=(V_{j_1,\dots,j_l}^{\pi},\overline{E_{j_1,\dots,j_l}^{\pi}})$.

\item 
We define the partition
$\overline{\pi}\in\cP(m)$ as the partition given by
$u\overset{\overline{\pi}}\sim v$ whenever
$\overline{e_u}=\overline{e_v}$.
\end{enumerate}
Some examples can be seen in Figure \ref{Example_TandT^pi}.
\end{definition}

\begin{remark}
For $\{m_{j_1},\dots,m_{j_l}\}\subset\{m_1,\dots,m_r\}$ the
graph $T_{\!m_{j_1},\dots,m_{j_l}}^{\pi}$ is itself a
quotient graph of $T_{m_{j_1},\dots,m_{j_l}}$
under $\pi^\prime$, with $\pi^\prime$ being the restriction
of $\pi$ to $V_{j_1,\dots,j_l}$.
\end{remark}

\begin{notation}
Let $\pi\in\cP(m)$. Let $(V,E),(V^{\pi},E^{\pi})$ and
$(V^{\pi},\overline{E^{\pi}})$ be the graphs $T_{m_1,\dots,m_r}$,
$T_{m_1,\dots,m_r}^{\pi}$ and
$\overline{T_{m_1,\dots,m_r}^{\pi}}$ respectively,
\begin{enumerate}

\item
For $\overline{e}\in \overline{E^{\pi}}$ and
$\{j_1,j_2,\dots,j_l\} \subset\{1,\dots,r\}$ we say that
$\overline{e}$ is a $(j_1,\dots,j_l)$-\textit{edge} if $e\in
E_{j_1,\dots,j_l}$ $\forall e\in\overline{e}$ and
$\overline{e}$ contains at least one edge from each
$E_{j_u}$ for $u=1,\dots,l$. When $l\geq2$ we say that
$\overline{e}$ is a \textit{connecting edge} and it connects
the basic cycles
$T_{m_{j_1}},T_{m_{j_2}},\dots,T_{m_{j_l}}$. When $l=1$ then
all edges in the equivalence class $\overline{e}$ belongs to
the same basic cycle, say $T_{m_{j}}$, in that case we say
that $\overline{e}$ is a \textit{non-connecting edge} and we
call $\overline{e}$ a $j$-\textit{edge}.

\item
Let $B=\{u_1,\dots,u_s\}$ be a block of $\overline{\pi}$ and
$\{j_1,j_2,\dots,j_l\} \subset\{1,\dots,r\}$. We say that
$B$ is a $(j_1,\dots,j_l)$-\textit{block} of $\overline{\pi}$ if
$\overline{e_{u_1}}$ is a $(j_1,\dots,j_l)$-edge. This is
well defined as $\overline{e_{u_i}}=\overline{e_{u_j}}$
$\forall 1\leq i,j\leq s$. When $l \geq 2$ we say that
$\overline{\pi}$ connects the basic cycles
$T_{m_{j_1}},\dots,T_{m_{j_l}}$. When $l=1$ then $B$
is a $j$-block for some $1\leq j\leq r$, which corresponds
to the case where $e_{u_1},\dots, e_{u_s}$ are all from the
same basic cycle $T_{m_j}$ , i.e $\overline{e_{u_1}}$ is a $j$-edge.

\end{enumerate}
\end{notation}

\begin{remark}\label{Remark:RelationPiBarAndEdgeSet}
Let $B= \{i_1,\dots,i_n\}$ be a subset of $[m]$. Observe that $B$ is a block of $\overline{\pi}$ if and only if $\{e_{i_i},\dots,e_{i_n}\}$ is an equivalence class of $\overline{E^{\pi}}$. In this sense, the edges of $\overline{T_{m_1,\dots,m_r}^{\pi}}$ can be regarded as blocks of $\overline{\pi}$. A $(j_1,\dots,j_l)$-block of $\overline{\pi}$ corresponds to a $(j_1,\dots,j_l)$-edge of $\overline{E^{\pi}}$. In the first case it means the block of $\overline{\pi}$ contains only elements from the cycles $[\![m_{j_1} ]\!],\dots, [\![m_{j_l} ]\!]$ of $\gamma_{m_1,\dots,m_r}$ while in the second case it means $\overline{e}$ contains edges only from $E_{j_1},\dots,E_{j_l}$.
\end{remark}

\begin{figure}[t]
\begin{minipage}[b]{\textwidth}\centering
\includegraphics[width=0.7\textwidth]{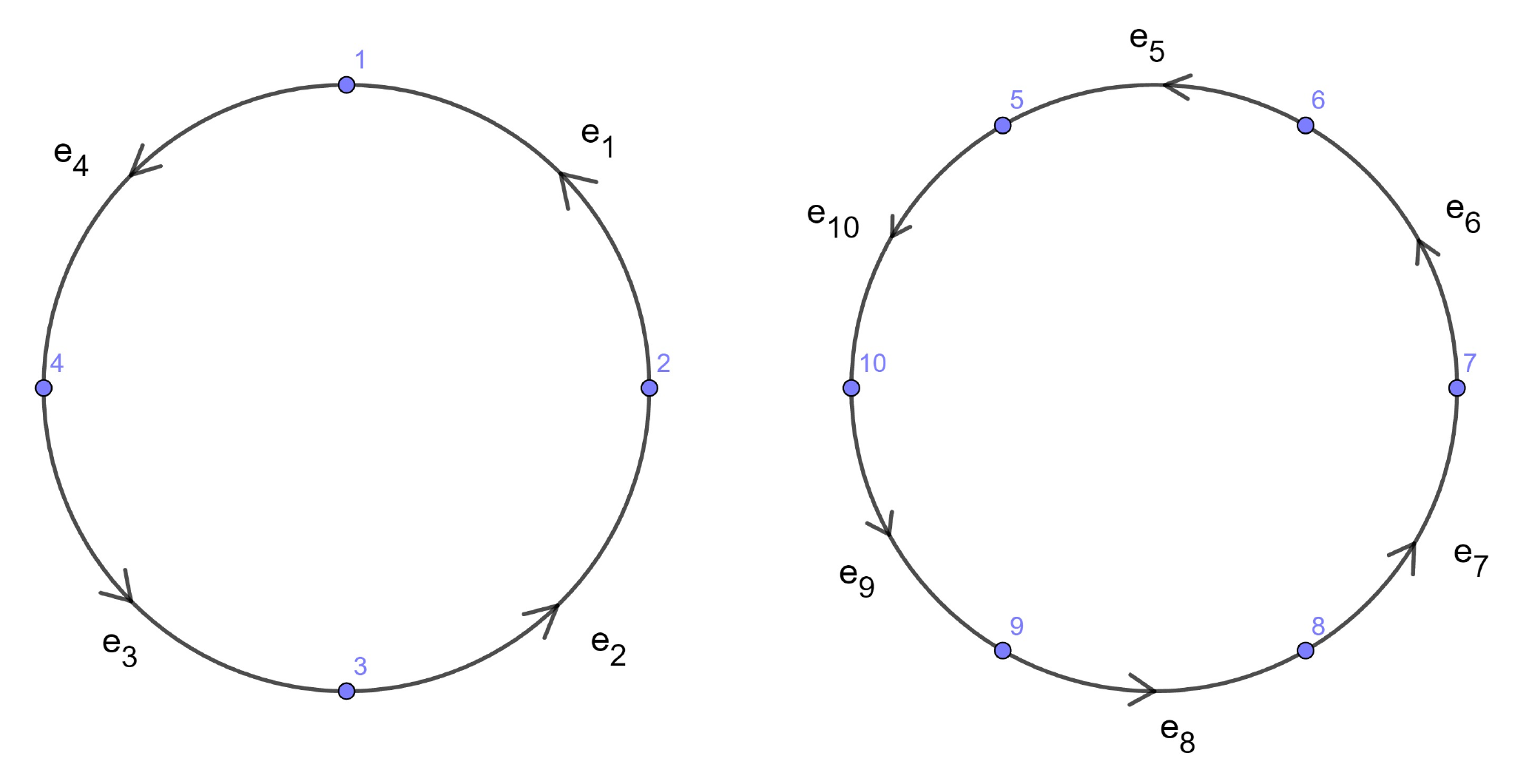}\\ \includegraphics[width=0.45\textwidth]{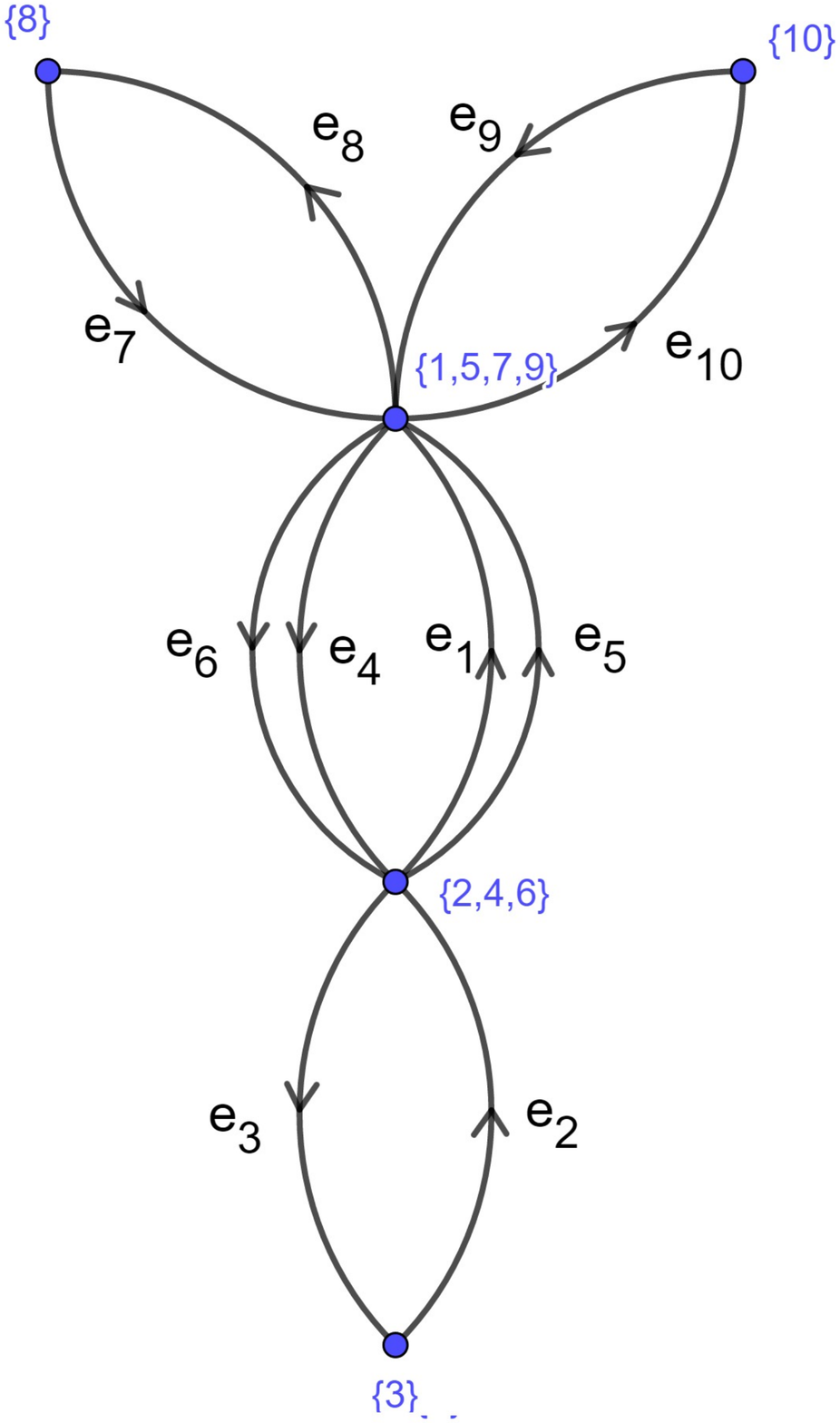}\hfill
\includegraphics[width=0.45\textwidth]{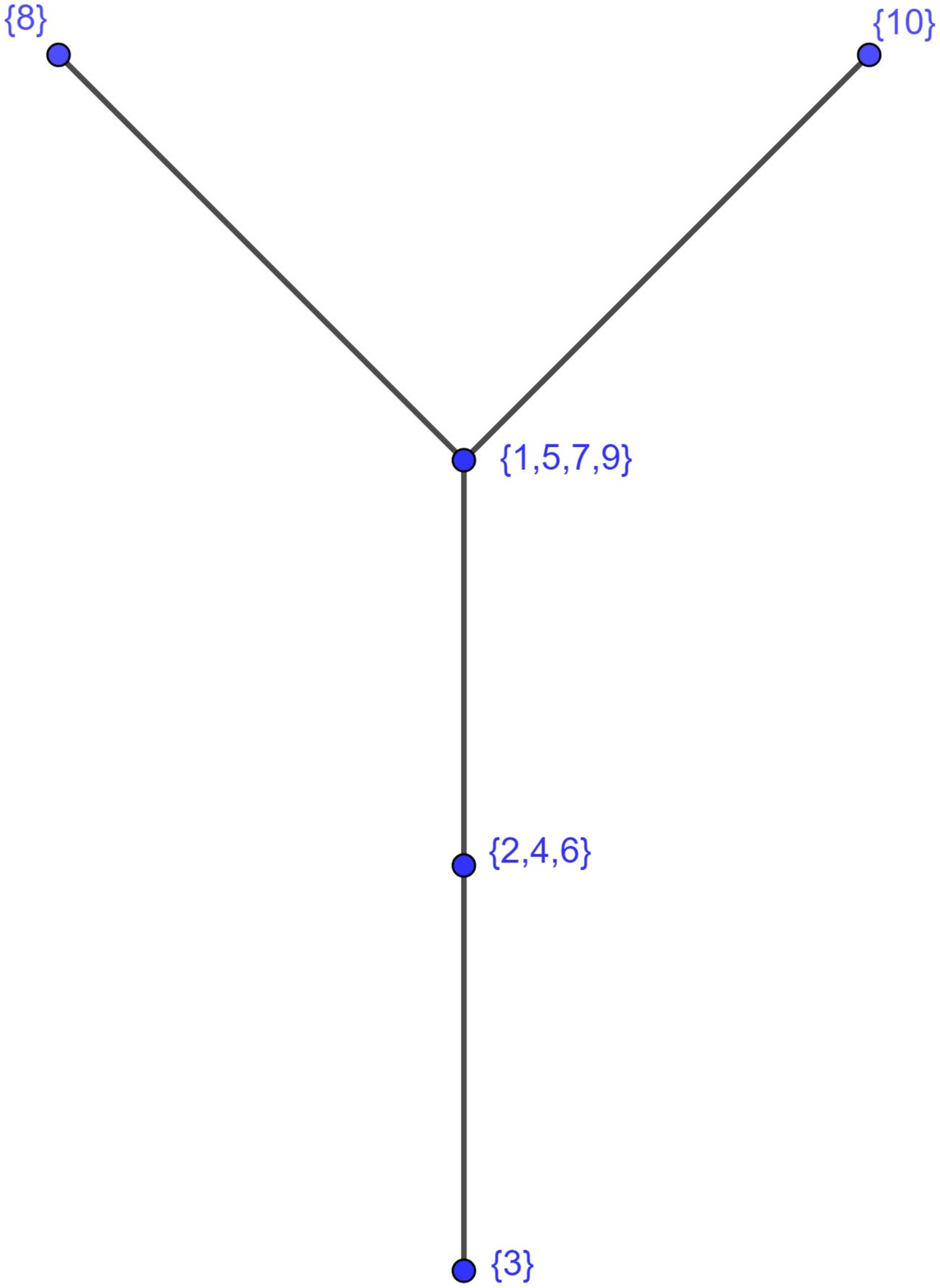}

\caption{\protect\raggedright\small%
  Above the graph $T_{4,6}$, below from left to
  right the graphs $T^{\pi}_{4,6}$ and
  $\overline{T^{\pi}_{4,6}}$ corresponding to the partitions
  $\pi=\{\{1,5,7,9\},\{2,4,6\},\{3\},\{8\},\{10\}\}$ and
  $\overline{\pi}=\{\{1,4,5,6\},\ab\{2,3\},\{7,8\},\{9,10\}\}$.}
\label{Example_TandT^pi}
\end{minipage}
\end{figure}

\begin{theorem}\label{Theorem:MultiplicityOfCuttingEdges}
Let $\pi\in\cP(m)$ and $(V,E),(V^{\pi},E^{\pi})$ and
$(V^{\pi},\overline{E^{\pi})}$ be the graphs $T_{m_1,\dots,m_r}$,
$T_{m_1,\dots,m_r}^{\pi}$ and $\overline{T_{m_1,\dots,m_r}^{\pi}}$
respectively, if $\overline{e}\in \overline{E^{\pi}}$ is a
cutting edge of $\overline{T_{m_1,\dots,m_r}^{\pi}}$ then
the number of edges in $\overline{e}$ coming from each basic
cycle is even, with an equal number of edges in each orientation. In particular
$\mult(\overline{e})$ is even.
\end{theorem}
\begin{proof}
Since the claim is about each basic cycle, it
is enough to look at each subgraph $T_{m_i}^{\pi}$.  So let
us assume $r=1$. We may also assume $m := m_1>2$ as the case
$m=2$ is clear and when $m=1$ it is not possible to have a
cutting edge.

Let $\overline{e}=(U,V)$ with $U,V\in \pi$ blocks of $\pi$.
Without loss of generality we may assume that
$e_1 =([2]_\pi, [1]_\pi)\in
\overline{e}$, that is $\overline{e_1}=\overline{e}$, with
$U=[2]_{\pi}$ and $V=[1]_{\pi}$.

Since $\ols{e}$ is a cutting edge, removing $\ols {e}$
disconnects $\ols{T^\pi_m}$ into two subgraphs, $S_1$, the
part containing the vertex $[1]_\pi$, and $S_2$, the part
containing the vertex $[2]_\pi$. We consider $\ols e$ to be
the ``bridge'' connecting $S_1$ to $S_2$. Now consider the
path $e_m, \dots
,e_1$ in $T^\pi_m$ which starts and
ends at $[1]_\pi$, and where $e_i$
is the oriented edge in $T^\pi_m$ which starts at
$[\gamma(i)]_\pi$ and ends at $[i]_\pi$.

This path starts and ends in $S_1$, so it must cross the
``bridge'' an even number of times; in particular the number
of times it passes from $S_1$ to $S_2$ must equal the number
of times it passes from $S_2$ to $S_1$. This is what was
claimed.
\end{proof}

\begin{remark}\label{Remark:MultiplicityOfCuttingEdges}
A cutting edge $\overline{e}$ of multiplicity $2$ must consist of two edges in opposite orientations from the same basic cycle, that is: $\overline{e}=\{e_u,e_v\}$ with $e_u,e_v \in E_j$ for some $1\leq j\leq r$. This is the same as saying that $\overline{e}$ is non-connecting $j$-edge. In terms of $\overline{\pi}$, it means the block $\{u,v\}$ is a non-through string. More generally, if $\mult(\overline{e})=2n$ then $\overline{e}$ can connect at most $n$ basic cycles, in that case, $\overline{e}$ consist of two edges with the opposite orientation from each basic cycle it connects.
\end{remark}

\begin{notation}
Let $\pi\in\cP(m)$ and
$(V^{\pi},E^{\pi})$ be the graph $T_{m_1,\dots,m_r}^{\pi}$, let
$U\in V^{\pi}$. For $1\leq j\leq r$ we define, 
\begin{enumerate}
\item $\rdg(U)_j=\sum_{e\in E^{\pi}_{j}} \mathbbm{1}_{t(e)=U}$
\item $\ldg(U)_j=\sum_{e\in E^{\pi}_{j}} \mathbbm{1}_{s(e)=U}$
\item $\deg(U)_j=\rdg(U)_j+\ldg(U)_j$
\end{enumerate}
We call $\deg(U)_j$ the $j^{th}$-degree of $U$, that is, the
number of incoming and outgoing edges of $U$ coming from $T_{m_j}$
where by definition a loop is counted twice. We let $\deg(U)$
to be the sum of $\deg(U)_j$ for all $j$, and we call that the
degree of $U$.
\end{notation}

\begin{remark}
The degree of an oriented and an unoriented graph coincides and it represents the number of adjacent edges (either incoming or outgoing for oriented graphs) to the vertex. In both cases (oriented and unoriented) a loop contributes two to the degree of a vertex.
\end{remark}

\begin{theorem}\label{Theorem:DegreeOfVertices}
Let $\pi\in\cP(m)$ and
$(V^{\pi},E^{\pi})$ be the graph $T_{m_1,\dots,m_r}^{\pi}$. For
$1\leq j\leq n$ and $U\in V^{\pi}$, 
$\rdg(U)_j=\ldg(U)_j$, and $\deg(U)_j$ and $\deg(U)$ are both even.
\end{theorem}
\begin{proof}
Let $1\leq j\leq r$, by definition of a basic cycle each
vertex of $T_{m_j}$ has degree $2$ consisting of an incoming and an
outgoing edge. After identifying vertices, the number of
incoming and outgoing edges remains the same.
\end{proof}

\begin{theorem}\label{Chapter1_Theorem_Configuration of edges of double unicyclic}
Let $\pi\in\cP(m)$. Let $(V^{\pi},E^{\pi})$ and $(V^{\pi},\overline{E^{\pi}})$
be the graphs $T_{m_1,\dots,m_r}^{\pi}$ and $\overline{T_{m_1,\dots,m_r}^{\pi}}$
respectively. Assume than $\overline{T_{m_1,\dots,m_r}^{\pi}}$ is a connected
graph with a unique \cycles(equivalently $|V^{\pi}|-|\overline{E^{\pi}}|=0$). Let $\overline{e}$ be an edge of $\overline{T_{m_1,\dots,m_r}^{\pi}}$ in within the \cycle. Let $1\leq j\leq r$ be such that $\overline{e}$ have an odd number of edges from $E_j$, then:
\begin{enumerate}[(i)]
\item any adjacent edge $\overline{e^\prime}$ to $\overline{e}$ in within the \cycles have an odd number of edges from $E_j$, and
\item all edges in the \cycles have an odd number of edges from $E_j$.
\end{enumerate}
\end{theorem}

\begin{proof}
It is enough to prove $(i)$ as $(ii)$ follows from $(i)$. If $\overline{e}$ is a loop then this is the \cycles of the graph and therefore there are no more edges in within the \cycles so this case is completed. So we may assume $\overline{e}$ is not a loop, let $\overline{e^\prime}$ be a distinct edge to $\overline{e}$ in the \cycle, let $U$ be the vertex in the \cycles which is incident to $\overline{e}$ and $\overline{e^\prime}$. Observe that apart from $\overline{e}$ and $\overline{e^\prime}$, $U$ is adjacent to only edges not in the \cycles and therefore cutting edges. These cutting edges consist of an even number of edges from $E_j$. Since $\deg(U)_j$ is even then the number of edges in $\overline{e}$ from $E_j$ and the number of edges in $\overline{e^\prime}$ from $E_j$ must add up to an even number. Hence the number of edges in $\overline{e^\prime}$ from $E_j$ must be odd which proves $(i)$.
\end{proof}

\begin{corollary}\label{Chapter1_Corollary_Configuration of double edges of double unicyclic}
Let $\pi\in\cP(m)$. Let $(V^{\pi},E^{\pi})$ and $(V^{\pi},\overline{E^{\pi}})$
be the graphs $T_{m_1,\dots,m_r}^{\pi}$ and $\overline{T_{m_1,\dots,m_r}^{\pi}}$
respectively. Assume than $\overline{T_{m_1,\dots,m_r}^{\pi}}$ is a connected
graph with a unique \cycles and such that all the edges in the \cycles have multiplicity $2$. Let $\overline{e}$ be an edge of $\overline{T_{m_1,\dots,m_r}^{\pi}}$ in the \cycles such that $\overline{e}$ is an $(i,j)$-edge, this is: it contains one edge from $E_i$ and one edge from $E_j$ and suppose these edges have opposite orientation. Then all edges in the \cycles are $(i,j)$-edges whose edges are in opposite orientation.
\end{corollary}

\begin{proof}
By Theorem \ref{Chapter1_Theorem_Configuration of edges of double unicyclic} all edges in the \cycles are $(i,j)$-edges, i.e. they contain one edge from $E_i$ and one edge from $E_j$ so it remains to verify that these edges have the opposite orientation. Let $U \in V^{\pi}$ be a vertex incident to $\overline{e}$. 

Let $A\subset \overline{E^{\pi}}$ be the set of edges incident to $U$. $A$ has exactly two edges which are in within the \cycle, one is $\overline{e}$ and let us denote to the other one by $\overline{e^\prime}$. All others edges in $A$ are cutting edges. We let $A^\prime= A\setminus \{\overline{e}, \overline{e}'\}$. Theorem \ref{Theorem:MultiplicityOfCuttingEdges} says that any edge in $A^\prime$ consists of the same number of edges in each orientation, so they contribute the same number of incoming and outgoing edges of $U$. By hypothesis $\overline{e}$ consist of one incoming and one outgoing edge of $U$. However, Theorem \ref{Theorem:DegreeOfVertices} says that the number of incoming and outgoing edges of $U$ must be the same, so this forces to $\overline{e^\prime}$ to consist of one incoming and one outgoing edge of $U$, thus the edges in the equivalence class $\overline{e^\prime}$ have opposite orientation.
\end{proof}

An example of a graph satisfying the hypotheses of Corollary \ref{Chapter1_Corollary_Configuration of double edges of double unicyclic} can be seen in Figure \ref{Chapter1_Figure_Double unicircuit graph}.

\begin{figure}[H]
\begin{center}
\includegraphics[width=200pt]{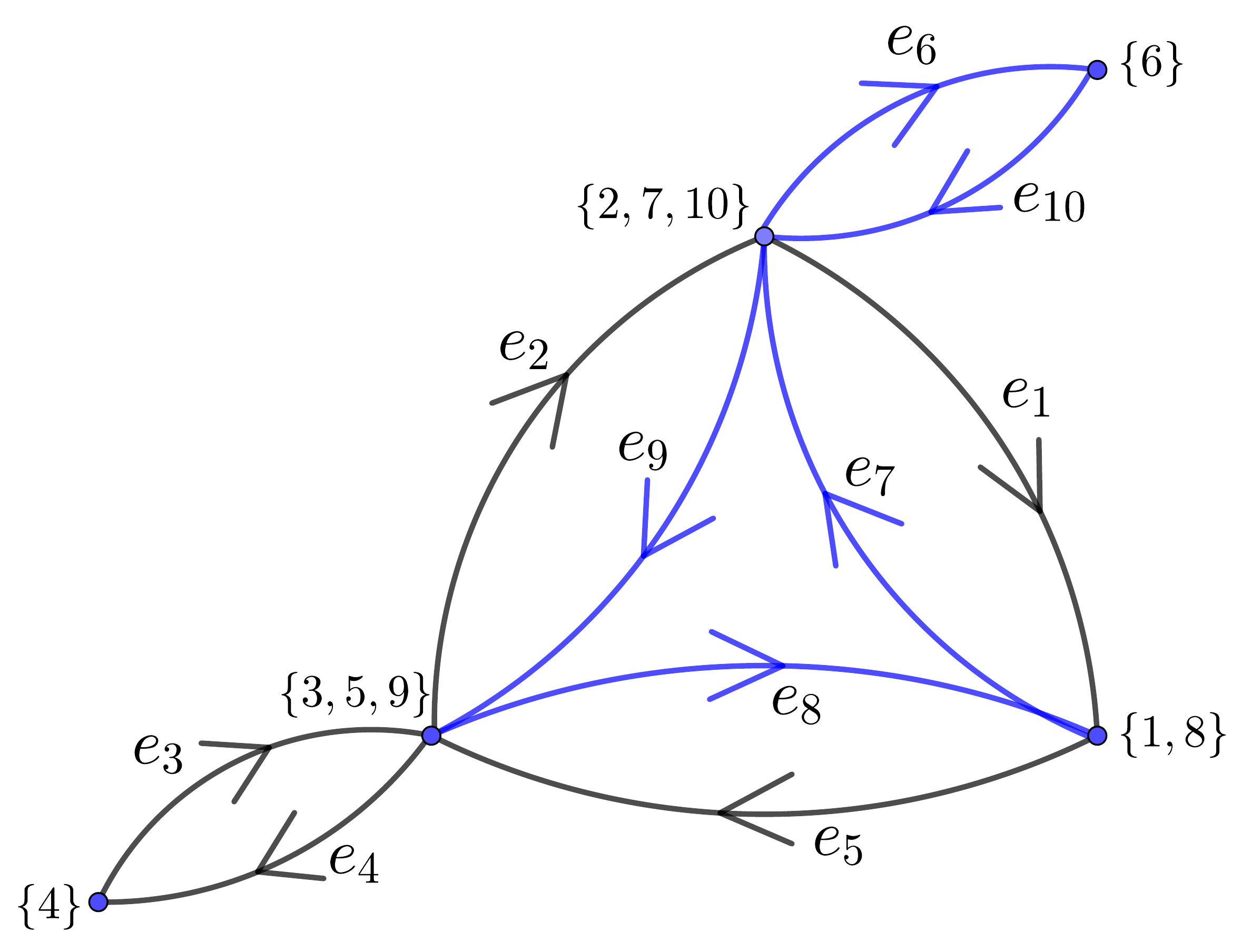}
\end{center}
\caption{\small\label{Chapter1_Figure_Double unicircuit graph} The oriented graph $T_{m_1,m_2}^{\pi}$ with $m_1=m_2=5$ corresponding to the partition $\pi=\{\{1,8\},\{2,7,10\},\{3,5,9\},\{4\},\{6\}\}.$ Any edge $\overline{e}$ in the unique \cycles of $\overline{T_{5,5}^{\pi}}$ is a $(1,2)$-edge and it consists of two edges in opposite orientation.}
\end{figure}

\section{Asymptotic expressions for third order moments}
\label{sec:asymptotic-expressions}
For this section we set $r=3$, $m_1,m_2,m_3\in\mathbb{N}$,
$m=m_1+m_2+m_3$ and $T_{m_1,m_2,m_3}=(V,E)$ and
$\gamma \vcentcolon = \gamma_{m_1,m_2,m_3}\in S_m$ be defined as in Section
\ref{Section:GraphTheory}. 

Note that for each $1\leq j\leq 3$,
$$\Tr(X_N^{m_j})=N^{-m_j/2}\sum_{i_1,\dots,i_{m_j}=1}^N x_{i_1,i_2}x_{i_2,i_3}\cdots x_{i_{m_j},i_1}$$
If we think of $i$ as a function from the vertices $V_j$ to $[N]$ then we can rewrite the equation above as follows:
\begin{eqnarray*}
\Tr(X_N^{m_j}) = N^{-m_j/2}\sum_{i^{(j)}:V_j\rightarrow [N]} \rp(i^{(j)})
\end{eqnarray*}
with,
\begin{eqnarray*}
\rp(i^{(j)}) & = & \prod_{v=m_0+\cdots + m_{j-1}+1}^{m_0+\cdots + m_j} x_{i^{(j)}(v),i^{(j)}(\gamma(v))} \nonumber \\
& = & x_{i^{(j)}(m_0+\dots +m_{j-1}+1),i^{(j)}(\gamma(m_0+\dots +m_{j-1}+1))}\cdots x_{i^{(j)}(m_0+\dots +m_j),i^{(j)}(\gamma(m_0+\dots +m_j))}
\end{eqnarray*}
where by convention we let $m_0=0$. The latest notation permit us to rewrite the term $\C_3(\Tr(X_N^{m_1}),\Tr(X_N^{m_2}),\Tr(X_N^{m_3}))$ in the following manner.

\begin{eqnarray}\label{Chapter2_Equation_First expression for almost alpha N}
&& \C_3(\Tr(X_N^{m_1}),\Tr(X_N^{m_2}),\Tr(X_N^{m_3})) \nonumber \\ & = & N^{-m/2}\C_3\left(\sum_{i^{(1)}:V_1\rightarrow [N]} \rp(i^{(1)}),\sum_{i^{(2)}:V_2\rightarrow [N]} \rp(i^{(2)}), \sum_{i^{(3)}:V_3\rightarrow [N]} \rp(i^{(3)})\right) \nonumber \\
& = & N^{-m/2}\sum_{i:V\rightarrow [N]}\C_3(\rp(i^{(1)}),\rp(i^{(2)}), \rp(i^{(3)})) \nonumber
\end{eqnarray}
where in the last equality we just used the linearity of the cumulants and $i:V\rightarrow [N]$ is a function from $V$ to $[N]$ obtained by considering all functions $i^{(j)}:V_j\rightarrow [N]$ together. In other words, $i^{(j)}$ is the function $i$ restricted to $V_j$. This permit us to rewrite the term $
\alpha_{m_1,m_2,m_3}^\psN$ as follows,

\begin{equation}\label{Equation_Number1}
\alpha_{m_1,m_2,m_3}^\psN  
= N^{-m/2+1}\sum_{i:V\rightarrow [N]}\C_3(\rp(i^{(1)}),\rp(i^{(2)}), \rp(i^{(3)}))
\end{equation}

Note that
$\C_3(\rp(i^{(1)}),\rp(i^{(2)}), \rp(i^{(3)}))$
depends only on $i$ and the graph $T_{m_1,m_2,m_3}$, so we
use the notation 
\[ \C_3(T_{m_1,m_2,m_3},i) \vcentcolon = \C_3(\rp(i^{(1)}),\rp(i^{(2)}), \rp(i^{(3)})).\]
We can write
(\ref{Equation_Number1}) as,
\begin{eqnarray}\label{Equation_Number2}
\alpha_{m_1,m_2,m_3}^\psN&=&\sum_{\pi\in \cP(m)}\sum_{\substack{i:V\rightarrow [N] \\ \ker(i)=\pi}}N^{-m/2+1}\C_3(T_{m_1,m_2,m_3},i)
\end{eqnarray}
where $\ker(i)$ is the partition of $\cP(m)$ given by $u\sim_{\ker(i)} v$ if and only if $i(u)=i(v)$.
\begin{theorem}\label{Theorem:OrderOfTerms}
Let $\pi\in\cP(m)$ then,
$$\sum_{\substack{i:V\rightarrow [N] \\ \ker(i)=\pi}}N^{-m/2+1}\C_3(T_{m_1,m_2,m_3},i)=O(N^{\#(\pi)-m/2+1})$$
\end{theorem}
\begin{proof}
For $1\leq j\leq 3$ and $i:V\rightarrow [N]$ such that $\ker(i)=\pi$ let us denote by $r_j$ to $\rp(i^{(j)})$, let $\sigma\in \cP(3)$ then,
\begin{eqnarray*}
|\E_\sigma(r_1,r_2,r_3)|&=&\prod_{\substack{B\in \sigma \\ B=\{j_1,\dots,j_s\}}}\Big|\E(\prod_{u=1}^s r_{j_u})\Big| \leq \prod_{\substack{B\in \sigma \\ B=\{j_1,\dots,j_s\}}}\E(|\prod_{u=1}^s r_{j_u}|) \\
&=& \prod_{\substack{B\in \sigma \\ B=\{j_1,\dots,j_s\}}} \prod_{\overline{e}\in \overline{E_{j_1,\dots, j_s}^{\pi}}} \E(\prod_{e\in\overline{e}} |x_{i(s(e)),i(t(e))}|)
\end{eqnarray*}
where in last equality we write the expectation of the product as the product of the expectations as for $\overline{e_1}\neq \overline{e_2}$ we have that $x_{i(s(e_1)),i(t(e_1))}$ and $x_{i(s(e_2)),i(t(e_2))}$ are independent. Moreover if $\overline{e_1}=\overline{e_2}$ then 
\[
|x_{i(s(e_1)),i(t(e_1))}|\ab = |x_{i(s(e_2)),i(t(e_2))}|
\] 
which means that $\E(\prod_{e\in\overline{e}} |x_{i(s(e)),i(t(e))}|)$ is either equal to $\E(|x_{1,2}|^{\mult(\overline{e})})$ or $\E(|x_{1,1}|^{\mult(\overline{e})})$ depending whether $i(s(e))=i(t(e))$ or not. In either case 
$\E(\prod_{e\in\overline{e}} |x_{i(s(e)),i(t(e))}|)\leq \beta_{\mult(\overline{e})}$ 
with $\beta_n=\max\{\ab \E(|x_{1,2}|^n), \ab \E(|x_{1,1}|^n)\}$. This means,
\begin{eqnarray*}
|\E_\sigma(r_1,r_2,r_3)| \leq \prod_{\substack{B\in \sigma \\ B=\{j_1,\dots,j_s\}}} \prod_{\overline{e}\in \overline{E_{j_1,\dots, j_s}^{\pi}}} \beta_{\mult(\overline{e})}
\end{eqnarray*}
The last upper bound depends only on $\sigma$ and $T_{m_1,m_2,m_3}^{\pi}$, so we may write, $|\E_\sigma(r_1,r_2,r_3)|\leq \beta(\sigma,T_{m_1,m_2,m_3}^{\pi})$. Next, by the moment-cumulant relation, \cite[Equation 1.4]{MS},
\begin{align*}\lefteqn{
|\C_3(T_{m_1,m_2,m_3},i)|}\\
&=
|\sum_{\sigma\in \cP(3)} \Mob(\sigma,1_3)\E_\sigma(r_1,r_2,r_3)| \leq
\sum_{\sigma\in \cP(3)} |\Mob(\sigma,1_3)|\beta(\sigma,T_{m_1,m_2,m_3}^{\pi})
\end{align*}
where the last expression depends only on $T_{m_1,m_2,m_3}^{\pi}$ so we have 
\[
|\C_3(T_{m_1,m_2,m_3},i)|\leq \beta(T_{m_1,m_2,m_3}^{\pi})\] 
and hence 
\begin{align*}\lefteqn{
\bigg|\sum_{\substack{i:V\rightarrow [N] \\ \ker(i)=\pi}}\C_3(T_{m_1,m_2,m_3},i)\bigg|
\leq 
\sum_{\substack{i:V\rightarrow [N] \\ \ker(i)=\pi}}
     \beta(T_{m_1,m_2,m_3}^{\pi}) }\\
&=
\beta(T_{m_1,m_2,m_3}^{\pi})N(N-1)\cdots (N-\#(\pi)+1)=O(N^{\#(\pi)}).
\end{align*}
\end{proof}

\begin{notation}
For a given $\pi\in\cP(m)$ note that 
\[
\sum_{\substack{i:V\rightarrow [N] \\ \ker(i)=\pi}}N^{-m/2+1}\C_3(T_{m_1,m_2,m_3},i)
\] 
depends only on $\pi$ and $m_1,m_2,m_3$, so, let us denote,
$$C_{\pi} \vcentcolon =\sum_{\substack{i:V\rightarrow [N] \\ \ker(i)=\pi}}N^{-m/2+1}\C_3(T_{m_1,m_2,m_3},i).$$
When necessary, we write $C_{\pi}(m_1,m_2,m_3)$ to specify the values of $m_i$.
\end{notation}

Theorem \ref{Theorem:OrderOfTerms} shows that  $C_{\pi} = O(N^{\#(\pi) -m/2+1})$; the next theorem gives an upper bound for that order when $C_{\pi}\neq 0$.

\begin{theorem}\label{Theorem:UpperBoundOfOrder}
Let $\pi\in \cP(m)$. If $C_{\pi}\neq 0$ then $\#(\pi)-m/2+1\leq 0$.
\end{theorem}

In order to proof Theorem \ref{Theorem:UpperBoundOfOrder} we introduce the following.
\begin{definition}
Let $\pi\in\cP(m)$ we denote by,
\begin{enumerate}
\item $q_1(\pi)=\#(\pi)-\#(\overline{\pi})$
\item $q_2(\pi)=\#(\overline{\pi})-m/2$
\item $q(\pi)=q_1(\pi)+q_2(\pi)$
\item $\overset{\rightarrow}q(\pi)=(q_1(\pi),q_2(\pi))$
\end{enumerate}
\end{definition}

\begin{theorem}\label{Theorem:Maybe}
Let $\pi\in \cP(m)$ and let $\gamma = \gamma_{m_1, m_2, m_3}$. If $C_{\pi}\neq 0$ then all the following are satisfied.
\begin{enumerate}
\item $\overline{\pi}\vee\gamma = 1_m.$
\item $\overline{T_{m_1,m_2,m_3}^{\pi}}$ is a connected graph.
\item No block of $\overline{\pi}$ is of size $1$.
\item Every block of $\overline{\pi}$ has at least $2$ elements, and  $q_2(\pi)\leq 0$ with equality if and only if all blocks of $\overline{\pi}$ have size $2$.
\item If $\{u,v\}$ is a block of $\overline{\pi}$ then $e_u$ and $e_v$ are edges of $T_{m_1,m_2,m_3}^{
\pi}$ connecting the same pair of vertices and with the opposite orientation,
\item For any $i:V\rightarrow [N]$ such that $\ker(i)=\pi$ the following formula holds, 
\begin{equation}\label{Equation_Number6}\\
\C_3(T_{m_1,m_2,m_3},i)=\sum_{\substack{\sigma\in\cP(m) \\ \sigma\vee\gamma=1_m \\ \sigma\leq \overline{\pi}}} \C_{\sigma}(x_{i_1,i_{\gamma(1)}},\dots, x_{i_m,i_{\gamma(m)}}).
\end{equation}
where $i_u$ denotes $i(u)$.
\end{enumerate}
\end{theorem}

\begin{proof}
First we show that for $\sigma \in \cP(m)$ we have that $\C_{\sigma}(x_{i_1,i_{\gamma(1)}},\dots,\ab  x_{i_m,i_{\gamma(m)}})\ab = 0$ unless $\sigma \leq \overline{\pi}$.  Let $\sigma\in\cP(m)$ and suppose that $u\sim_\sigma v$ for $u$ and $v \in[m]$. If neither $x_{u,\gamma(u)}=x_{v,\gamma(v)}$ nor $x_{u,\gamma(u)}=x_{\gamma(v),v}$ (equivalently $\overline{e_u}\neq \overline{e_v}$), then by the independence of the entries $\C_{\sigma}(x_{1,\gamma(1)},\dots,\ab x_{m,\gamma(m)})=0$. Therefore $\C_{\sigma}(x_{i_1,i_{\gamma(1)}},\dots,\ab  x_{i_m,i_{\gamma(m)}})\ab = 0$ unless $e_u=e_v$ for any $u$ and $v$ that are in the same block of $\sigma$, the latest is equivalent to $\sigma \leq \overline{\pi}$. We prove $(vi)$ first. By the formula of cumulants with product as entries, \cite[Theorem 11.30]{NS},
\begin{eqnarray*} 
\C_3(T_{m_1,m_2,m_3},i)=\sum_{\substack{\sigma\in\cP(m) \\ 
                               \sigma\vee\gamma = 1_m}} 
\C_{\sigma}(x_{i_1,i_{\gamma(1)}},\dots, x_{i_m,i_{\gamma(m)}}).
\end{eqnarray*}
Combining this with the requirement that $\sigma \leq \overline{\pi}$ proves $(vi)$. To prove $(i)$ note that if $\overline{\pi}\vee\gamma<1_m$ then there is no $\sigma$ satisfying both conditions $\sigma\leq \overline{\pi}$ and $\sigma\vee\gamma = 1_m$ because in that case we get $1_m= \sigma\vee\gamma\leq \overline{\pi}\vee\gamma<1_m$. Therefore, if $\overline{\pi}\vee\gamma<1_m$ by Equation (\ref{Equation_Number6}) we get $\C_3(T_{m_1,m_2,m_3},i)=0$, and then $C_{\pi}=0$. This proves $(i)$.

If $\overline{\pi}$ has a block of size $1$ then any $\sigma\leq \overline{\pi}$ would have a block of size $1$ which implies that $\C_{\sigma}(x_{i_1,i_{\gamma(1)}},\dots, x_{i_m,i_{\gamma(m)}})$ has a factor of the form $\C_1(x_{i_u,i_{\gamma(u)}})$ which equals $0$,  and so also $\C_3(T_{m_1,m_2,m_3},i)=0$. Thus $C_{\pi}=\ds\sum_{\substack{i:V\rightarrow [N] \\ \ker(i)=\pi}}N^{-m/2+1}\C_3(T_{m_1,m_2,m_3},i) = 0$, proving $(iii)$.

Note that $(iv)$ follows directly from $(iii)$. 

To prove $(v)$, let $\{u,v\}$ be a block of $\overline{\pi}$. Then by the definition of $\overline{\pi}$, $e_u$ and $e_v$ connect the same pair of vertices. If $e_u$ and $e_v$ are loops we are done, so we may assume they are not loops. Let $\sigma\leq \overline{\pi}$, then either $\{u,v\}$ is a block of $\sigma$ or $\sigma$ has blocks $\{u\}$ and $\{v\}$, in the second case $\C_{\sigma}(x_{i_1,i_{\gamma(1)}},\dots, x_{i_m,i_{\gamma(m)}})=0$ as the singleton $\{u\}$ produces a factor of the form $\C_1(x_{i_u,i_{\gamma(u)}})=0$, in the first case $\C_{\sigma}(x_{i_1,i_{\gamma(1)}},\dots, x_{i_m,i_{\gamma(m)}})$ has a factor of the form $\C_2(x_{i_u,i_{\gamma(u)}},x_{i_v,i_{\gamma(v)}})$. If the edges $e_u$ and $e_v$ have the same orientation then this factor is of the form $\C_2(x_{1,2},x_{1,2})=0$, so, if the edges have the same orientation Equation (\ref{Equation_Number6}) says that $\C_3(T_{m_1,m_2,m_3},i)=0$ and thus so is $C_{\pi}=0$. Thus the edges $e_u$ and $e_v$ must have opposite orientations.

Finally $(i)$ implies that the graph $T_{m_1,m_2,m_3}^{\pi}$ is connected as all basic cycles are connected by pairing some of their edges, therefore $\overline{T_{m_1,m_2,m_3}^{\pi}}$ must be connected proving $(ii)$. 
\end{proof}

Now we are able to provide a proof to Theorem \ref{Theorem:UpperBoundOfOrder}.

\begin{proof}[Proof of Theorem \ref{Theorem:UpperBoundOfOrder}]
Theorem \ref{Theorem:Maybe} $(ii)$ says that $\overline{T^{\pi}_{m_1,m_2,m_3}}=(V^{\pi},\ab \overline{E^{\pi}})$ is connected.   Then $q_1(\pi)=|V^{\pi}|-|\overline{E^{\pi}}| \leq 1$ with equality if and only if $\overline{T^{\pi}_{m_1,m_2,m_3}}$ is a tree. 

Theorem \ref{Theorem:Maybe} $(iv)$ says that $q_2(\pi)\leq 0$ with equality if and only if all blocks of $\overline{\pi}$ are of size 2. These inequalities force $q(\pi)\leq 1$ so it remains to prove that $q(\pi)=1,1/2,0,-1/2$ are not possible.

If $q(\pi)=1$, then $\overset{\rightarrow}q(\pi)=(1,0)$, which means $\overline{T^{\pi}_{m_1,m_2,m_3}}$ is a tree and all blocks of $\overline{\pi}$ are of size 2, equivalently all edges of $\overline{T^{\pi}_{m_1,m_2,m_3}}$ have multiplicity $2$. By Remark \ref{Remark:MultiplicityOfCuttingEdges}, any edge of $\overline{T^{\pi}_{m_1,m_2,m_3}}$ is non-connecting meaning that the condition $\overline{\pi}\vee \gamma_{m_1,m_2,m_3}=1_m$ cannot be satisfied.

If $q(\pi)=1/2$, then $\overset{\rightarrow}q(\pi)=(1,-1/2)$. This is not possible as $q_1(\pi)=1$ means $\overline{T^{\pi}_{m_1,m_2,m_3}}$ is a tree and hence all its edges are cutting which means they have even multiplicity (by Theorem \ref{Theorem:MultiplicityOfCuttingEdges}), however $q_2(\pi)=-1/2$ only happens if there is an edge of multiplicity 3.

If $q(\pi)=0$, then either $\overset{\rightarrow}q(\pi)=(1,-1)$ or $\overset{\rightarrow}q(\pi)=(0,0)$. In the first case, $\overline{T^{\pi}_{m_1,m_2,m_3}}$ is a tree, hence all its edges have even multiplicity, and  therefore $q_2(\pi)=-1$ implies that all edges of $\overline{T^{\pi}_{m_1,m_2,m_3}}$ have multiplicity $2$ except for one of multiplicity $4$. Moreover any edge of multiplicity $2$ is non-connecting; and the edge of multiplicity $4$ might be connecting, but it can connect at most two basic cycles (Remark \ref{Remark:MultiplicityOfCuttingEdges}), meaning $\overline{\pi}\vee \gamma_{m_1,m_2,m_3}=1_m$ again cannot be satisfied.

In the second case $q_1(\pi)=|V^{\pi}|-|\overline{E^{\pi}}|= 0$, so by Remark \ref{Theorem:UnicyclicCaracterization} $\overline{T^{\pi}_{m_1,m_2,m_3}}$ is a graph with a single \cycle, and as $q_2(\pi) = 0$, all edges are of multiplicity $2$, any edge not in the \cycles of the graph is cutting,  and hence non-connecting. So the only connecting edges can be the ones in the \cycle. Assume one of them is a $(1,2)$-edge, Corollary \ref{Chapter1_Corollary_Configuration of double edges of double unicyclic} says that all edges in the \cycles are are $(1,2)$-edges meaning $\overline{\pi}\vee \gamma_{m_1,m_2,m_3}=1_m$ again cannot be satisfied. 

The last case is $q(\pi)=-1/2$, this can happen only if $\overset{\rightarrow}q(\pi)=(0,-1/2)$ as the case $\overset{\rightarrow}q(\pi)=(1,-3/2)$ cannot happen by the same argument about the multiplicity of the edges. In this case $\overline{T_{m_1,m_2,m_3}^{\pi}}$ is  a uni\cycles graph with all edges of multiplicity $2$ except by one with multiplicity $3$. Let $\overline{e}=(U,V)$ be the edge of multiplicity $3$, If $\overline{e}$ is not a loop then $U$ is adjacent to only edges of multiplicity $2$ besides $\overline{e}$, thus $deg(U)$ is odd which is a contradiction to Theorem \ref{Theorem:DegreeOfVertices}. Hence, it must be $\overline{e}$ is a loop which correspond to the unique \cycle. Any other edge of $\overline{T_{m_1,m_2,m_3}^{\pi}}$ is cutting and hence the corresponding block of $\overline{\pi}$ is a non-through string. Since we require $\gamma \vee \overline{\pi} = 1_m$ then the edge of multiplicity $3$ must consist of one edge from each basic cycle, that is, the corresponding block of $\overline{\pi}$ contains one element from each cycle of $\gamma$. Let $i:V\rightarrow [N]$ be such that $\ker(i)=\pi$. Equation (\ref{Equation_Number6}) says:
\begin{equation*}
\C_3(T_{m_1,m_2,m_3},i) = \sum_{\substack{\tau \in \cP(m) \\ \tau \vee \gamma = 1_m  \\ \tau \leq \overline{\pi}}}\C_{\tau}(x_{i(1),i(\gamma(1))},\dots,x_{i(m),i(\gamma(m))}).
\end{equation*}
Let $\tau \leq \overline{\pi}$, if a block of $\tau$ has size $1$ then $\C_{\tau}(x_{i(1),i(\gamma(1))},\dots,x_{i(m),i(\gamma(m))})=0$, so the block with size $3$ of $\overline{\pi}$ must be a block of $\tau$ which corresponds to a factor $\C_3(x_{1,1},x_{1,1},x_{1,1})=0$ because this block correspond to the loop of the graph. Therefore $\C_3(T_{m_1,m_2,m_3},i)$ is zero and so is $C_{\pi}$ which is not possible. This concludes the proof.

\end{proof}

If we combine Equation (\ref{Equation_Number2}) with Theorems \ref{Theorem:OrderOfTerms} and \ref{Theorem:UpperBoundOfOrder} we obtain the following corollary.
 
\begin{corollary}\label{Corollary:Asymptotic_Expression}
\begin{eqnarray}
\label{Equation:Asymptotic_Expression}
\alpha_{m_1,m_2,m_3}^\psN&=&\sum_{\substack{\pi\in \cP(m) \\ q(\pi)=-1}}\sum_{\substack{i:V\rightarrow [N] \\ \ker(i)=\pi}}N^{-m/2+1}\C_3(T_{m_1,m_2,m_3},i)+o(1)
\end{eqnarray}
\end{corollary}

\section{The leading order of $\alpha_{m_1,m_2,m_3}^\psN$}
\label{sec:leading-order}
Equation (\ref{Equation:Asymptotic_Expression}) says that partitions 
$\pi\in\cP(m)$ satisfying $q(\pi)=-1$ and for which $C_{\pi}\neq 0$, are the only ones which contribute in the large $N$ limit. The objective of this section will be identify
these limit partitions.
For this section we set $r=3$, $m_1,m_2,m_3\in\mathbb{N}$, $m=m_1+m_2+m_3$
and $T_{m_1,m_2,m_3}=(V,E)$ and $\gamma \vcentcolon = \gamma_{m_1,m_2,m_3}$ be defined as in Section \ref{Section:GraphTheory}.
Let us rewrite Equation (\ref{Equation:Asymptotic_Expression}) as,
\begin{eqnarray}\label{Equation:Asymptotic_Expression_2}
\alpha_{m_1,m_2,m_3}^\psN&=&\sum_{\substack{\pi\in \cP(m) \\ q(\pi)=-1 \\C_{\pi}\neq 0}}\sum_{\substack{i:V\rightarrow [N] \\ \ker(i)=\pi}}N^{-m/2+1}\C_3(T_{m_1,m_2,m_3},i)+o(1)
\end{eqnarray}
adding the extra condition $C_{\pi}\neq 0$ doesn't change anything as we are
only removing the partitions whose contribution on
the sum is zero.

\begin{definition}
Let $\pi\in \cP(m)$ we say that $\pi$ is a \textit{\VP} if $q(\pi)=-1$ and $C_{\pi}\neq 0$, in such a case we call the graph $T_{m_1,m_2,m_3}^{\pi}$ a \textit{\VG}. 
We denote by $\VPSet$ and $\VGSet$ to the set of \VP\text{s }and \VG\text{s }respectively. 
\end{definition}

In Lemmas  \ref{Lemma:ValidPartitions(1,-2)Case}  and \ref{Lemma:ValidPartitions(0,-1)Case} we identify the elements of $\VPSet$.

\begin{definition}
Let $\pi\in \cP(m)$,
\begin{enumerate}

\item 
For $1\leq j\leq 3$, we say that $T_{m_j}^{\pi}$ is a \textit{double tree} if $\overline{T_{m_j}^{\pi}}$ is a tree and every edge $\overline{e}\in \overline{E_j^{\pi}}$ has multiplicity $2$ and consists of two edges $e_{i},e_j\in E_1^{\pi}$ with opposite orientations.

\item 
For $\{j_1,j_2\}\subset \{1,2,3\}$, we say that $T_{m_{j_1},m_{j_2}}^{\pi}$ is a \textit{double unicircuit graph} if the following are satisfied:
\begin{itemize}
\item[$(a)$] 
$\overline{T_{m_{j_1}m_{j_2}}^{\pi}}$ is a connected graph with a single \cycle, i.e it has the same number of vertices and edges, 

\item[$(b)$] 
every edge $\overline{e}\in \overline{E_{j_1,j_2}^{\pi}}$ has multiplicity $2$, and consists of two edges $e_i,e_j\in E_{j_1,j_2}^{\pi}$ with opposite orientations, and 

\item[$(c)$]
 all edges $\overline{e}\in \overline{E_{j_1,j_2}^{\pi}}$ in the \cycles of the graph $\overline{T_{m_{j_1}m_{j_2}}^{\pi}}$ are $(j_1,j_2)$-edges.
\end{itemize}\end{enumerate}
Some examples can be seen in Figure \ref{Figure:Double unicyclic and double tree}.
\end{definition}

\begin{figure}
\begin{center}
\includegraphics[width=340pt,height=110pt]{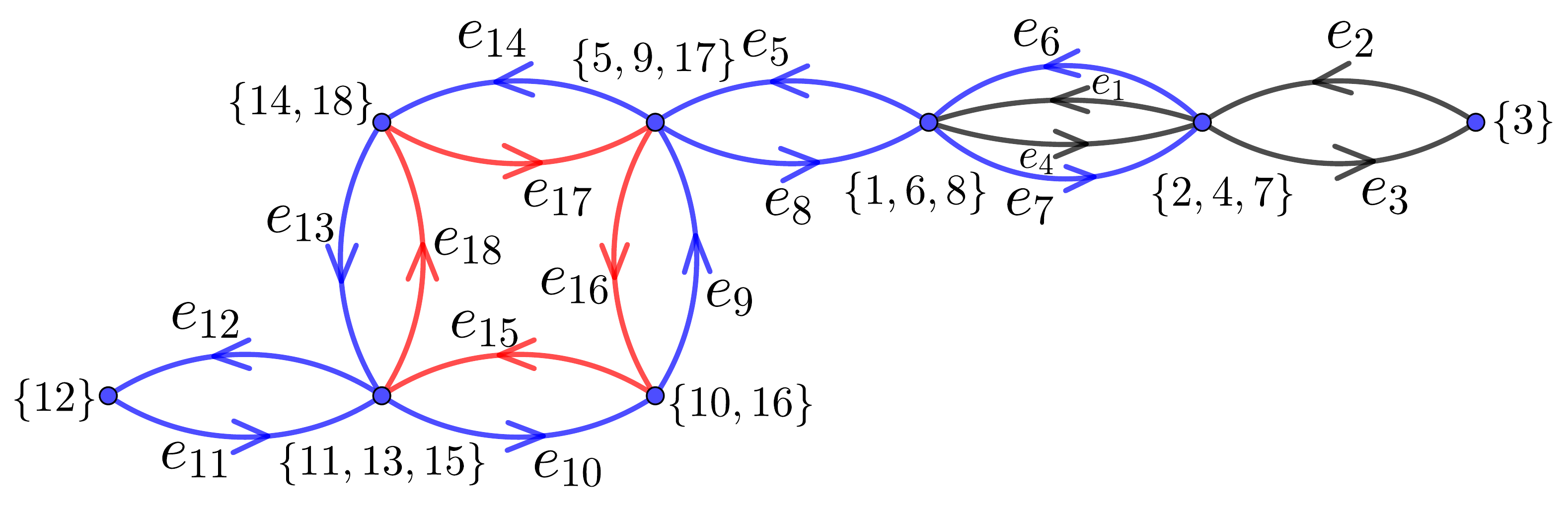}
\\
\includegraphics[width=170pt,height=60pt]{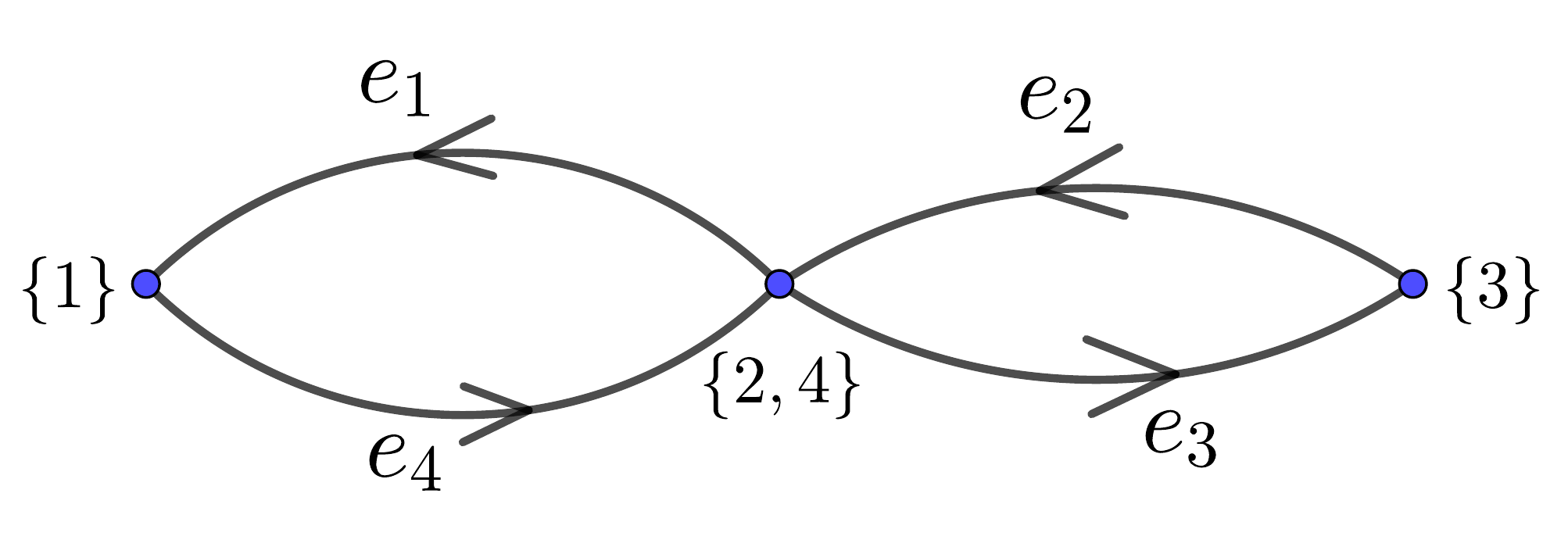}
\includegraphics[width=170pt,height=60pt]{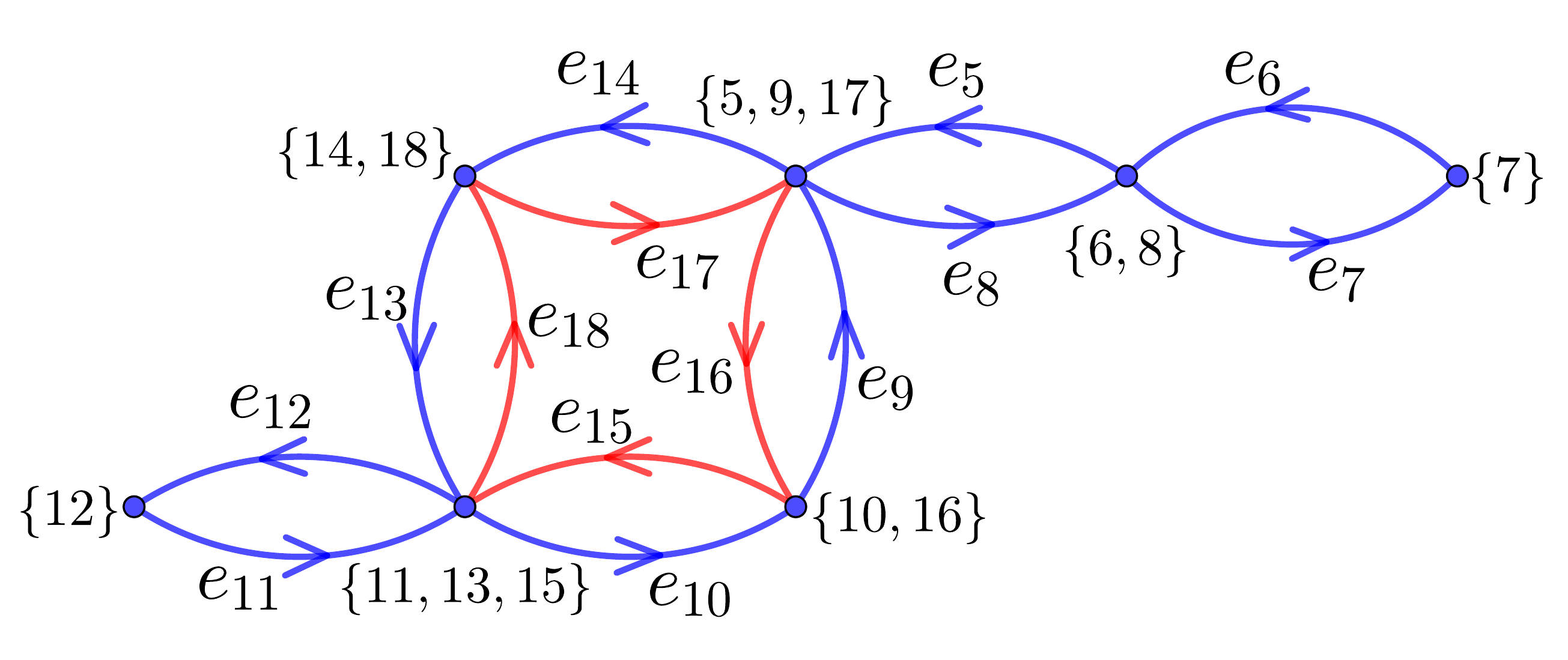}
\end{center}
\caption{\small Above the graph $T_{m_1,m_2,m_3}^{\pi}$ with $m_1=4$, $m_2=10$, $m_3=4$ and 
\begin{center}
$\pi=\{\{1,6,8\},\{2,4,7\},\{3\}, 
\{5,9,17\},\{10,16\},\ab \{11,13,15\},\{12\}\}.$
\end{center} 
Below from left to right the graphs $T_{m_1}^{\pi}$ and $T_{m_2,m_3}^{\pi}$ which are of double tree and double unicircuit type respectively.}
\label{Figure:Double unicyclic and double tree}
\end{figure}

\begin{notation}
We denote by,
$$\C_6=\C_6(x_{1,2},x_{1,2},x_{1,2},x_{2,1},x_{2,1},x_{2,1})$$
$$\C_4=\C_4(x_{1,2},x_{1,2},x_{2,1},x_{2,1})$$
$$\mathring{\C}_4=\C_4(x_{1,1},x_{1,1},x_{1,1},x_{1,1}),$$
these are the sixth and fourth classical cumulants respectively.
\end{notation}

\begin{definition}\label{def:double_tree_types}
Suppose $\pi \in \cP(m)$ and that all three graphs $T_{m_1}^{\pi}$,$T_{m_2}^{\pi}$ and $T_{m_3}^{\pi}$ are double trees.

If $T_{m_1,m_2,m_3}^{\pi}$ is obtained by joining $T_{m_1}^{\pi}$,$T_{m_2}^{\pi}$ and $T_{m_3}^{\pi}$ along the same edge (which must be of multiplicity $6$ in $\overline{E^\pi}$), then we say that $T_{m_1,m_2,m_3}^{\pi}$ is of {\em$2$-$6$-tree type}.

If $T_{m_1,m_2,m_3}^{\pi}$ is obtained by joining $T_{m_1}^{\pi}$,  $T_{m_2}^{\pi}$ and $T_{m_3}^{\pi}$ along two distinct edges (which will each have multiplicity $4$ in $\overline{E^\pi}$), then we say that $T_{m_1,m_2,m_3}^{\pi}$ is of {\em $2$-$4$-$4$-tree type}.
\end{definition}
Some examples of these graphs can be seen in Figure \ref{Figure:3-valid graphs-a}.

\begin{lemma}\label{Lemma:ValidPartitions(1,-2)Case}
Suppose $\pi\in\cP(m)$. Then $C_{\pi}\neq 0$ and $\overset{\rightarrow}q(\pi)=(1,-2)$ if and only if 
$T_{m_1,m_2,m_3}^{\pi}$ is either of  $2$-$6$-tree type or of $2$-$4$-$4$-tree type. In the first case $C_{\pi}=p(N)(\C_6+6\C_4+2)$ and the second case 
$C_{\pi}=p(N)(\C_4+1)^2$, where $p(N)=N^{-m/2+1}N(N-1)\cdots (N-\#(\pi)+1)$. 
\end{lemma}

\begin{proof}
Suppose first $C_{\pi}\neq 0$ and $\overset{\rightarrow}q(\pi)=(1,-2)$, then $\overline{T_{m_1,m_2,m_3}^{\pi}}$ is a tree with all edges of multiplicity $2$ except one of multiplicity $6$ or two edges of multiplicity $4$. Lemma \ref{Theorem:MultiplicityOfCuttingEdges} says that any edge of multiplicity $2$, $\overline{e}\in \overline{E^{\pi}}$, consists of two edges $e_i,e_j\in E^{\pi}$ coming from the same basic cycle and with opposite orientations. Moreover Theorem \ref{Theorem:Maybe} says that $\overline{\pi}\vee\gamma_{m_1,m_2,m_3}=1_m$. Thus, if we have an edge of multiplicity $6$, then such an edge must be a $(1,2,3)$-edge. In addition, Theorem \ref{Theorem:MultiplicityOfCuttingEdges} says that the equivalence class of such an edge consists of two edges from each basic cycle with opposite orientations. Thus $T_{m_1,m_2,m_3}^{\pi}$ is of  $2$-$6$-tree type. 

If we have two edges, each of multiplicity $4$, the same argument forces to the two edges of multiplicity $4$ to be connecting edges, each one can connect at most two basic cycles, so one is a $(i,j)$-edge and the other one must be a $(j,k)$-edge where $\{i, j, k\}$ is some permutation of $\{1,2,3\}$. Thus $T_{m_1,m_2,m_3}^{\pi}$ is of  $2$-$4$-$4$ tree type.

Conversely, let $T_{m_1,m_2,m_3}^{\pi}$ a $2$-$6$-tree type. Clearly $\overset{\rightarrow}q(\pi)=(1,-2)$, so it remains to verify that 
$$C_{\pi} =\sum_{\substack{i:V\rightarrow [N] \\ \ker(i)=\pi}}N^{-m/2+1}\C_3(T_{m_1,m_2,m_3},i) \neq 0.$$ Let $i:V\rightarrow[N]$ such that $\ker(i)=\pi$, Equation (\ref{Equation_Number6}) says that,
\[
\C_3(T_{m_1,m_2,m_3},i)
=
\sum_{\substack{\sigma\in\cP(m) \\ \sigma\vee\gamma=
1_m \\ \sigma\leq \overline{\pi}}} 
\C_{\sigma}(x_{i_1,i_{\gamma(1)}},\dots, x_{i_m,i_{\gamma(m)}}).
\]
Let $\sigma\leq \overline{\pi}$, note that if $\sigma$ has a block of size $1$ then $$\C_{\sigma}(x_{i_1},i_{\gamma_{m_1,m_2,m_3}(1)},\dots, x_{i_m,i_{\gamma_{m_1,m_2,m_3}(m)}})=0,$$ so we may assume every block of $\sigma$ has at least size $2$, let $B=\{u,v\}$ be a block of size $2$ of $\overline{\pi}$, then $B$ is a block of $\sigma$ and the contribution of that block is $\C_2(x_{i_u,i_{\gamma(u)}},x_{i_v,i_{\gamma(v)}})=\C_2(x_{1,2},x_{2,1})=1$. Recall that by definition $\overline{\pi}$ has all its blocks of size $2$ except one of multiplicity $6$. 

Assume that  $B=\{u_1,u_2,v_1,v_2,w_1,w_2\}$ is a block of $\overline{\pi}$  of size $6$ with $e_{u_1},e_{u_2}\in E_1$, $e_{v_1},e_{v_2}\in E_2$ and $e_{w_1},e_{w_2}\in E_3$ and $e_{u_1},e_{v_1}$ and $e_{w_1}$ all with the same orientation. Since $\sigma\leq \overline{\pi}$, $B$ is a disjoint union blocks of $\sigma$. The possibilities are the following.
\begin{enumerate}
\item $B$ is also a block of $\sigma$ in which case the contribution is $\C_6$.
\item $\sigma$ has $3$ blocks each one of size $2$, since we also require $\sigma\vee\gamma_{m_1,m_2,m_3}=1$ and $\sigma$ cannot pair edges with the same orientation (as in that case we get a factor of the form $\C_2(x_{1,2},x_{1,2})=0$). Thus the options are either $u_1$ is paired with $v_2$ or it is paired with $w_2$. Once this choice is made the other pairs are forced. Hence there are only two possibilities:
$$\{\{u_1,v_2\},\{u_2,w_1\},\{v_1,w_2\}\} \text{ or }\{\{u_1,w_2\},\{u_2,v_1\},\{v_2,w_1\}\}$$ each with a contribution of $(\C_2(x_{1,2},x_{2,1}))^3=1$.
\item $\sigma$ has one block of size $4$ and one of size $2$. There are 6 ways of choosing the block of size $2$ such that the pair of edges have the opposite orientations: $\{u_1,v_2\}$, $\{u_1,w_2\}$,$\{u_2,v_1\}$, $\{u_2,w_1\}$, $\{v_1,w_2\}$ or $\{v_2,w_1\}$. In each case the block of size $4$ is determined . Each option contributes $\C_4$.
\end{enumerate}
Adding all together up gives $\C_3(T_{m_1,m_2,m_3},i)=\C_6+6\C_4+2$. Thus
$C_{\pi}=p(N)(\C_6+6\C_4+2) \not = 0$. 
 
If $T_{m_1,m_2,m_3}^{\pi}$ is of $2$-$4$-$4$ tree type then clearly $\overset{\rightarrow}q(\pi)=(1,-2)$. To prove $C_{\pi}\neq 0$ we proceed similarly. Let $i:V\rightarrow [N]$ be such that $\ker(i)=\pi$. Any block of size $4$ of $\overline{\pi}$ permits two options for the blocks of $\sigma\leq \overline{\pi}$, either a block of size $4$ contributing $\C_4$ or two blocks of size $2$ contributing $1$. Since there are two blocks of size $4$ then all possible combinations for the blocks of $\sigma$ are
\begin{enumerate}
\item two blocks of size $4$ with a combined contribution of $\C_4^2$,
\item One block of size $4$ and two of size $2$ which can be done in two ways (by dividing one of the two blocks of $\overline{\pi}$ of size $4$), each one with a contribution of $\C_4$,
\item four blocks of size $2$ with a contribution of $1$
\end{enumerate}
Adding all together up gives $\C_3(T_{m_1,m_2,m_3},i)=\C_4^2+2\C_4+1=(\C_4+1)^2$. Thus $C_{\pi}=p(N)(\C_4+1)^2 \not = 0$. This proves the lemma. 
\end{proof}

\begin{definition}\label{def:(0,-1)-graphs}
Let $\pi \in \cP(m)$. 

\begin{enumerate}

\item 
Suppose $\{i, j, k\} = \{1, 2, 3\}$, and one of the graphs, say $T_{m_i}^{\pi}$, is a graph with two loops at the same vertex and removing these loops results in a double tree.  Also, suppose  $T_{m_j,m_k}^{\pi}$ is a \DUs graph where the \cycles of $\overline{T_{m_j,m_k}^{\pi}}$ is a loop (consisting of an edge from $\overline{T_{m_j}^{\pi}}$ and one from $\overline{T_{m_k}^{\pi}}$), and that $T_{m_1,m_2,m_3}^{\pi}$ is obtained by  joining $T_{m_i}^{\pi}$ to $T_{m_j,m_k}^{\pi}$ at the vertex containing the loops, then we say that $T_{m_1,m_2,m_3}^{\pi}$ is of \textit{$2$-$4$ uniloop type}. Apart from this loop of multiplicity 4, $T_{m_1,m_2,m_3}^{\pi}$ is a double tree.

\item 
Suppose one of the graphs, say $T_{m_i}^{\pi}$, is a double tree, and that $T_{m_2,m_3}^{\pi}$ is a \DUs graph, and $T_{m_1,m_2,m_3}^{\pi}$ is obtained by joining $T_{m_1}^{\pi}$ and $T_{m_2,m_3}^{\pi}$ along some edge. In this case we say that $T_{m_1,m_2,m_3}^{\pi}$ is of \textit{\TFUC}.

\end{enumerate}

\end{definition}

Some examples of these graphs can be seen in Figure \ref{Figure:3-valid graphs-b}.

\begin{lemma}\label{Lemma:ValidPartitions(0,-1)Case}
Let $\pi\in\cP(m)$. Then $C_{\pi}\neq 0$ and $\overset{\rightarrow}q(\pi)=(0,-1)$ if and only if $T_{m_1,m_2,m_3}^{\pi}$ is either of 
$2$-$4$-uniloop type or of \TFUC. In the first case 
$C_{\pi}=p(N)(\mathring{\C}_4+1)$ and in the second case $C_{\pi}=p(N)(\C_4+1)$, with $p(N)$ as before.
\end{lemma}

Before proving Lemma \ref{Lemma:ValidPartitions(0,-1)Case} let us prove the following Lemma.

\begin{lemma}\label{Chapter4_Lemma_Suficient conditions for 24 uniloop or unicircuit}
Let $\pi\in \cP(m)$ be such that $\overline{\pi} \vee \gamma =1_m$. If $T_{m_1,m_2,m_3}^{\pi}$ is a graph satisfying all followings conditions:
\begin{enumerate}
\item 
$\overline{T_{m_1,m_2,m_3}^{\pi}}$ is a connected uni\cycles graph with all edges of multiplicity $2$ except for one of multiplicity $4$,

\item 
All edges in $\overline{E^{\pi}}$ of multiplicity $2$, consist of two edges  in $E^{\pi}$  with the opposite orientation,

\end{enumerate}
then $T_{m_1,m_2,m_3}^{\pi}$ is either of \TFULs or \TFUC.
\end{lemma}

\begin{proof}

Let $\{i,j,k\}=\{1,2,3\}$. Let us consider all possible cases:

\medskip
{\bfseries Case 1. The \cycles of $\overline{T_{m_1,m_2,m_3}^{\pi}}$ is a loop and this is the edge of multiplicity 4.} Any other edge is cutting and hence non-connecting, that forces the edge of multiplicity $4$ to be a $(1,2,3)$-edge, with two edges from one basic cycle, say $T_{m_i}$,  and one edge from each basic cycle $T_{m_j}$ and $T_{m_k}$.  Thus  $T_{m_1,m_2,m_3}^{\pi}$ is of \TFUL.

\medskip
\textbf{Case 2. The edge of multiplicity $4$ is not an edge of the \cycles of $\overline{T_{m_1,m_2,m_3}^{\pi}}$.} Any edge of multiplicity $2$ not in the \cycles is cutting and hence non-connecting. The edge of multiplicity $4$ is cutting and hence it can connect at most two basic cycles. On the other hand, if one edge in the \cycles is connecting, say a $(j,k)$-edge, then all edges in the \cycles are $(j,k)$-edges according to Corollary \ref{Chapter1_Corollary_Configuration of double edges of double unicyclic}. This means the only way to have $\overline{\pi}\vee\gamma=1_m$ is that the edges in the \cycles are all $(j,k)$-edges and the edge of multiplicity $4$ is either an $(i,j)$-edge or an $(i,k)$-edge. Finally as the edge of multiplicity $4$ is cutting, it consists of two edges with opposite orientations from each of $E_i$ and ($E_j$ or $E_k$) (depending whether this is an $(i,j)$-edge or an $(i,k)$-edge). Thus $T_{m_1,m_2,m_3}^{\pi}$ is of \TFUC.\par

\medskip
\textbf{Case 3. The edge of multiplicity $4$ is in the \cycles of $\overline{T_{m_1,m_2,m_3}^{\pi}}$} 
\textbf{and this \cycles is not a loop.} Any edge not in the \cycles is cutting and hence non-connecting. Since $\overline{\pi} \vee \gamma=1_m$ then there has to be a connecting edge in the \cycle. Assume one of the edges of multiplicity $2$ in the \cycles is connecting, say an $(i,j)$-edge. Theorem \ref{Chapter1_Theorem_Configuration of edges of double unicyclic} says that any other adjacent edge of multiplicity $2$ in the \cycles must be an $(i,j)$-edge. Therefore all edges in the circuit must be $(i,j)$-edges except possibly by the edge of multiplicity $4$. Let $\overline{e}=(U,V)$ be the edge of multiplicity $4$. Theorem \ref{Chapter1_Theorem_Configuration of edges of double unicyclic} says that $\overline{e}$ must have an even number of edges from $E_i$, that is, either $1$ or $3$. Similarly, $\overline{e}$ must have either $1$ or $3$ edges from $E_j$. Since $\overline{\pi} \vee \gamma=1_m$ there must be an edge of $\overline{E^{\pi}}$ which connects $E_k$ to $E_i\cup E_j$. As any edge of multiplicity $2$ is either non-connecting or $(i,j)$-edge, $\overline{e}$ is forced to have at least one edge from $E_k$. The only way of satisfying all these conditions is that $\overline{e}$ consist of exactly one edge from each $E_i$ and $E_j$ and two edges from $E_k$. On the other hand, we claim that that $\overline{e}\cap E_k$ is a cutting edge of $\overline{T_{m_k}^{\pi}}$. Indeed, if it is non-cutting then there is a \cycles completely contained in $\overline{T_{m_k}^{\pi}}$. This is also a \cycles of $\overline{T_{m_1,m_2,m_3}^{\pi}}$ but $\overline{T_{m_1,m_2,m_3}^{\pi}}$ has a unique \cycles whose edges are all $(i,j)$-edges (except by $\overline{e}$) so it cannot be contained in $\overline{T_{m_k}^{\pi}}$. We conclude that $\overline{e}\cap E_k$ must be a cutting edge of $\overline{T_{m_k}^{\pi}}$. This proves that $\overline{T_{m_k}^{\pi}}$ is a double tree and therefore the edges in $\overline{e}\cap E_k$ have opposite orientation. 
Let us recall that the number of incoming and outgoing edges of each vertex must be the same. Note that, apart from $\overline{e}$, $U$ is adjacent to edges of multiplicity $2$ whose edges have opposite orientations. This means that each one contributes one incoming and one outgoing edge of $U$. Moreover, we proved that the edges in $\overline{e}$ coming from $E_k$ are in opposite orientations. Therefore it must be that the edges in $\overline{e}$ from $E_i$ and $E_j$ have opposite orientation. This proves that all edges of $T_{m_i,m_j}^{\pi}$ have multiplicity $2$ and with the opposite orientation. Moreover, the \cycles of $\overline{T_{m_1,m_2,m_3}^{\pi}}$ is also a \cycles of $\overline{T_{m_i,m_j}^{\pi}}$ where any edge consist of one edge from each of $E_i$ and $E_j$. $\overline{T_{m_i,m_j}^{\pi}}$ cannot have another \cycles as that would represent another \cycles of $\overline{T_{m_1,m_2,m_3}^{\pi}}$, i.e. $T_{m_i,m_j}^{\pi}$ is of double uni\cycles type, hence $T_{m_1,m_2,m_3}^{\pi}$ is of \TFUC. \\
Let us recall that at the very beginning of this case we assumed that one of the edges of multiplicity $2$ in the \cycles is non-cutting, so, to finish this case off let us now assume that none of these edges is connecting. This forces to the edge of multiplicity $4$ to connect all basic cycles, that is, this edge must consist of one edge from each $E_i$ and $E_j$ and two edges from $E_k$. Let $\overline{e}=(U,V)$ be this edge of multiplicity $4$. Observe that, apart from $\overline{e}$, $U$ is adjacent to only non-connecting edges and therefore these edges consist of either $0$ or $2$ edges from $E_i$. But $\overline{e}$ has only one edge from $E_i$. Therefore $\deg(U)_i$ is odd which yields a contradiction. This finishes the proof.\par

\end{proof}

Let us now prove Lemma \ref{Lemma:ValidPartitions(0,-1)Case}.

\begin{proof}[Proof of Lemma \ref{Lemma:ValidPartitions(0,-1)Case}]
Suppose $C_{\pi}\neq 0$ and $\overset{\rightarrow}q(\pi)=(0,-1)$. Then $\overline{T_{m_1,m_2,m_3}^{\pi}}$ is a unicircuit graph with all edges of multiplicity $2$ except by either one of multiplicity $4$ or two edges of multiplicity $3$. We claim that the second case is not possible. Indeed, assume there are two edges of multiplicity $3$. Since any edge not in the \cycles is cutting and then has even multiplicity, these two edges must lie in the \cycle. The \cycles has at least length $3$. Therefore there is an edge of multiplicity $2$ in the \cycles which is adjacent to one of the edges with multiplicity $3$. We call $U$ to the common vertex of the edge of multiplicity $2$ and the edge of multiplicity $3$. Observe that $U$ is adjacent to only edges of multiplicity $2$ apart from the edge of multiplicity $3$ which means that its degree is odd which is impossible so it must be all edges have multiplicity $2$ except by one of multiplicity $4$. Theorem \ref{Theorem:Maybe} says that the following conditions hold, 
\begin{enumerate}

\item 
$\overline{\pi}\vee\gamma_{m_1,m_2,m_3}=1_m$

\item 
All edges in $\overline{E^{\pi}}$ of multiplicity $2$, consist of two edges  in $E^{\pi}$  with the opposite orientation.

\end{enumerate}
The latest means that $T_{m_1,m_2,m_3}^{\pi}$ satisfies all conditions of Lemma \ref{Chapter4_Lemma_Suficient conditions for 24 uniloop or unicircuit}. It follows that $T_{m_1,m_2,m_3}^{\pi}$ must be either of \TFULs or $2$-$4$ unicircuit type.

\medskip
For the converse, if either $T_{m_1,m_2,m_3}^{\pi}$ is of \TFUCs or of $2$-$4$ uniloop type then clearly $\overset{\rightarrow}q(\pi)=(0,-1)$. It remains to prove that $C_{\pi}\neq 0$ in each case. Let $i:V\rightarrow [N]$ be such that $\ker(i)=\pi$. Assume first $T_{m_1,m_2,m_3}^{\pi}$ is of \TFUL. Let us recall Equation (\ref{Equation_Number6}):
\begin{equation}\label{Eq1}
\C_3(T_{m_1,m_2,m_3},i)=\sum_{\substack{\sigma\in\cP(m) \\ \sigma\vee\gamma=1_m \\ \sigma\leq \overline{\pi}}} \C_{\sigma}(x_{i_1,i_{\gamma(1)}},\dots, x_{i_m,i_{\gamma(m)}}).
\end{equation}
Let $\sigma\leq \overline{\pi}$, we proceed similarly to Lemma \ref{Lemma:ValidPartitions(1,-2)Case}, we may assume that any block of size $2$ of $\overline{\pi}$ is also a block of $\sigma$ and it has a contribution on $\C_{\sigma}(x_{i(1),i(\gamma(1))},\dots,x_{i(m),i(\gamma(m))})$ of $\C_2(x_{1,2},x_{2,1})=1$. The block of size $4$ of $\overline{\pi}$ is of the form $B=\{u_1,u_2,v,w\}$ with $e_{u_1},e_{u_2}\in E_i$, $e_v\in E_j$, $e_w\in E_k$ and $\{i,j,k\}=\{1,2,3\}$. The orientation of the edges $e_{u_1},e_{u_2},e_v,e_w$ doesn't matter as they are all loops. Since $\sigma \leq \overline{\pi}$ then $\sigma$ restricted to $B$ is either $B,\{\{u_1,v\},\{u_2,w\}\}$ or $\{\{u_1,w\},\{u_2,v\}\}$, in the first case $\C_{\tau}(x_{i(1),i(\gamma(1))},\dots,x_{i(m),i(\gamma(m))})=\mathring{\C}_4$, in the last two cases $\C_{\tau}(x_{i(1),i(\gamma(1))},\dots,x_{i(m),i(\gamma(m))})=1$, adding all together up gives $\C_3(T_{m_1,m_2,m_3},i)=\mathring{\C}_4+2$ and then,
$$C_{\pi}^{(3)}=\sum_{\substack{i:V\rightarrow [N]\\ \ker(i)=\pi}}\C_3(T_{m_1,m_2,m_3},i)=p(N,\pi)(\mathring{\C}_4+2) \neq 0.$$

Similarly, if $T_{m_1,m_2,m_3}^{\pi}$ is of \TFUCs then $T_{m_i}^{\pi}$ is a double tree, $T_{m_j,m_k}^{\pi}$ is a double uni\cycle graph with $\{i,j,k\}=\{1,2,3\}$ and $T_{m_1,m_2,m_3}^{\pi}$ is obtained by joining $T_{m_i}^{\pi}$ and $T_{m_j,m_k}^{\pi}$ along some edge. There are two possible cases, either we join them along one edge in the \cycles of $\overline{T_{m_j,m_k}^{\pi}}$ or along one edge not in the \cycle. In the first case the edge of multiplicity $4$ correspond to a block of size $4$ of $\overline{\pi}$ of the form $B=\{u_1,u_2,v,w\}$ with $e_{u_1},e_{u_2}\in E_i$, $e_v\in E_j$ and $e_w\in E_k$. We assume without loss of generality that $e_{u_1}$ and $e_{v}$ have the same orientation. Let $\sigma \leq \overline{\pi}$ as in Equation (\ref{Eq1}), any block of size two of $\overline{\pi}$ is also a block of $\sigma$ and the corresponding contribution on $\C_{\sigma}(x_{i(1),i(\gamma(1))},\dots,x_{i(m),i(\gamma(m))})$ is $1$. The block of size $4$ of $\overline{\pi}$ can be either a block of $\sigma$ in which case the contribution is $\C_4$ or the union of two blocks of $\sigma$ each of size $2$ in which case the contribution is $1$. Note that the second case is only possible when the blocks of $\sigma$ are $\{u_1,v\}$ and $\{u_2,w\}$, as this is the only case where $\sigma$ pairs edges with opposite orientation and $\sigma \vee \sigma =1_m$. Therefore there are only two possible $\sigma's$, one with contribution $\C_4$ and one with contribution $1$, hence $\C_3(T_{m_1,m_2,m_3},i)=\C_4+1$, thus $C_{\pi}^{(3)}=p(N,\pi)(\C_4+1)$. In the second case the edge of multiplicity $4$ is of the form $B=\{u_1,u_2,v_1,v_2\}$ with $e_{u_1},e_{u_2}\in E_i$ and $e_{v_1},e_{v_2}$ in either $E_j$ or $E_k$. We may assume without loss of generality that $e_{v_1},e_{v_2}\in E_j$ and $e_{u_1},e_{v_1}$ have the same orientation. In this case $\sigma$ restricted to $B$ can be either $B$ itself or $\{\{u_1,v_2\},\{u_2,v_1\}\}$ with contributions $\C_4$ and $1$ respectively. Therefore, we get the same conclusion: $C_{\pi}^{(3)}=p(N,\pi)(\C_4+1) \neq 0$. This finishes the proof.
\end{proof}

Lemmas \ref{Lemma:ValidPartitions(0,-1)Case} and \ref{Lemma:ValidPartitions(1,-2)Case} describe almost all \VG\text{s}, the remaining case is $\pi\in\cP(m)$ such that $C_{\pi}\neq 0$ and $\overset{\rightarrow}q(\pi)=(-1,0)$, this motivates the following notation.

\begin{notation}\label{Chapter4_Definition_double bicircuit}
Let $\pi \in \cP(m)$. If $\vec{q}(\pi)=(-1,0)$ and $C_{\pi}^{(3)} \neq 0$ we say that the graph $T_{m_1,m_2,m_3}^{\pi}$ is a \textit{double bi\cycles graph}.
\end{notation}

\begin{notation}\label{Notation:ValidGraphsEachType}
\text{  }
\begin{enumerate}
\item We denote the set of \TSTTs graphs by $\TSTTSet$, 
\item We denote the set of \TFFTTs graphs by $\TFFTTSet$,
\item We denote the set of \TFULs graphs by $\TFULSet$,
\item We denote the set of \TFUCs type graphs by $\TFUCSet$,
\item We denote the set of \DBs graphs by $\DBSet$.
\end{enumerate}
\end{notation}

Lemmas \ref{Lemma:ValidPartitions(0,-1)Case} and \ref{Lemma:ValidPartitions(1,-2)Case} prove that \VG\text{s} must be one of the cases described in either \ref{def:double_tree_types}, \ref{def:(0,-1)-graphs} or \ref{Chapter4_Definition_double bicircuit}, i.e.
$$\VGSet = \TSTTSet \cup \TFFTTSet \cup \TFULSet \cup \TFUCSet \cup \DBSet.$$ 
Examples  of each of the possible limit graphs can be seen in Figures \ref{Figure:3-valid graphs-a}, \ref{Figure:3-valid graphs-b}, and \ref{Figure:3-valid graphs}.

\begin{figure} 
\begin{center}
\includegraphics[width=0.49\textwidth]{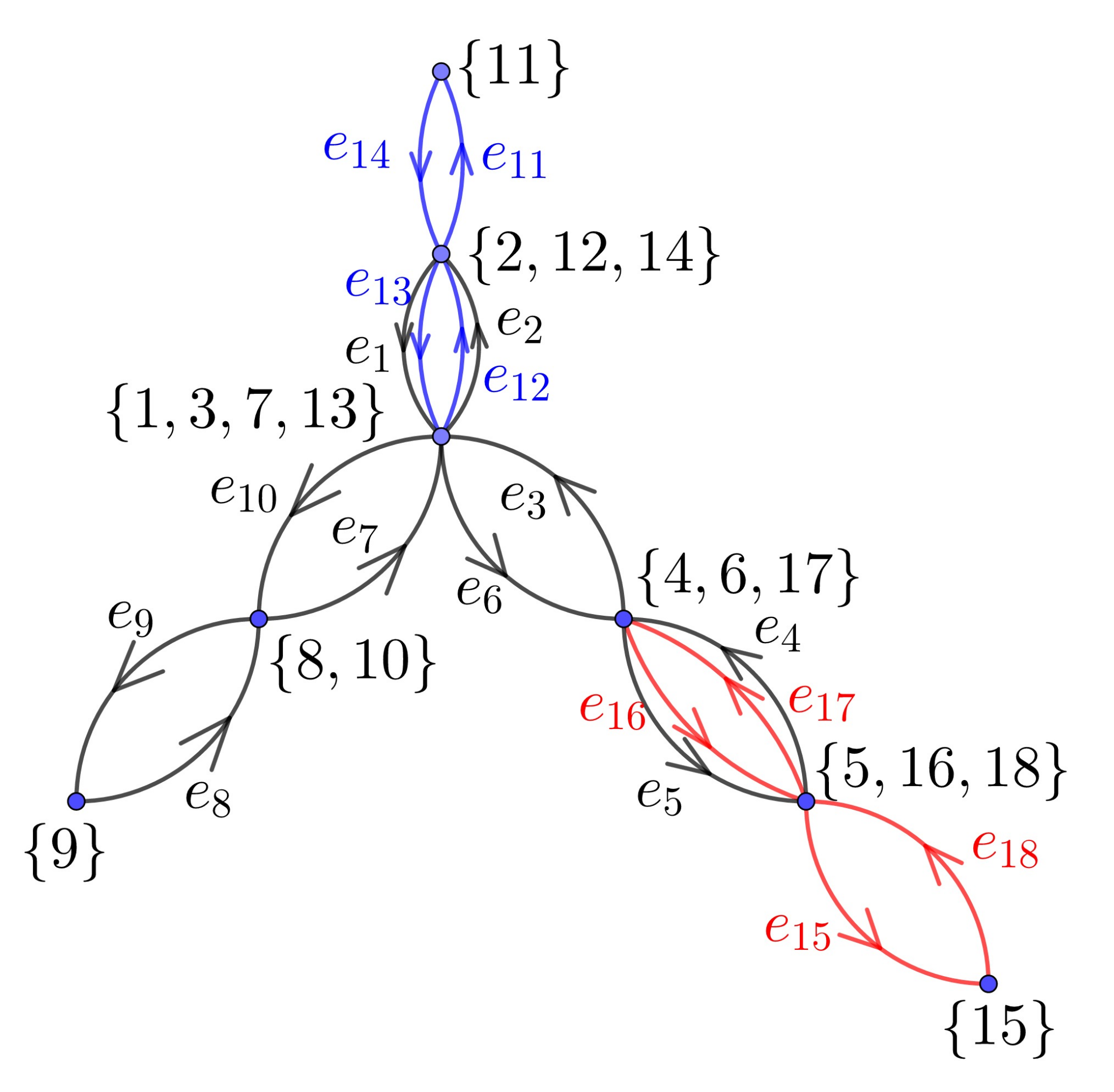} 
\includegraphics[width=0.49\textwidth]{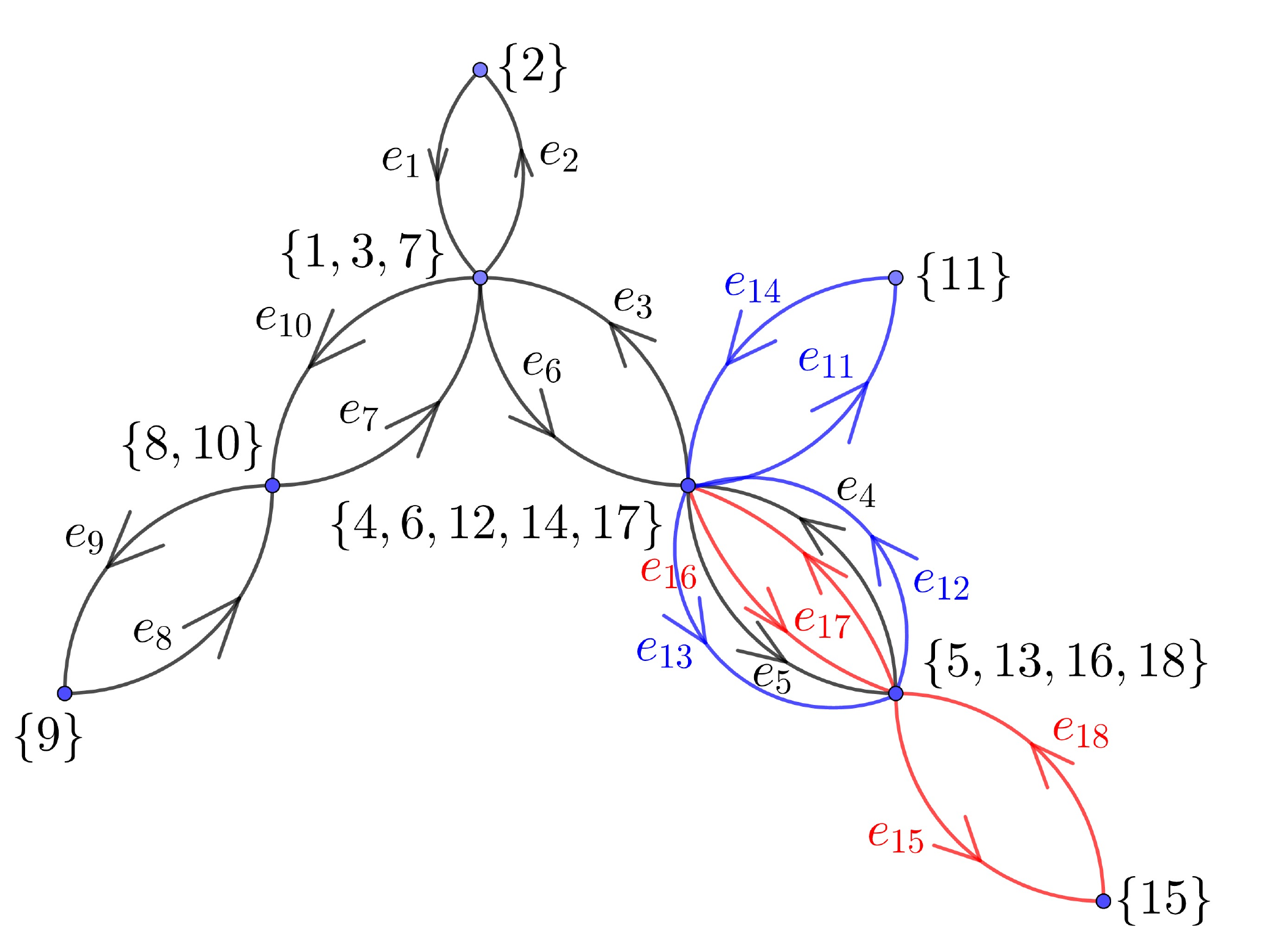}
\end{center}
\caption{\small Left: A 2-4-4 tree type, $T_{10,4,4}^{\pi}$ corresponding to $\pi=\{\{1,3,7,13\},  \{2,12,14\}, \{4,6,17\}, \{5,16,18\}, \ab\{8,\ab 10\}, \ab\{9\},\{11\},\{15\}\}$. Right: A 2-6 tree type, $T_{10,4,4}^{\pi}$ corresponding to $\pi=\{\{1,3,7\}, \{2\}, \{4,6,12,14,17\}, \{5,13, 16,\ab 18\}, \{8,10\}, \{9\}, \{11\}, \{15\}\}$. In this figure and the next two, $T_{m_1}^{\pi},T_{m_2}^{\pi}$ and $T_{m_3}^{\pi}$ are coloured black, blue and red respectively.\label{Figure:3-valid graphs-a}} 
\end{figure}

\begin{figure}
\begin{center}
\includegraphics[width=0.49\textwidth]{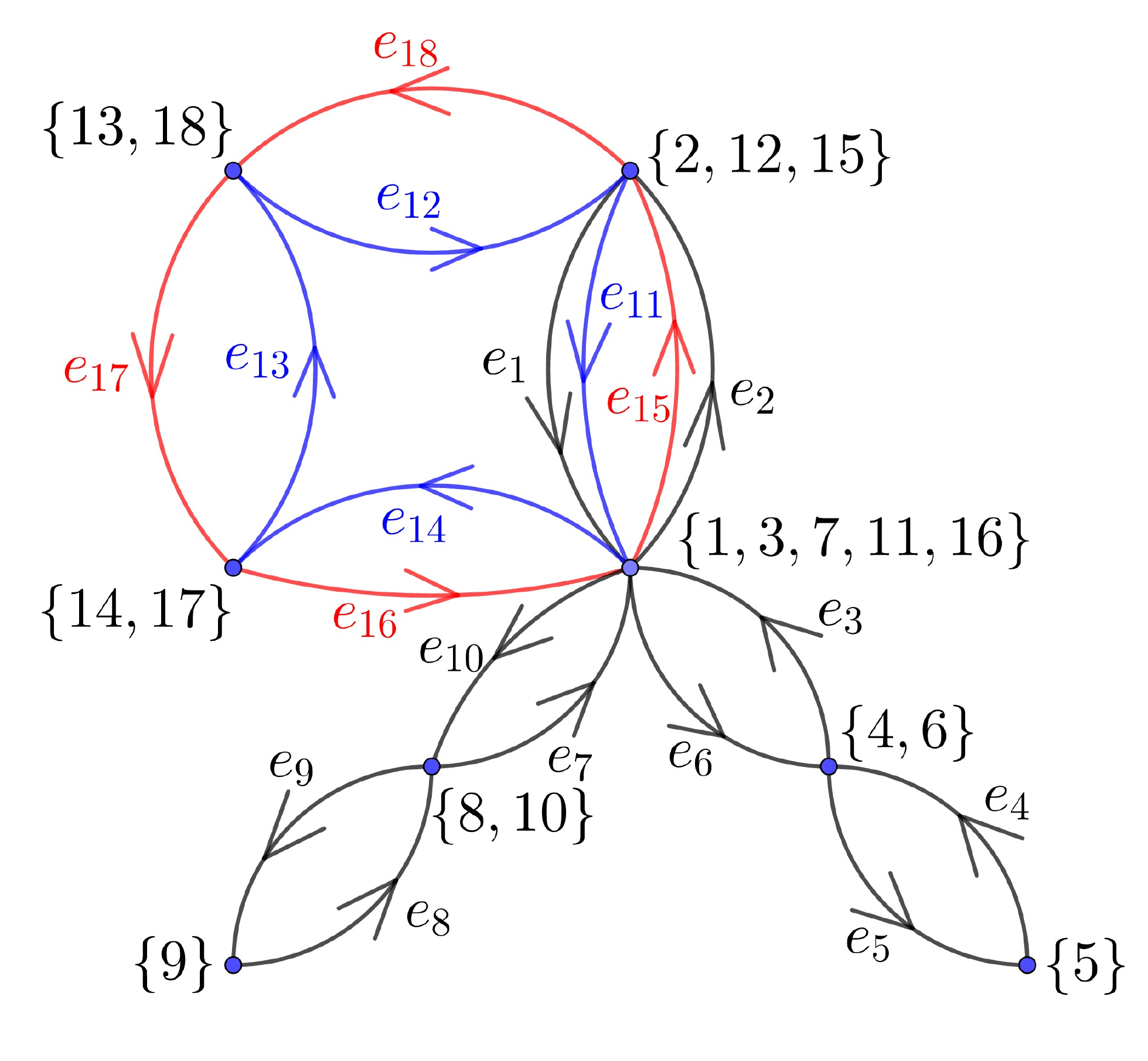}
\includegraphics[width=0.49\textwidth]{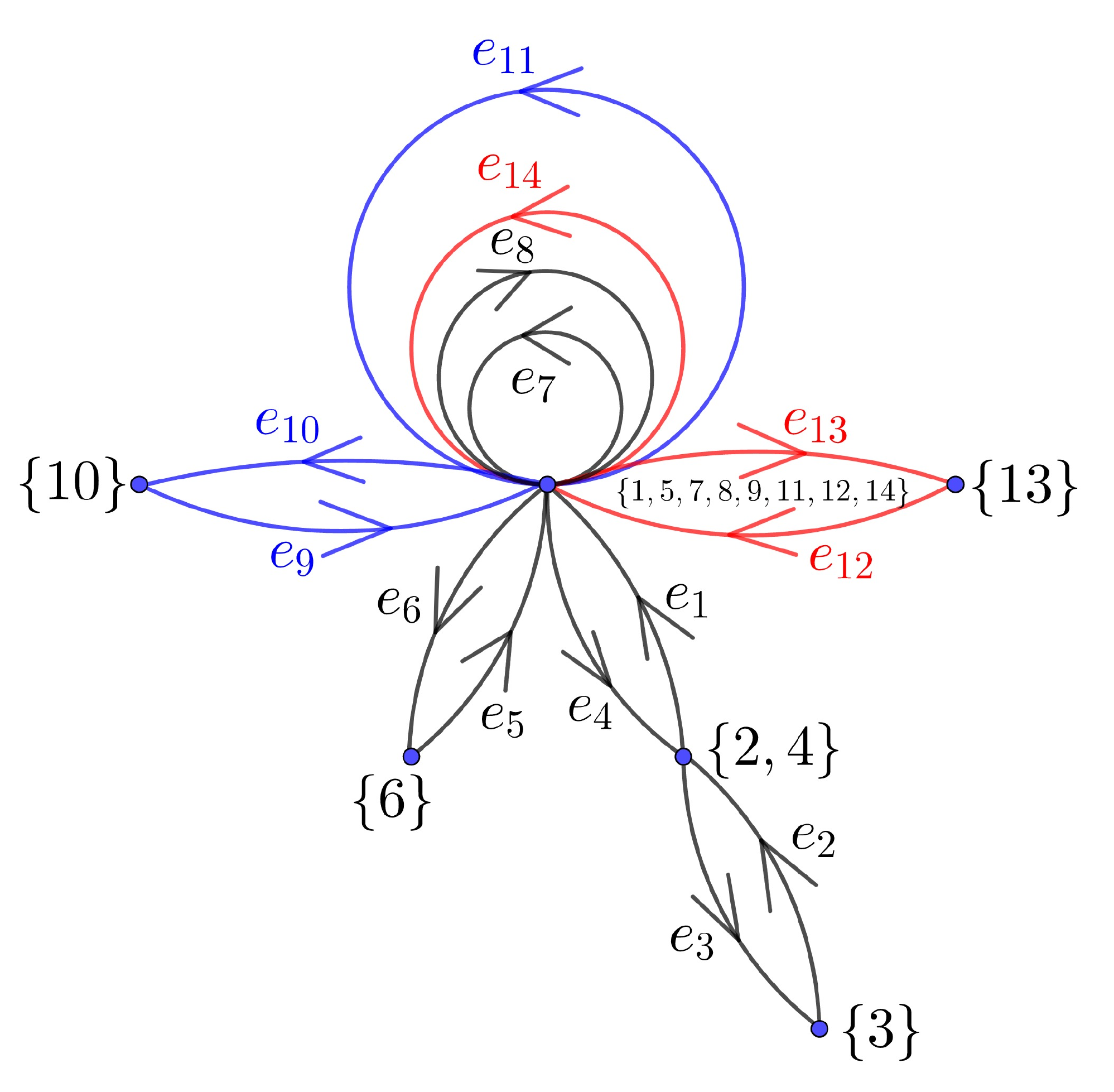}
\end{center}
\caption{\small Left: A 2-4 unicircuit type, $T_{10,4,4}^{\pi}$ corresponding to $\pi=\{\{1,3,7,11,16\},\{2,12,15\},\{4,6\}, \{5\}, \{8,\ab 10\},\{9\},\{13,18\},\{14,17\}\}$. Right: 2-4 uniloop type, $T_{8,3,3}^{\pi}$ corresponding to $\pi=\{\{1,5,7,8,9,11,12,14\},\{2,4\},\{3\}, \ab\{6\},\{10\},\{13\}\}$.\label{Figure:3-valid graphs-b}}
\end{figure}

\setbox1=\hbox to 290pt{\begin{minipage}{290pt}{\small
{
\begin{multline*}
\pi=\{\{1,5,19\},\{2,24\},\{3,23\},\{4,20,22\},\{6,10,17\}, \\ \{7,12,14,16\}\ab,\{8,11\},\{9,18\},\{13\},\{15\},\{21\}\}.
\end{multline*}
}}\end{minipage}}

\setbox2=\hbox to 290pt{\begin{minipage}{290pt}{\small
{\setlength{\textwidth}{190pt}
\begin{multline*}
\pi= \{\{1,10,\ab12\}, \{2,19\}, \{3,18,24\},\{4,23\}, \{5,7,22\},\\ \{6\},
\{8,16,21\},\{9,13,15\}, \{11\}, \{14\},\{17,20\}\}.
\end{multline*}
}}\end{minipage}}

\begin{figure}
\begin{center}
\includegraphics[width=0.49\textwidth]{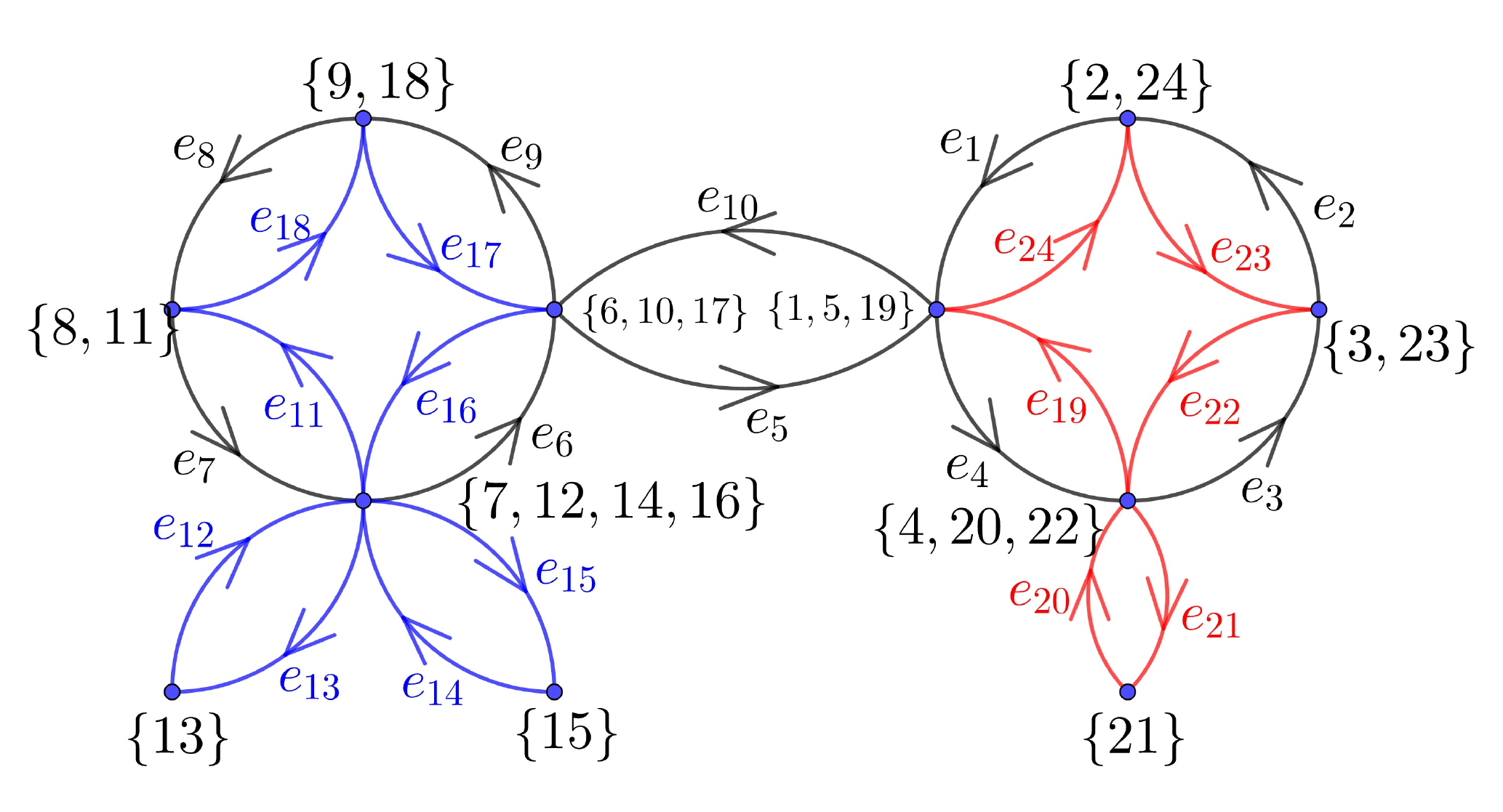}
\includegraphics[width=0.49\textwidth]{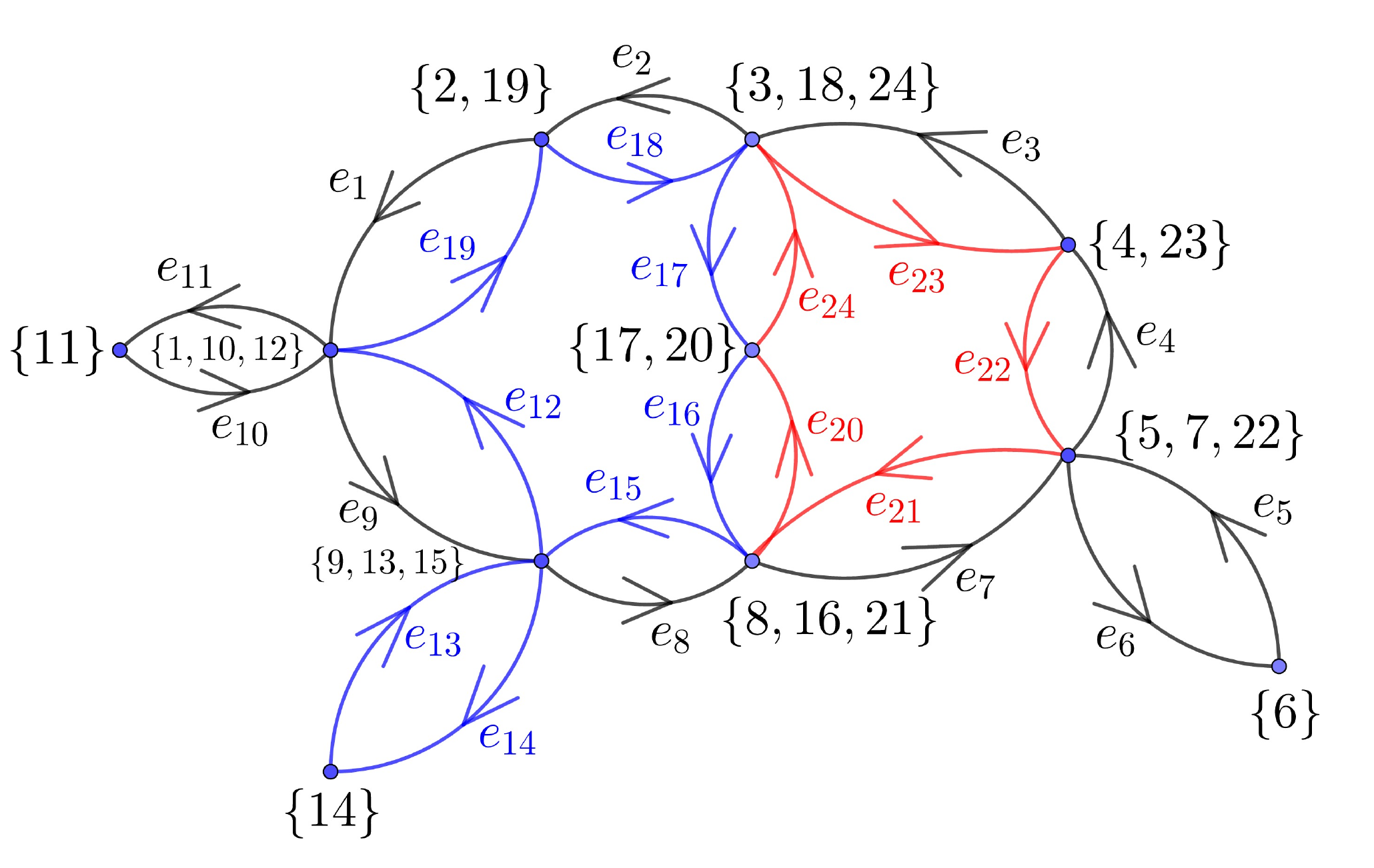}
\end{center}
\caption{\small
Two examples of double bicircuit graphs. Left: A double bicircuit type, $T_{10,8,6}^{\pi}$ corresponding to $\pi=\{\{1,5,19\},\{2,24\},\{3,23\},\{4,20,22\},\{6,10,17\}, \{7,12,14,16\}\ab,\{8,11\},\{9,18\},\{13\},\{15\},\{21\}\}.$
Right: A double bicircuit type, $T_{11,8,5}^{\pi}$ corresponding to $\pi= \{\{1,10,\ab12\}, \{2,19\}, \{3,18,24\},\{4,23\}, \{5,7,22\}, \{6\},
\{8,16,21\},\{9,13,15\}, \{11\}, \{14\},\{17,20\}\}.$
}
\label{Figure:3-valid graphs}
\end{figure}

\begin{lemma}\label{Lemma:ValidPartitions(-1,0)Case}
Let $\pi\in\cP(m)$ such that $T_{m_1,m_2,m_3}^{\pi}$ is of \DB, then $\C_3(T_{m_1,m_2,m_3},i)=1$ for any $i:V\rightarrow[N]$ such that $\ker(i)=\pi$, consequently, $$C_{\pi}=p(N)$$
with $p(N)$ as before.
\end{lemma}

\begin{proof}
Let $i:V\rightarrow[N]$ be such that $\ker(i)=\pi$. Equation (\ref{Equation_Number6}) says,
\begin{equation*}
\C_3(T_{m_1,m_2,m_3},i)=\sum_{\substack{\sigma\in\cP(m) \\ \sigma\vee\gamma=1_m \\ \sigma\leq \overline{\pi}}} \C_{\sigma}(x_{i_1,i_{\gamma(1)}},\dots, x_{i_m,i_{\gamma(m)}}).
\end{equation*}
Let $\sigma \leq \overline{\pi}$, as done before, we may assume that any block of $\sigma$ must have at least size $2$. On the other hand as $q_2(\pi)=0$ then all blocks of $\overline{\pi}$ must be of size $2$ which forces to $\sigma$ be equal to $\overline{\pi}$, moreover any block $\{u,v\}$ of $\sigma$ contributes to $\C_{\sigma}(x_{i_1,i_{\gamma(1)}},\dots, x_{i_m,i_{\gamma(m)}})$ a factor of the form $\C_2(x_{i_u,i_{\gamma(u)}},x_{i_v,i_{\gamma(v)}})=\C_2(x_{1,2},x_{2,1})=1$, therefore $\C_3(T_{m_1,m_2,m_3},i)=1$ as required.
\end{proof}

\begin{corollary}\label{Corollary:FirstExpressionOfAlpha}
\begin{multline*}
\alpha_{m_1,m_2,m_3} = \sum_{\substack{\pi\in \cP(m) \\ T_{m_1,m_2,m_3}^{\pi}\in \TSTTSet}}(\C_6+6\C_4+2)
+ \sum_{\substack{\pi\in \cP(m) \\ T_{m_1,m_2,m_3}^{\pi}\in \TFFTTSet}}(\C_4+1)^2 \\
+ \sum_{\substack{\pi\in \cP(m) \\ T_{m_1,m_2,m_3}^{\pi}\in \TFULSet}}(\mathring{\C}_4+2)
+ \sum_{\substack{\pi\in \cP(m) \\ T_{m_1,m_2,m_3}^{\pi}\in \TFUCSet}}(\C_4+1)
+ \sum_{\substack{\pi\in \cP(m) \\ T_{m_1,m_2,m_3}^{\pi}\in \DBSet}}1 \\
\end{multline*}
\end{corollary}
\begin{proof}
Equation (\ref{Equation:Asymptotic_Expression_2}) says 
\begin{eqnarray*}
\alpha_{m_1,m_2,m_3}^{\psN}=\sum_{\substack{\pi\in\cP(m) \\ T_{m_1,m_2,m_3}^{\pi}\in \VGSet}}C_{\pi}+o(1).
\end{eqnarray*}
For each type of valid graph the contribution is of the form $C_{\pi}=p(N)A$ where $A$ is a number only depending on the type of graph. Since $\lim_{N\rightarrow\infty}p(N)=1$, then when applying limit as $N$ goes to infinity to $\alpha_{m_1,m_2,m_3}^{\psN}$ gives the desired expression.
\end{proof}

\section{Counting the \VG\text{s}}
\label{sec:counting-the-graphs}
Corollary \ref{Corollary:FirstExpressionOfAlpha} determines the third order cumulants provided we are able to count the \VG\text{s}, which is rather complicated. However, we can establish a relation between \VG\text{s} and partitioned permutations; this permits us to write the third order moments in terms of the third order cumulants, naturally. Next, we will count the \VG\text{s} using the set of partitioned permutations. From now on we let $r\in\mathbb{N}$, $m_1,\dots,m_r\in\mathbb{N}$, $m=m_1+\cdots+m_r$, $\gamma_{m_1,\dots,m_r}\in S_m$ and the graph $T_{m_1,\dots,m_r}=(V,E)$ and the permutation $\gamma_{m_1,\dots,m_r}$ as defined in Section \ref{Section:GraphTheory}.

\subsection{The graph $G_{\sigma}^{\gamma}$}

\begin{definition}\label{Definition:Gamma_Sigma^Gamma}
Let $\sigma\in \cP_2(m)$ and $\gamma\in S_m$. We define the unoriented graph $G_{\sigma}^{\gamma}$ as the graph $(V,E)$ where,
\begin{enumerate}
\item $V=\{\text{blocks of } \gamma\sigma\}$
\item $E=\{([u]_{\gamma\sigma},[v]_{\gamma\sigma}):\{u,v\}\text{ is a block of } \sigma\}$
\end{enumerate}
where $[u]_{\gamma\sigma}$ denotes the block of $\gamma\sigma$ containing $u$ and we think of $\sigma$ as a permutation and $\gamma\sigma$ as a partition.
\end{definition}

\begin{definition}
Let $\pi\in \cP_2(m)$ and $\gamma\in S_m$. We define the oriented graph $T(G_{\sigma}^{\gamma})$ as the graph $(V,E)$ where,
\begin{enumerate}
\item $V=\{\text{blocks of } \gamma\sigma\}$
\item $E=\{([u]_{\gamma\sigma},[v]_{\gamma\sigma}):\{u,v\}\text{ is a block of } \sigma\}\bigcup \{([v]_{\gamma\sigma},[u]_{\gamma\sigma}):\{u,v\}\text{ is a block of } \sigma\}$
\end{enumerate}
where for the purpose of definition we think of $\sigma$ as a permutation and $\gamma\sigma$ as a partition.
\end{definition}

\begin{remark}\label{Remark:DoublingTheEdges}
The graph $T(G_{\sigma}^{\gamma})$ is obtained by doubling the edges of the graph $G_{\sigma}^{\gamma}$ once in each orientation.
\end{remark}

\begin{theorem}\label{Theorem:DoublingGGivesQuotientGraph}
Let $\sigma\in \cP_2(m)$. Then the graph $T(G_{\sigma}^{\gamma_{m_1,\dots,m_r}})$ is a quotient graph of $T_{m_1,\dots,m_r}$. More precisely, $T(G_{\sigma}^{\gamma_{m_1,\dots,m_r}})=T_{m_1,m_2,\dots, m_r}^{\gamma_{m_1,\dots,m_r}\sigma}$. 
\end{theorem}
\begin{proof}
By definition $T(G_{\sigma}^{\gamma_{m_1,\dots,m_r}})$ and $T_{m_1,\dots,m_r}^{\gamma_{m_1,\dots,m_r}\sigma}$ have the same set of vertices. Let $([u]_{\gamma_{m_1,\dots,m_r}\sigma},[v]_{\gamma_{m_1,\dots,m_r}\sigma})$ be an edge of $T(G_{\sigma}^{\gamma_{m_1,\dots,m_r}})$ corresponding to the block $\{u,v\}$ of $\sigma$, and note that $\gamma_{m_1,\dots,m_r}\sigma(u) \ab =\gamma_{m_1,\dots,m_r}(v)$. Thus $[u]_{\gamma_{m_1,\dots,m_r}\sigma}=[\gamma_{m_1,\dots,m_r}(v)]_{\gamma_{m_1,\dots,m_r}\sigma}$. Therefore $$([\gamma_{m_1,\dots,m_r}(v)]_{\gamma_r\sigma},[v]_{\gamma_{m_1,\dots,m_r}\sigma})=([u]_{\gamma_{m_1,\dots,m_r}\sigma},[v]_{\gamma_{m_1,\dots,m_r}\sigma}),$$ and so the edge $([u]_{\gamma_{m_1,\dots,m_r}\sigma},[v]_{\gamma_{m_1,\dots,m_r}\sigma})$ of $T(G_{\sigma}^{\gamma_{m_1,\dots,m_r}})$ is also the edge $e_v=([\gamma_{m_1,\dots,m_r}(v)]_{\gamma_{m_1,\dots,m_r}\sigma},[v]_{\gamma_{m_1,\dots,m_r}\sigma})$ of $T^{\gamma_{m_1,\dots,m_r}\sigma}_{m_1,\dots,m_r}$, that is, any edge of $T(G_{\sigma}^{\gamma_{m_1,\dots,m_r}})$ is also an edge of $T_{m_1,\dots,m_r}^{\gamma_{m_1,\dots,m_r}\sigma}$. Hence they must have the same edge set as both have $m$ edges.
\end{proof}

\begin{lemma}\label{Lemma:PropertiesofTsigmagamma}
Let $\sigma\in NC_2(m_1,\dots,m_r)$, then the following are satisfied,
\begin{enumerate}
\item $\sigma \leq \overline{\gamma_{m_1,\dots,m_r}\sigma}$
\item $\overline{\gamma_{m_1,\dots,m_r}\sigma}\vee \gamma_{m_1,\dots,m_r}=1_m$
\item if $\{u,v\}$ is a block of $\sigma$ then $e_u$ and $e_v$ are edges of $T_{m_1,\dots,m_r}^{\gamma_{m_1,\dots,m_r}\sigma}$ connecting the same pair of vertices and with the opposite orientation.
\item $\overline{T_{m_1,\dots,m_r}^{\gamma_{m_1,\dots,m_r}\sigma}}$ is a connected graph where all edges have even multiplicity and each edge $\overline{e}\in \overline{E^{\gamma_{m_1,\dots,m_r}\sigma}}$, which is an equivalent class, consist of the same number of edges in each orientation.
\end{enumerate}
\end{lemma}
\begin{proof}
We prove $(i)$ firstly. Let $\{u,v\}$ be a block of $\sigma$. By definition $([u]_{\gamma_{m_1,\dots,m_r}\sigma},[v]_{\gamma_{m_1,\dots,m_r}\sigma})$ and $([v]_{\gamma_{m_1,\dots,m_r}\sigma},[u]_{\gamma_{m_1,\dots,m_r}\sigma})$ are both edges of $T(G_{\sigma}^{\gamma_{m_1,\dots,m_r}})$ connecting the same pair of vertices. Theorem \ref{Theorem:DoublingGGivesQuotientGraph} proves $T(G_{\sigma}^{\gamma_{m_1,\dots,m_r}})$ and $T_{m_1,\dots,m_r}^{\gamma_{m_1,\dots,m_r}\sigma}$ are the same graph and those edges correspond to $e_u$ and $e_v$ respectively, meaning $u\overset{\overline{\gamma_{m_1,\dots,m_r}\sigma}}\sim v$, thus $\sigma\leq \overline{\gamma_{m_1,\dots,m_r}\sigma}$. Moreover this also proves any block $\{u,v\}$, of $\sigma$, correspond to two edges $e_u,e_v$ of $T_{m_1,\dots,m_r}^{\gamma_{m_1,\dots,m_r}\sigma}$ connecting the same pair of vertices and with the opposite orientations which proves $(iii)$. $(iii)$ implies that any edge of $\overline{T_{m_1,\dots,m_r}^{\gamma_{m_1,\dots,m_r}\sigma}}$ have even multiplicity and it consist of the same number of edges in each orientation, so to prove $(iv)$ it remains to verify that $\overline{T_{m_1,\dots,m_r}^{\gamma_{m_1,\dots,m_r}\sigma}}$ is a connected graph. Note that $1_m=\sigma\vee\gamma_{m_1,\dots,m_r}\leq \overline{\gamma_{m_1,\dots,m_r}\sigma}\vee \gamma_{m_1,\dots,m_r}\leq 1_m$, thus $\overline{\gamma_{m_1,\dots,m_r}\sigma}\vee \gamma_{m_1,\dots,m_r}=1_m$ proving $(ii)$ and that $\overline{T_{m_1,\dots,m_r}^{\gamma_{m_1,\dots,m_r}\sigma}}$ is a connected graph.
\end{proof}

\begin{theorem}\label{Theorem:ExtractingPairingProperties}
Let $\pi\in\cP(m)$ and let $\sigma\in\cP_2(m)$ be such that if $\{u,v\}$ is a block of $\sigma$ then $e_u$ and $e_v$ connect the same pair of vertices of $T_{m_1,\dots,m_r}^{\pi}$ with the opposite orientation. Then, $\gamma_{m_1,\dots,m_r}\sigma\leq \pi$. 
\end{theorem}
\begin{proof}
Let $\{u,v\}$ be a block of $\sigma$. The edge $e_u$ of $T_{m_1,\dots,m_r}^{\pi}$ goes from $[\gamma_{m_1,\dots,m_r}(u)]_{\pi}$ to $[u]_{\pi}$, by hypothesis $e_v$ is connecting the same pair of vertices as $e_u$ but in opposite orientation. Thus $e_v$ goes from $[u]_{\pi}$ to $[\gamma_{m_1,\dots,m_r}(u)]_{\pi}$. However we know that $e_v$ goes from $[\gamma_{m_1,\dots,m_r}(v)]_{\pi}$ to $[v]_{\pi}$ which means that $[u]_{\pi}=[\gamma_{m_1,\dots,m_r}(v)]_{\pi}$, i.e $u\overset{\pi}\sim \gamma_{m_1,\dots,m_r}(v)$, equivalently $u\overset{\pi}\sim \gamma_{m_1,\dots,m_r}\sigma(u)$. Let $p,q$ be such that $p\overset{\gamma_{m_1,\dots,m_r}\sigma}\sim q$. This means that if we think of $\gamma_{m_1,\dots,m_r}\sigma$ as a permutation then $p$ and $q$ are in the same cycle of $\gamma_{m_1,\dots,m_r}\sigma$, thus $p=(\gamma_{m_1,\dots,m_r}\sigma)^{(t)}(q)$ for some $t\geq 1$. By the proved before we get,
$$q\overset{\pi}\sim \gamma_{m_1,\dots,m_r}\sigma(q)\overset{\pi}\sim (\gamma_{m_1,\dots,m_r}\sigma)^{(2)}(q)\cdots \overset{\pi}\sim (\gamma_{m_1,\dots,m_r}\sigma)^{(t)}(q)=p,$$
proving that $p\overset{\pi}\sim q$.
\end{proof}

\subsection{The \VG\text{s} and the set of non-crossing pairings}

For this subsection we will work with the particular case $r=3$, so we set $m=m_1+m_2+m_3$, $\gamma_{m_1,m_2,m_3}$ and $T_{m_1,m_2,m_3}=(V,E)$ as defined in Section \ref{Section:GraphTheory} unless other is specified.

\begin{lemma}\label{Lemma:ImageofNonCrossingPairings}
Let $\sigma\in NC_2(m_1,m_2,m_3)$, then $T_{m_1,m_2,m_3}^{\gamma_{m_1,m_2,m_3}\sigma}$ is a \VG.
\end{lemma}
\begin{proof}
Since $\sigma\in NC_2(m_1,m_2,m_3)$ then $\#(\sigma)+\#(\gamma_{m_1,m_2,m_3}\sigma^{-1})=m-1$, thus $q(\gamma_{m_1,m_2,m_3}\sigma)=\#(\gamma_{m_1,m_2,m_3}\sigma)-m/2=-1$. Now we go through all possible cases where $q(\gamma_{m_1,m_2,m_3}\sigma)=-1$.

\noindent
\textbf{Case 1. $q_1(\gamma_{m_1,m_2,m_3}\sigma)=1$ and $q_2(\gamma_{m_1,m_2,m_3}\sigma)=-2$.} By Lemma \ref{Lemma:PropertiesofTsigmagamma}, $\overline{\gamma_{m_1,\dots,m_r}\sigma}\vee \gamma_{m_1,\dots,m_r}=1_m$. Thus $\overline{T_{m_1,m_2,m_3}^{\gamma_{m_1,m_2,m_3}\sigma}}$ is a tree, it follows that $T_{m_1,m_2,m_3}^{\gamma_{m_1,m_2,m_3}\sigma}\in \TSTTSet\cup \TFFTTSet$ and the proof follows exactly as in Lemma \ref{Lemma:ValidPartitions(1,-2)Case}.

\noindent
\textbf{Case 2. $q_1(\gamma_{m_1,m_2,m_3}\sigma)=0$ and $q_2(\gamma_{m_1,m_2,m_3}\sigma)=-1$.} Lemma \ref{Lemma:PropertiesofTsigmagamma} says $\overline{\gamma_{m_1,\dots,m_r}\sigma}\vee \gamma_{m_1,\dots,m_r}=1_m$ and $\overline{T_{m_1,m_2,m_3}^{\gamma_{m_1,m_2,m_3}\sigma}}$ is a connected graph with all edges of even multiplicity, so it must be $\overline{T_{m_1,m_2,m_3}^{\gamma_{m_1,m_2,m_3}\sigma}}$ is a graph with a single \cycles and all edges of multiplicity $2$ except by one of multiplicity $4$, moreover the edges of multiplicity $2$ consist of edges in opposite orientation, those are exactly the same properties described in Lemma \ref{Chapter4_Lemma_Suficient conditions for 24 uniloop or unicircuit}, thus $T_{m_1,m_2,m_3}^{\gamma_{m_1,m_2,m_3}\sigma}\in \TFULSet\cup \TFUCSet$.

\noindent
\textbf{Case 3. $q_1(\gamma_{m_1,m_2,m_3}\sigma)=-1$ and $q_2(\gamma_{m_1,m_2,m_3}\sigma)=0$}. By Lemma \ref{Lemma:PropertiesofTsigmagamma} all blocks of $\overline{\gamma_{m_1,m_2,m_3}\sigma}$ are of even multiplicity, therefore the condition $q_2(\gamma_{m_1,m_2,m_3}\sigma)=0$ implies all its blocks are of size $2$. Let $i:V\rightarrow [N]$ be such that $\ker(i)=\gamma_{m_1,m_2,m_3}\sigma$, Equation (\ref{Equation_Number6}) says,
\begin{eqnarray*}
\C_3(T_{m_1,m_2,m_3},i)=
\kern-1em
\sum_{\substack{\tau\in\cP(m) \\ \tau\vee\gamma_{m_1,m_2,m_3}=1_m \\ \tau\leq \overline{\gamma_{m_1,m_2,m_3}\sigma}}}
\kern-1em
\C_{\tau}(x_{i_1,i_{\gamma_{m_1,m_2,m_3}(1)}},\dots, x_{i_m,i_{\gamma_{m_1,m_2,m_3}(m)}}),
\end{eqnarray*}
As seen before if $\tau$ has a block of size $1$ then $$\C_{\tau}(x_{i_1,i_{\gamma_{m_1,m_2,m_3}(1)}},\dots, x_{i_m,i_{\gamma_{m_1,m_2,m_3}(m)}})=0,$$ so, all blocks of $\tau$ must be at least of size $2$ and therefore it must be $\tau=\overline{\gamma_{m_1,m_2,m_3}\sigma}$. Moreover Lemma \ref{Lemma:PropertiesofTsigmagamma} says that for a block $B=\{u,v\}$ of $\overline{\gamma_{m_1,m_2,m_3}\sigma}$ the edges $e_u$ and $e_v$ connect the same pair of vertices in $T_{m_1,m_2,m_3}^{\gamma_{m_1,m_2,m_3}\sigma}$ and with the opposite orientations. This means $u$ and $\gamma_{m_1,m_2,m_3}(v)$ are in the same block of $\gamma_{m_1,m_2,m_3}\sigma$. Similarly $v$ and $\gamma_{m_1,m_2,m_3}(u)$ are in the same block of $\gamma_{m_1,m_2,m_3}\sigma$, since $\gamma_{m_1,m_2,m_3}\sigma = \ker(i)$ then by definition of $\ker(i)$ we have $i_u=i_{\gamma_{m_1,m_2,m_3}(v)}$ and $i_v=i_{\gamma_{m_1,m_2,m_3}(u)}$. Thus,
$$\C_2(x_{i_u,i_{\gamma_{m_1,m_2,m_3}(u)}},x_{i_v,i_{\gamma_{m_1,m_2,m_3}(v)}})=\C_2(x_{1,2},x_{2,1})=1.$$
As that holds for any block of $\overline{\gamma_{m_1,m_2,m_3}\sigma}$ then
$$\C_3(T_{m_1,m_2,m_3},i)=\C_{\overline{\gamma_{m_1,m_2,m_3}\sigma}}(x_{i_1,i_{\gamma_{m_1,m_2,m_3}(1)}},\dots, x_{i_m,i_{\gamma_{m_1,m_2,m_3}(m)}})=1.$$
Therefore,
$$C_{\gamma_{m_1,m_2,m_3}\sigma}
= \kern-1em
\sum_{\substack{ i:V\rightarrow [N] \\ \ker(i)=\gamma_{m_1,m_2,m_3}\sigma}}\kern-1.5em 1
=N^{-m/2+1}N(N-1)\cdots (N-\#(\pi)+1)\neq 0.$$
The last equation proves $T_{m_1,m_2,m_3}^{\gamma_{m_1,m_2,m_3}\sigma}\in \DBSet$.
\end{proof}

Lemma \ref{Lemma:ImageofNonCrossingPairings} says that any quotient graph $T_{m_1,m_2,m_3}^{\gamma_{m_1,m_2,m_3}\sigma}$ is a \VG\text{ } \\provided $\sigma$ is a non-crossing pairing, for the rest of the subsection we explore the converse, is any \VG\text{ }obtained by the quotient of a non-crossing pairing composed with $\gamma_{m_1,m_2,m_3}$? it turns out that the answer is affirmative.

\begin{theorem}\label{Theorem:CaracterizationOfEssentialSigma}
Let $r\in\mathbb{N}$, $m_1,\dots,m_r\in\mathbb{N}$, $m=m_1+\cdots +m_r$ and $\gamma_{m_1,\dots,m_r}\in S_m$ be defined as in Section \ref{Section:GraphTheory}. Let $\sigma\in\cP_2(m)$, $\pi\in\cP(m)$ be such that $q(\pi)=2-r$ and let $T_{m_1,\dots,m_r}^{\pi}$ be defined as in Section \ref{Section:GraphTheory}. Then $\sigma\in NC_2(m_1,\dots,m_r)$ and $\pi=\gamma_{m_1,\cdots,m_r}\sigma$ if and only if the following conditions are satisfied:
\begin{enumerate}
\item $\sigma\vee\gamma_{m_1,\dots,m_r}=1_m$
\item if $\{u,v\}$ is a block of $\sigma$ then $e_u$ and $e_v$ connect the same pair of vertices and have the opposite orientation in $T_{m_1,\dots,m_r}^{\pi}$.
\end{enumerate}
\end{theorem}
\begin{proof}
Suppose $\sigma\in NC_2(m_1,\dots,m_r)$ and $\pi=\gamma_{m_1,\dots,m_r}\sigma$. By definition of non-crossing pairing $\sigma\vee\gamma_{m_1,\dots,m_r}=1_m$ so it remains proving the second condition. Let $\{u,v\}$ be a block of $\sigma$. By Lemma \ref{Lemma:PropertiesofTsigmagamma}, $e_u$ and $e_v$ are edges of $T_{m_1,\dots,m_3}^{\gamma_{m_1,\dots,m_3}\sigma}$ connecting the same pair of vertices and with the opposite orientation. However by hypothesis $\gamma_{m_1,\dots,m_3}\sigma=\pi$ so this proves $(ii)$. For the converse, by Theorem \ref{Theorem:ExtractingPairingProperties}, $\gamma_{m_1,\dots,m_r}\sigma\leq \pi$, thus $\#(\gamma_{m_1,\dots,m_r}\sigma)\geq \#(\pi)$. Therefore,
\begin{eqnarray*}
\#(\gamma_{m_1,\dots,m_r}\sigma)+\#(\sigma) &\geq & \#(\pi)+\#(\sigma)=\#(\pi)+m/2 \\
&=& \#(\pi)-m/2+m=m+2-r 
\end{eqnarray*}
On the other hand it is well known, \cite[2.10]{MN}, that,
$$\#(\sigma)+\#(\gamma_{m_1,\dots,m_r}\sigma)+\#(\gamma_{m_1,\dots,m_r})\leq m+2\#(\gamma_{m_1,\dots,m_r}\vee\sigma)=m+2.$$
Thus, $\#(\sigma)+\#(\gamma_{m_1,\dots,m_r}\sigma)\leq m+2-r$, so it must be 
$$\#(\gamma_{m_1,\dots,m_r}\sigma)+\#(\sigma) = \#(\pi)+\#(\sigma)=m+2-r,$$
i.e $\#(\gamma_{m_1,\dots,m_r}\sigma)= \#(\pi)$ and $\gamma_{m_1,\dots,m_r}\sigma= \pi$. Moreover, the equation $\#(\gamma_{m_1,\dots,m_r}\sigma)+\#(\sigma) =m+2-r$ proves $\sigma\in NC_2(m_1,\dots,m_r)$.
\end{proof}

\begin{remark}
Setting $r=3$ in Theorem \ref{Theorem:CaracterizationOfEssentialSigma} says that a \VG, $T_{m_1,m_2,m_3}^{\pi}$ can be written as $T_{m_1,m_2,m_3}^{\gamma_{m_1,m_2,m_3}\sigma}$ for some $\sigma\in NC_2(m_1,m_2,m_3)$ if and only if such a pairing $\sigma$ satisfies conditions $(i)$ and $(ii)$, this provides a powerful tool to test whether such a pairing exist or not just by looking at the properties of the quotient graph $T_{m_1,m_2,m_3}^{\pi}$, as we will verify, it turns out that all \VG\text{s }have a pairing associated to them satisfying those conditions, moreover, for $\sigma_1,\sigma_2\in NC_2(m_1,m_2,m_3)$ with $\sigma_1\neq\sigma_2$ the elements $\gamma_{m_1,m_2,m_3}\sigma_1$ and $\gamma_{m_1,m_2,m_3}\sigma_2$ are distinct when we think of them as permutations in $S_m$, however when we think of them as partitions they might be the same, this means there are possibly more than one pairing inducing the same \VG. Below we will identify the pairings inducing the same \VG.
\end{remark}

\begin{lemma}\label{Theorem:EssentialPairingofTpi26TreeType}
Let $\pi\in\cP(m)$ and let $T_{m_1,m_2,m_3}^{\pi}\in \TSTTSet$, then there exist exactly two non-crossing pairings $\sigma_1,\sigma_2\in NC_2(m_1,m_2,m_3)$ such that $\gamma_{m_1,m_2,m_3}\sigma_i=\pi$ for $i=1,2$.
\end{lemma}
\begin{proof}
By Theorem \ref{Theorem:CaracterizationOfEssentialSigma} it is enough to prove that there exist exactly two pairings $\sigma_1,\sigma_2\in\cP_2(m)$ satisfying the conditions,
\begin{enumerate}
\item $\sigma_{i}\vee\gamma_{m_1,m_2,m_3}=1_m$
\item if $\{u,v\}$ is a block of $\sigma_{i}$ then $e_u$ and $e_v$ connect the same pair of vertices and have the opposite orientation in $T_{m_1,m_2,m_3}^{\pi}$.
\end{enumerate}
Firstly, note that if there exist a pairing $\sigma\in\cP_2(m)$ satisfying $(ii)$ then by definition of $\overline{\pi}$, $\sigma\leq \overline{\pi}$. Let $\overline{e}=\{e_u,e_v\}\in \overline{E^{\pi}}$ be an edge of multiplicity $2$ of $\overline{T_{m_1,m_2,m_3}^{\pi}}$, by the noted before $\{u,v\}$ must be a block of any $\sigma\in\cP_2(m)$ satisfying $(ii)$, i.e. the blocks of $\overline{\pi}$ of size two determines uniquely the blocks of any pairing $\sigma\in\cP_2(m)$ satisfying $(ii)$. Now we consider the block of size $6$ of $\overline{\pi}$, which correspond to the edge of multiplicity $6$ of $\overline{T_{m_1,m_2,m_3}^{\pi}}$. Let $\overline{e}=\{e_{i_1},e_{i_2},e_{i_3},e_{i_4},e_{i_5},e_{i_6}\}$ be this edge of multiplicity $6$, with $e_{i_1},e_{i_2}\in E_1$, $e_{i_3},e_{i_4}\in E_2$ and $e_{i_5},e_{i_6}\in E_3$ and with $e_{i_1},e_{i_3}$ and $e_{i_5}$ having the same orientation.

Any pairing $\sigma\in\cP_2(m)$ satisfying $(ii)$ must pair only elements with opposite orientations, so the possible pairings $\sigma_i$ restricted to $\{i_1,\dots,\ab i_6\}$ are given in the table.

\begin{center}
\begin{tabular}{| c | c |}\hline
$\sigma$ & Blocks of $\sigma$ \\ \hline
$\sigma_1$ & $\{i_1,i_4\},\{i_2,i_5\},\{i_3,i_6\}$  \\
$\sigma_2$ & $\{i_1,i_6\},\{i_2,i_3\},\{i_4,i_5\}$ \\
$\sigma_3$ & $\{i_1,i_4\},\{i_2,i_3\},\{i_5,i_6\}$ \\
$\sigma_4$ & $\{i_1,i_6\},\{i_2,i_5\},\{i_3,i_4\}$ \\
$\sigma_5$ & $\{i_1,i_2\},\{i_3,i_6\},\{i_4,i_5\}$ \\
$\sigma_6$ & $\{i_1,i_2\},\{i_3,i_4\},\{i_5,i_6\}$ \\ \hline
\end{tabular}
\end{center}

By the point made earlier,  $\sigma_i$ and $\overline{\pi}$ are exactly the same pairings restricted to the blocks of size two of $\overline{\pi}$ for $i=1,\dots,6$, that completely determines  all possible pairings $\sigma_1,\dots,\sigma_6$ satisfying condition $(ii)$. We now look at condition $(i)$, let us recall that $\FirstCycleNotation,\SecondCycleNotation$ and $\ThirdCycleNotation$ denotes the cycles of $\gamma_{m_1,m_2,m_3}$. 

Note that any edge of multiplicity $2$ of $\overline{T_{m_1,m_2,m_3}^{\pi}}$, which is an equivalent class, consist of two edges from the same basic cycle, and therefore the corresponding block of $\sigma$ is a non-through string, i.e contains elements from the same cycle of $\gamma_{m_1,m_2,m_3}$. This means that 
$\gamma_{m_1,m_2,m_3}\vee\sigma_i=1_m$ is completely determined by $\sigma_i$ restricted to $\{i_1,\dots,i_6\}$. This can be tested as follows, if a block of $\sigma_i$ contains an element from $\CycleithNotation$ and $\CyclejthNotation$ for some $i,j\in\{1,2,3\}$ then $\CycleithNotation\cup \CyclejthNotation$ must be a block of $\gamma_{m_1,m_2,m_3}\vee\sigma_i$, the table below shows $\gamma_{m_1,m_2,m_3}\vee\sigma_i$ for each $\sigma_i$,
\begin{center}
\begin{tabular}{| c | c | c |}\hline
$\sigma_i$ & Blocks of $\sigma_i$ & $\gamma_{m_1,m_2,m_3}\vee\sigma_i$ \\ \hline
$\sigma_1$ & $\{i_1,i_4\},\{i_2,i_5\},\{i_3,i_6\}$ & $1_m$ \\
$\sigma_2$ & $\{i_1,i_6\},\{i_2,i_3\},\{i_4,i_5\}$ & $1_m$ \\
$\sigma_3$ & $\{i_1,i_4\},\{i_2,i_3\},\{i_5,i_6\}$ & $\FirstCycleNotation\cup \SecondCycleNotation,\ThirdCycleNotation$ \\
$\sigma_4$ & $\{i_1,i_6\},\{i_2,i_5\},\{i_3,i_4\}$ & $\FirstCycleNotation\cup \ThirdCycleNotation,\SecondCycleNotation$ \\
$\sigma_5$ & $\{i_1,i_2\},\{i_3,i_6\},\{i_4,i_5\}$ & $\SecondCycleNotation\cup \ThirdCycleNotation,\FirstCycleNotation$ \\
$\sigma_6$ & $\{i_1,i_2\},\{i_3,i_4\},\{i_5,i_6\}$ & $\gamma_{m_1,m_2,m_3}=\FirstCycleNotation,\SecondCycleNotation,\ThirdCycleNotation$ \\ \hline
\end{tabular}
\end{center}
As shown in the table above only $\sigma_1$ and $\sigma_2$ satisfy the condition $(i)$, so those are the only pairings satisfying both conditions.
\end{proof}

\begin{lemma}\label{Theorem:EssentialPairingofTpi244TreeType}
Let $\pi\in\cP(m)$ and let $T_{m_1,m_2,m_3}^{\pi}\in \TFFTTSet$. Then there exist a unique non-crossing pairing $\sigma\in NC_2(m_1,m_2,m_3)$ such that $\gamma_{m_1,m_2,m_3}\sigma=\pi$.
\end{lemma}
\begin{proof}
We proceed as before, it is enough to prove that there exist a unique pairing $\sigma\in\cP_2(m)$ satisfying,
\begin{enumerate}
\item $\sigma\vee\gamma_{m_1,m_2,m_3}=1_m$
\item if $\{u,v\}$ is a block of $\sigma$ then $e_u$ and $e_v$ connect the same pair of vertices and have the opposite orientation in $T_{m_1,m_2,m_3}^{\pi}$.
\end{enumerate}
As pointed in Theorem \ref{Theorem:EssentialPairingofTpi26TreeType} the edges of multiplicity $2$ of $\overline{T_{m_1,m_2,m_3}^{\pi}}$ determines uniquely the blocks of any $\sigma$ satisfying $(ii)$. Let $\overline{e}=\{e_{i_1},e_{i_2},e_{i_3},e_{i_4}\}$ and $\overline{e^\prime}=\{e_{j_1},e_{j_2},e_{j_3},e_{j_4}\}$ be the edges of multiplicity $4$ of $\overline{T_{m_1,m_2,m_3}^{\pi}}$. Without loss of generality assume $\overline{e}$ is a $(1,2)$-edge and $\overline{e^\prime}$ is a $(1,3)$-edge such that $e_{i_1},e_{i_2},e_{j_1},e_{j_2}\in E_1$, $e_{i_3},e_{i_4}\in E_2$, $e_{j_3},e_{j_4}\in E_3$ and with $e_{i_1}$ and $e_{i_3}$ having the same orientation and $e_{j_1}$ and $e_{j_3}$ having the same orientation. Any $\sigma$ satisfying condition $(ii)$ must pair only elements in the same equivalence class and with opposite orientations, so the possible $\sigma_i$ restricted to $\{i_1,\dots,i_4,j_1,\dots,j_4\}$ are given in next table.
\begin{center}
\begin{tabular}{| c | c | c |}\hline
$\sigma_i$ & Blocks of $\sigma_i$ & $\gamma_{m_1,m_2,m_3}\vee\sigma_i$ \\ \hline
$\sigma_1$ & $\{i_1,i_2\},\{i_3,i_4\},\{j_1,j_2\},\{j_3,j_4\}$ & $\gamma_{m_1,m_2,m_3}$ \\
$\sigma_2$ & $\{i_1,i_2\},\{i_3,i_4\},\{j_1,j_4\},\{j_2,j_3\}$ & $\FirstCycleNotation\cup \ThirdCycleNotation,\SecondCycleNotation$ \\
$\sigma_3$ & $\{i_1,i_4\},\{i_2,i_3\},\{j_1,j_2\},\{j_3,j_4\}$ & $\FirstCycleNotation\cup \SecondCycleNotation,\ThirdCycleNotation$ \\
$\sigma_4$ & $\{i_1,i_4\},\{i_2,i_3\},\{j_1,j_4\},\{j_2,j_3\}$ & $1_m$ \\ \hline
\end{tabular}
\end{center}
Last column shows that only $\sigma_4$ satisfies condition $(i)$ and so that is the only pairing.
\end{proof}

\begin{lemma}\label{Theorem:EssentialPairingofTpiDUL}
Let $\pi\in\cP(m)$ and let $T_{m_1,m_2,m_3}^{\pi}\in \TFULSet$. Then there exist exactly two non-crossing pairings $\sigma_1,\sigma_2\in NC_2(m_1,m_2,m_3)$ such that $\gamma_{m_1,m_2,m_3}\sigma_i=\pi$ for $i=1,2$.
\end{lemma}
\begin{proof}
It suffices to prove that there exist exactly two pairings $\sigma_1,\sigma_2\in\cP_2(m)$ satisfying,
\begin{enumerate}
\item $\sigma_i\vee\gamma_{m_1,m_2,m_3}=1_m$
\item if $\{u,v\}$ is a block of $\sigma_i$ then $e_u$ and $e_v$ connect the same pair of vertices and have the opposite orientation in $T_{m_1,m_2,m_3}^{\pi}$.
\end{enumerate}
We know that edges of multiplicity $2$ of $\overline{T_{m_1,m_2,m_3}^{\pi}}$ determines uniquely the blocks of any $\sigma$ satisfying $(ii)$. Let $\overline{e}=\{e_{i_1},e_{i_2},e_{i_3},e_{i_4}\}$ be the edge of multiplicity $4$ of $\overline{T_{m_1,m_2,m_3}^{\pi}}$. Without loss of generality assume $e_{i_1},e_{i_2}\in E_1$, $e_{i_3}\in E_2$, $e_{i_4}\in E_3$. As this edge of multiplicity $4$ is a loop then any pairing of $\{i_1.i_2,i_3,i_4\}$ satisfies $(ii)$, so the possible pairings $\sigma_i$ restricted to $\{i_1,\dots,i_4\}$ are given in next table,
\begin{center}
\begin{tabular}{| c | c | c |}\hline
$\sigma_i$ & Blocks of $\sigma_i$ & $\gamma_{m_1,m_2,m_3}\vee\sigma_i$ \\ \hline
$\sigma_1$ & $\{i_1,i_2\},\{i_3,i_4\}$ & $\FirstCycleNotation\cup \SecondCycleNotation,\ThirdCycleNotation$ \\
$\sigma_2$ & $\{i_1,i_3\},\{i_2,i_4\}$ & $1_m$ \\
$\sigma_3$ & $\{i_1,i_4\},\{i_2,i_3\}$ & $1_m$ \\ \hline
\end{tabular}
\end{center}
Last column shows that only $\sigma_2$ and $\sigma_3$ satisfies condition $(i)$ and so those are the only pairings satisfying both conditions.
\end{proof}

\begin{lemma}\label{Theorem:EssentialPairingofTpiDU}
Let $\pi\in\cP(m)$ and let $T_{m_1,m_2,m_3}^{\pi}\in \TFUCSet$, then there exist a unique non-crossing pairings $\sigma\in NC_2(m_1,m_2,m_3)$ such that $\gamma_{m_1,m_2,m_3}\sigma=\pi$.
\end{lemma}
\begin{proof}
It suffices to prove that there exist a unique pairings $\sigma\in\cP_2(m)$ satisfying,
\begin{enumerate}
\item $\sigma\vee\gamma_{m_1,m_2,m_3}=1_m$
\item if $\{u,v\}$ is a block of $\sigma$ then $e_u$ and $e_v$ connect the same pair of vertices and have the opposite orientation in $T_{m_1,m_2,m_3}^{\pi}$.
\end{enumerate}
The edges of multiplicity $2$ of $\overline{T_{m_1,m_2,m_3}^{\pi}}$ determines uniquely the blocks of any $\sigma$ satisfying $(ii)$. It remains considering the edge of multiplicity $4$. In this case we have two possibilities.

\medskip\noindent
\textbf{Case 1. The edge of multiplicity $4$ is not in the \cycle{} of $\overline{T_{m_1,m_2,m_3}^{\pi}}$.} In this case any edge in the \cycles is connecting. Suppose without loss of generality any edge in the \cycles is a $(2,3)$-edge. This means the edge of multiplicity $4$ is either a $(1,2)$-edge or a $(1,3)$-edge. Assume it is a $(1,2)$-edge. Let $\overline{e}=\{e_{i_1},e_{i_2},e_{i_3},e_{i_4}\}$ be the edge of multiplicity $4$ with $e_{i_1},e_{i_2}\in E_1$, $e_{i_3},e_{i_4}\in E_2$ and $e_{i_1}$ and $e_{i_3}$ having the same orientation. Any $\sigma$ satisfying $(ii)$ must pair edges in opposite orientations, so the pairings $\sigma_i$ restricted to $\{i_1,\dots,i_4\}$ are given in next table.
\begin{center}
\begin{tabular}{| c | c | c |}\hline
$\sigma_i$ & Blocks of $\sigma_i$ & $\gamma_{m_1,m_2,m_3}\vee\sigma_i$ \\ \hline
$\sigma_1$ & $\{i_1,i_2\},\{i_3,i_4\}$ & $\FirstCycleNotation,\SecondCycleNotation\cup \ThirdCycleNotation$ \\
$\sigma_2$ & $\{i_1,i_4\},\{i_2,i_3\}$ & $1_m$ \\ \hline
\end{tabular}
\end{center}
Last column shows that only $\sigma_1$ satisfies condition $(i)$ and so that is the only pairing.\\
\textbf{Case 1. The edge of multiplicity $4$ is in the \cycle. of $\overline{T_{m_1,m_2,m_3}^{\pi}}$.} In this case the edge of multiplicity $4$ consist of two edges from one basic cycle, say $E_1$, and one edge from each $E_2$ and $E_3$, let $\overline{e}=\{e_{i_1},e_{i_2},e_{i_3},e_{i_4}\}$ be the edge of multiplicity $4$ with $e_{i_1},e_{i_2}\in E_1$, $e_{i_3}\in E_2$, $e_{i_4}\in E_3$ and such that $e_{i_1}$ and $e_{i_3}$ have the same orientation. Any $\sigma$ satisfying $(ii)$ must pair edges in opposite orientations, so the possible pairings $\sigma_i$ restricted to $\{i_1,\dots,i_4\}$ are given in next table,
\begin{center}
\begin{tabular}{| c | c | c |}\hline
$\sigma_i$ & Blocks of $\sigma_i$ & $\gamma_{m_1,m_2,m_3}\vee\sigma_i$ \\ \hline
$\sigma_1$ & $\{i_1,i_2\},\{i_3,i_4\}$ & $\FirstCycleNotation,\SecondCycleNotation\cup \ThirdCycleNotation$ \\
$\sigma_2$ & $\{i_1,i_4\},\{i_2,i_3\}$ & $1_m$ \\ \hline
\end{tabular}
\end{center}
Last column shows that only $\sigma_1$ satisfies condition $(i)$ and so that is the only pairing satisfying both conditions, in any case we get that there exist a unique pairing satisfying $(i)$ and $(ii)$.
\end{proof}

\begin{lemma}\label{Theorem:EssentialPairingofTpiDB}
Let $\pi\in\cP(m)$ and let $T_{m_1,m_2,m_3}^{\pi}\in \DBSet$, then there exist a unique non-crossing pairings $\sigma\in NC_2(m_1,m_2,m_3)$ such that $\gamma_{m_1,m_2,m_3}\sigma=\pi$.
\end{lemma}
\begin{proof}
It suffices to prove that there exist a unique pairings $\sigma\in\cP_2(m)$ satisfying,
\begin{enumerate}
\item $\sigma\vee\gamma_{m_1,m_2,m_3}=1_m$
\item if $\{u,v\}$ is a block of $\sigma$ then $e_u$ and $e_v$ connect the same pair of vertices and have the opposite orientation in $T_{m_1,m_2,m_3}^{\pi}$.
\end{enumerate}
Since $T_{m_1,m_2,m_3}^{\pi}\in \DBSet$ then any edge has multiplicity $2$, so we let $\sigma=\overline{\pi}$, note that any edge of multiplicity $2$ of $\overline{T_{m_1,m_2,m_3}^{\pi}}$ consist of $2$ edges in opposite orientation (Theorem \ref{Theorem:Maybe}), so $\sigma$ satisfies $(ii)$. As we proved in Theorem \ref{Theorem:EssentialPairingofTpi26TreeType} the edges of multiplicity $2$ of $\overline{T_{m_1,m_2,m_3}^{\pi}}$ determines uniquely the blocks of any $\sigma$ satisfying $(ii)$, so in this case such a $\sigma$ satisfying $(ii)$ is unique and is given by $\sigma=\overline{\pi}$. It remains proving $\sigma=\overline{\pi}$ satisfies $(i)$, however, this is an immediate consequence of Theorem \ref{Theorem:Maybe} as $C_{\pi}\neq 0$.
\end{proof}

\begin{corollary}\label{Corollary:PreImageofNonCrossingPairings}
Let $\pi\in\cP(m)$ be such that $T_{m_1,m_2,m_3}^{\pi}\in \VGSet$ then there exist $\sigma\in NC_2(m_1,m_2,m_3)$ such that $\pi=\gamma_{m_1,m_2,m_3}\sigma$, consequently $T_{m_1,m_2,m_3}^{\pi}=T_{m_1,m_2,m_3}^{\gamma_{m_1,m_2,m_3}\sigma}$. Furthermore,
\begin{enumerate}
\item If $T_{m_1,m_2,m_3}^{\pi}\in \TSTTSet\cup \TFULSet$ there exist exactly two non-crossing pairings $\sigma_1,\sigma_2\in NC_2(m_1,m_2,m_3)$ such that $$T_{m_1,m_2,m_3}^{\pi}=T_{m_1,m_2,m_3}^{\gamma_{m_1,m_2,m_3}\sigma_1}=T_{m_1,m_2,m_3}^{\gamma_{m_1,m_2,m_3}\sigma_2}.$$
\item If $T_{m_1,m_2,m_3}^{\pi}\in \TFFTTSet\cup \TFUCSet\cup \DBSet$ there exist a unique non-crossing pairings $\sigma\in NC_2(m_1,m_2,m_3)$ such that $$T_{m_1,m_2,m_3}^{\pi}=T_{m_1,m_2,m_3}^{\gamma_{m_1,m_2,m_3}\sigma}.$$
\end{enumerate}
\end{corollary}

Corollary \ref{Corollary:PreImageofNonCrossingPairings} together with Lemma \ref{Lemma:ImageofNonCrossingPairings} determines the relation between non-crossing pairing and \VG\text{s}. This is given as follows.

\begin{lemma}\label{Lemma:CountingAllVG}
Let $\pi\in\cP(m)$ then $T_{m_1,m_2,m_3}^{\pi}\in \VGSet$ if and only if there exist $\sigma\in NC_2(m_1,m_2,m_3)$ such that $\pi=\gamma_{m_1,m_2,m_3}\sigma$, furthermore, 
\begin{eqnarray*}
|NC_2(m_1,m_2,m_3)|=|\VGSet|+|\TFULSet|+|\TSTTSet|
\end{eqnarray*}
\end{lemma}
\begin{proof}
Corollary \ref{Corollary:PreImageofNonCrossingPairings} and Lemma \ref{Lemma:ImageofNonCrossingPairings} proves that the mapping, $$T:NC_2(m_1,m_2,m_3)\rightarrow \VGSet,$$ given by,
$$T(\sigma)=T_{m_1,m_2,m_3}^{\gamma_{m_1,m_2,m_3}\sigma},$$
is surjective which proves the first part of Lemma. For the second part note that Corollary \ref{Corollary:PreImageofNonCrossingPairings} proves that only if $T_{m_1,m_2,m_3}^{\pi}\in \TFULSet\cup \TSTTSet$ there exist exactly two non-crossing pairings being mapped under $T$ to the same quotient graph $T_{m_1,m_2,m_3}^{\pi}$ which proves the second part.
\end{proof}

\begin{corollary}\label{Corollary:CountingDB}
$$|\DBSet|=|NC_2(m_1,m_2,m_3)|-|\TFUCSet|-|\TFFTTSet|-2|\TFULSet|-2|\TSTTSet|$$
\end{corollary}

Corollary \ref{Corollary:CountingDB} reduces the problem of counting all \VG\text{s} to only counting $\TSTTSet,\TFFTTSet,\TFULSet$ and $\TFUCSet$, so, for the rest of the paper we will work in counting each type of these.

\subsection{Counting double trees and double unicircuit graphs} Our motivation for counting double trees and \DUs graphs is that most of the limit graphs can be expressed in terms of these; in fact double trees and \DUs graphs appear when computing the second order moments as seen in \cite{MMPS}, where they provide a way of counting  certain graphs using the set of non-crossing pairings in an $(m_1, m_2)$-annulus. In this subsection we provide our alternative proof of that result.

\begin{lemma}\label{Corollary:CountingDoubleTrees}
Let $m\in\mathbb{N}$, $\pi\in\cP(m)$ and $\gamma_m=(1,\dots,m)\in S_m$. Let $T_m=(V,E)$ be the graph consisting of a single basic cycle. Let $\pi\in \cP(m)$. $T_m^{\pi}$ is a double tree if and only if $\pi=\gamma_m \sigma$ for some $\sigma \in NC_2(m)$. Moreover if $\sigma_1 \neq \sigma_2$ then $T_m^{\gamma_m\sigma_1}$ and $T_m^{\gamma_m\sigma_2}$ are distinct quotient graphs, consequently,
$$|\{\pi\in \cP(m): T_m^{\pi}\text{ is a double tree}\}|=|NC_2(m)|.$$
\end{lemma}
\begin{proof}
Suppose $T_m^{\pi}=(V,E)$ is a double tree, we let $\sigma=\overline{\pi} \in \cP_2(m)$. Theorem \ref{Theorem:ExtractingPairingProperties} says, $\gamma_m\sigma \leq \pi$, therefore,
$$m+1 = |V|+|E| = \#(\pi)+\#(\sigma) \leq \#(\gamma_m\sigma)+\#(\sigma) \leq m+1$$
which forces $\pi=\gamma_m\sigma$ and $\sigma\in NC_2(m)$. Conversely let $\sigma\in \NC_2(m)$ and $\pi=\gamma_m\sigma$. Note that $G_{\sigma}^{\gamma_m}$ has $\#(\gamma_m\sigma)$ vertices and $\#(\sigma)$ edges and it is connected because $T_m^{\pi}$ is connected. Thus,
$$1 \geq \#(\gamma_m\sigma)-\#(\sigma) =\#(\gamma_m\sigma)+\#(\sigma)-m =1,$$
therefore all above must be equality which means $G_{\sigma}^{\gamma_m}$ is a tree. On the other hand Theorem \ref{Theorem:DoublingGGivesQuotientGraph} says $T(G_{\sigma}^{\gamma_m})=T_m^{\pi}$, therefore $T_m^{\pi}$ is a double tree, moreover if $\{u,v\}$ is a block of $\sigma$ then $e_u$ and $e_v$ are joining the same pair or vertices of $T_m^{\pi}$, i.e. $\sigma=\overline{\pi}$. The latest observation means that for $\sigma_1\neq \sigma_2$ the partitions $\overline{\gamma_m\sigma_1}$ and $\overline{\gamma_m\sigma_2}$ are distinct and therefore the quotient graphs $T_m^{\gamma\sigma_1}$ and $T_m^{\gamma\sigma_2}$ must be different.
\end{proof}

To count the double unicyclic graphs we will use the set of non-crossing pairings on the $(m_1,m_2)$-annulus.

\begin{lemma}\label{Chapter3_Lemma_Counting double unicircuit whose length is k}
Let $m_1,m_2\in\mathbb{N}$, $m=m_1+m_2$, $\pi\in\cP(m)$ and $\gamma_{m_1,m_2}=(1,\dots,m_1)(m_1+1,\dots,m)\in S_m$. Let $T_{m_1,m_2}$ be defined as in Section \ref{Section:GraphTheory}. Let $k\in \mathbb{N}$ with $k\neq 2$. $T^{\pi}_{m_1,m_2}$ is a double unicircuit graph where the unique \cycles of $\overline{T_{m_1,m_2}^{\pi}}$ has length $k$ if and only if there exist $\sigma\in NC_2^{(k)}(m_1,m_2)$ such that $\pi=\gamma_{m_1,m_2}\sigma$. Moreover for $\sigma_1 \neq \sigma_2$ the quotient graphs $T_{m_1,m_2}^{\gamma_{m_1,m_2}\sigma_1}$ and $T_{m_1,m_2}^{\gamma_{m_1,m_2}\sigma_2}$ are distinct.
\end{lemma}
\begin{proof}
Suppose $T^{\pi}_{m_1,m_2}$ is a double unicircuit graph where the unique \cycles of $\overline{T_{m_1,m_2}^{\pi}}$ has length $k$. $\overline{T_{m_1,m_2}^{\pi}}$ is a connected graph with $\#(\pi)$ vertices and $m/2$ edges, therefore $q(\pi)=\#(\pi)-m/2=0$. Let $\sigma\in \cP(m)$ be the partition defined by $u\overset{\sigma}\sim v$ if $e_u$ and $e_v$ connect the same pair of vertices of $T_{m_1,m_2}^{\pi}$ then $\sigma$ is a pairing satisfying $\sigma\vee\gamma_{m_1,m_2}=1_m$ and if $\{u,v\}$ is a block of $\sigma$ then $e_u$ and $e_v$ have the opposite orientation, Theorem \ref{Theorem:CaracterizationOfEssentialSigma} says that $\sigma\in NC_2(m_1,m_2)$ and $\pi=\gamma_{m_1,m_2}\sigma$. Moreover, by definition of $\sigma$ the through strings of $\sigma$ correspond to the edges in the \cycles of $\overline{T_{m_1,m_2}^{\pi}}$, so $\sigma\in NC_2^{(k)}(m_1,m_2)$. \\
Conversely, let $\sigma\in NC_2^{(k)}(m_1,m_2)$ and $\pi=\gamma_{m_1,m_2}\sigma$. Let $(V,E)$ be the graph $G_{\sigma}^{\gamma_{m_1,m_2}}$, Lemma \ref{Lemma:PropertiesofTsigmagamma} says $\overline{T^{\gamma_{m_1,m_2}\sigma}_{m_1,m_2}}$ is a connected graph and so is $G_{\sigma}^{\gamma_{m_1,m_2}}$. Note,
\begin{eqnarray*}
|V|-|E| &=& \#(\gamma_{m_1,m_2}\sigma)-\#(\sigma)= \#(\gamma_{m_1,m_2}\sigma)+\#(\sigma)-m=0,
\end{eqnarray*}
thus,  $G_{\sigma}^{\gamma_{m_1,m_2}}$ is a uni\cycles graph. Let $B=\{u,v\}$ be a block of $\sigma$ not in the \cycles of $G_{\sigma}^{\gamma_{m_1,m_2}}$, then $B$ is a cutting edge of $G_{\sigma}^{\gamma_{m_1,m_2}}$ and so $\overline{e_u}=\{e_u,e_v\}$ is a cutting edge of $\overline{T_{m_1,m_2}^{\gamma_{m_1,m_2}\sigma}}$, that forces to $e_u$ and $e_v$ being from the same basic cycle, and so $B$ is a non-through string of $\sigma$, that means that all through strings $B_1,\dots,B_k$ of $\sigma$ must correspond to edges in the \cycles of $G_{\sigma}^{\gamma_{m_1,m_2}}$. Suppose one of the edges in the \cycles is not a through string, then there is a vertex of $G_{\sigma}^{\gamma_{m_1,m_2}}$, $V$, in its unique \cycles whose two adjacent edges in within the \cycles correspond to one non-through string and one through string. When doubling the edges of $G_{\sigma}^{\gamma_{m_1,m_2}}$ we obtain that $V$ will be adjacent to an odd number of edges from $E_1$ (the trough string produces $1$ adjacent edge from $E_1$ and any non-trough string produces either $0$ or $2$ adjacent edges from $E_1$), this is not possible as $\deg(V)_1$ must be even, so all edges in the \cycles of $G_{\sigma}^{\gamma_{m_1,m_2}}$ must correspond to through strings of $\sigma$, i.e $G_{\sigma}^{\gamma_{m_1,m_2}}$ is a graph with a single \cycles of length $k$. Remember that doubling the edges of $G_{\sigma}^{\gamma_{m_1,m_2}}$ produces the graph $T^{\gamma_{m_1,m_2}\sigma}_{m_1,m_2}$ and since the there are no edges of $G_{\sigma}^{\gamma_{m_1,m_2}}$ connecting the same pair of vertices (that would mean having a \cycles of length $2$ and we assumed $k\neq 2$), then $T^{\gamma_{m_1,m_2}\sigma}_{m_1,m_2}$ is a graph such that $\overline{T^{\gamma_{m_1,m_2}\sigma}_{m_1,m_2}}$ has a unique \cycles and all its edges have multiplicity $2$ and consist of two edges in opposite orientation. Moreover, it must be that $\sigma=\overline{\pi}$ because $\sigma \leq \overline{\pi}$ (Theorem \ref{Lemma:PropertiesofTsigmagamma}) and both have only blocks of size $2$. To verify $T_{m_1,m_2}^{\pi}$ is a double uni\cycles graph it remains to prove that all edges in the \cycles of $\overline{T_{m_1,m_2}^{\pi}}$ correspond to non-through strings, this follows because all edges in the \cycles of $\overline{T_{m_1,m_2}^{\pi}}$ correspond to edges in the \cycles of $G_{\sigma}^{\gamma_{m_1,m_2}}$ and $\sigma=\overline{\pi}$. Finally the condition $\sigma=\overline{\pi}$ proves that if $\sigma_1\neq \sigma_2$ then $T_{m_1,m_2}^{\gamma_{m_1,m_2}\sigma_1}$ and $T_{m_1,m_2}^{\gamma_{m_1,m_2}\sigma_2}$ are distinct.
\end{proof}

The following is a consequence of Lemma \ref{Chapter3_Lemma_Counting double unicircuit whose length is k}.

\begin{corollary}\label{Corollary:CountingDU}
Let $k\in\mathbb{N}$ with $k\neq 2$. The following are satisfied.
\begin{enumerate}
\item 
\begin{multline*}
|NC_2^{(k)}(m_1,m_2)|= |\{\pi\in\cP(m): T_{m_1,m_2}^{\pi}\text{ is a double unicircuit graph} \\
\text{ with } \overline{T_{m_1,m_2}^{\pi}}\text{ having a \cycles of length }k\}|
\end{multline*}
\item 
\begin{multline*}
|NC_2(m_1,m_2)\setminus NC_2^{(2)}(m_1,m_2)| =
|NC_2(m_1,m_2)|-|NC_2^{(2)}(m_1,m_2)| \\=|\{\pi\in\cP(m): T_{m_1,m_2}^{\pi}\text{ is a double unicircuit graph}\}|
\end{multline*}
\end{enumerate}
\end{corollary}

\subsection{Counting $2$-$6$ and $2$-$4$-$4$-tree types} Counting $2$-$6$ and $2$-$4$-$4$ tree types is an immediate consequence of counting double trees. For this section we set $m_1,m_2,m_3\in\mathbb{N}$ and $m=m_1+m_2+m_3$. Let us introduce the following set of partitioned permutations.

\begin{notation}
For a partitioned permutation $$(\cV,\pi)\in \PSG \cup \PSSG,$$ 
we can write $\pi$ as $\pi=\pi_1\times \pi_2\times \pi_3$ with $\pi_i\in NC(m_i)$.
\begin{enumerate}
\item We denote by $\PS$ to the set of partitioned permutations $(\cV,\pi)\in \PSG$ such that $\pi_i\in NC_2(m_i)$ $\forall i=1,2,3$.
\item We denote by $\PSS$ to the set of partitioned permutations $(\cV,\pi)\in \PSSG$ such that $\pi_i\in NC_2(m_i)$ $\forall i=1,2,3$.
\end{enumerate}
\end{notation}

\begin{remark}\label{Remark:Counting PS and PSS}
The cardinality of $\PS$ can be computed easily. Each element $(\cV,\pi)\in \PS$ is such that $\pi=\pi_1\times \pi_2\times \pi_3$ with $\pi_i\in NC_2(m_i)$ and each block of $\cV$ is a cycle of $\pi$ except by one block which is the union of three cycles of $\pi$, thus the number of permutations $\pi$, can be counted by $$|NC_2(m_1)||NC_2(m_2)||NC_2(m_3)|,$$
while the number of partitions $\cV$, can be counted by choosing a cycle from each $\pi_i$, which can be done in $\frac{m_1}{2}\frac{m_2}{2}\frac{m_3}{2}$ ways. Therefore,
$$|\PS|=\frac{m_1}{2}\frac{m_2}{2}\frac{m_3}{2}|NC_2(m_1)||NC_2(m_2)||NC_2(m_3)|.$$
In a similar manner we can compute the cardinality of $\mathcal{PS}_{NC}^{(2,1,1)}(m_1,m_2, \ab m_3)$, namely,
\begin{multline*}
|\PSS| \\
=\frac{m_1}{2}\frac{m_2}{2}\frac{m_3}{2}(\frac{m}{2}-3)|NC_2(m_1)||NC_2(m_2)|
|NC_2(m_3)|.
\end{multline*}
\end{remark}

\begin{lemma}\label{Lemma:Counting2-6TT}
$$|\TSTTSet|=4|\PS|$$
\end{lemma}
\begin{proof}
For $\pi\in\cP(m)$ the quotient graph $T_{m_1,m_2,m_3}^{\pi}$ is a $2$-$6$ tree type if all three graphs $T_{m_1}^{\pi}$,$T_{m_2}^{\pi}$ and $T_{m_3}^{\pi}$ are double trees, Lemma \ref{Corollary:CountingDoubleTrees} says that each double tree can be chosen in $|NC_2(m_1)|$,$|NC_2(m_2)|$ and $|NC_2(m_3)|$ distinct ways respectively. For each choice of double trees we choose an edge from each one, which can be done in $\frac{m_1}{2}\frac{m_2}{2}\frac{m_3}{2}$ ways. Finally we can join the graph along those edges in $4$ different ways depending on the orientation, as each orientation correspond to a different partition $\pi$, then,
$$|\TSTTSet|=4\frac{m_1}{2}\frac{m_2}{2}\frac{m_3}{2} |NC_2(m_1)||NC_2(m_2)||NC_2(m_3)|.$$
As said in Remark \ref{Remark:Counting PS and PSS} last expression is precisely fourth times the cardinality of $\PS$.
\end{proof}

\begin{lemma}\label{Lemma:Counting2-4-4TT}
$$|\TFFTTSet|=4|\PSS|$$
\end{lemma}
\begin{proof}
We proceed as in Lemma \ref{Lemma:Counting2-6TT}. The double trees can be chosen in $|NC_2(m_1)||NC_2(m_2)||NC_2(m_3)|$ ways. For a choice of the double trees we choose two edges from one of them and one edge from each of the other two. Suppose we chose two edges from $\overline{T_{m_1}^{\pi}}$. Those can be chosen in $\frac{1}{2}\frac{m_1}{2}(\frac{m_1}{2}-1)$ ways. The edges of the other two double trees can be chosen in $\frac{m_2}{2}\frac{m_3}{2}$ ways. Once we selected the edges we have to make the union along these edges, we can pair them in $2$ different ways and for each pair there are two possible orientations of the edges, giving a total of $4$ ways. So the total number of ways is given by,
$$8\frac{1}{2}\frac{m_1}{2}(\frac{m_1}{2}-1)\frac{m_2}{2}\frac{m_3}{2}=4\frac{m_1}{2}(\frac{m_1}{2}-1)\frac{m_2}{2}\frac{m_3}{2}$$
Similarly when choosing two edges from $\overline{T_{m_2}^{\pi}}$ and $\overline{T_{m_3}^{\pi}}$ we get $4\frac{m_2}{2}(\frac{m_2}{2}-1)\frac{m_1}{2}\frac{m_3}{2}$ and $4\frac{m_3}{2}(\frac{m_3}{2}-1)\frac{m_1}{2}\frac{m_2}{2}$ ways respectively, so the total number of $2$-$4$-$4$ tree types is given by,
\begin{multline*}
4\left[ \frac{m_1}{2}(\frac{m_1}{2}-1)\frac{m_2}{2}\frac{m_3}{2}+\frac{m_2}{2}(\frac{m_2}{2}-1)\frac{m_1}{2}\frac{m_3}{2}+\frac{m_3}{2}(\frac{m_3}{2}-1)\frac{m_1}{2}\frac{m_2}{2}\right] \\
|NC_2(m_1)||NC_2(m_2)||NC_2(m_3)| \\
=4\frac{m_1}{2}\frac{m_2}{2}\frac{m_3}{2}(\frac{m}{2}-3)|NC_2(m_1)||NC_2(m_2)||NC_2(m_3)|.
\end{multline*}
This last expression is fourth times the cardinality of $\PSS$.
\end{proof}

\subsection{Counting $2$-$4$ uniloop and $2$-$4$ unicircuit types}$2$-$4$-uniloop and \TFUC s are made of \DUs and double tree graphs, this allow us to count them in a simple manner. For this section we set $m_1,m_2,m_3\in\mathbb{N}$ and $m=m_1+m_2+m_3$ unless other is specified. Let us introduce the following sets of partitioned permutations.

\begin{notation}
For a partitioned permutation $$(\cV,\pi)\in \PSG,$$ we write $\pi=\pi_1\times\pi_2\times\pi_3$ with $\pi_i\in NC(m_i)$ and each block of $\cV$ is a cycle of $\pi$ except by one block which is the union of three cycles of $\pi$ one from each permutation $\pi_i$. We denote by $\PSSS$ to the set of partitioned permutations $(\cV,\pi)\in \PSG$ satisfying the following conditions,
\begin{enumerate}
\item $\pi_{i_1}$ and $\pi_{i_2}$ have all cycles of size $2$ except by one cycle of size $1$ while $\pi_{i_3}\in NC_2(m_{i_3})$ with $(i_1,i_2,i_3)$ a permutation of $(1,2,3)$.
\item The block of $\cV$ which is the union of three cycles of $\pi$ consists of the cycles of size $1$ from $\pi_{i_1}$ and $\pi_{i_2}$ and any cycle from $\pi_{i_3}$.
\end{enumerate}
An example can be seen in Figure \ref{Figure:PSG211}
\end{notation}

\begin{figure}
\begin{center}
\includegraphics[width=200pt,height=200pt]{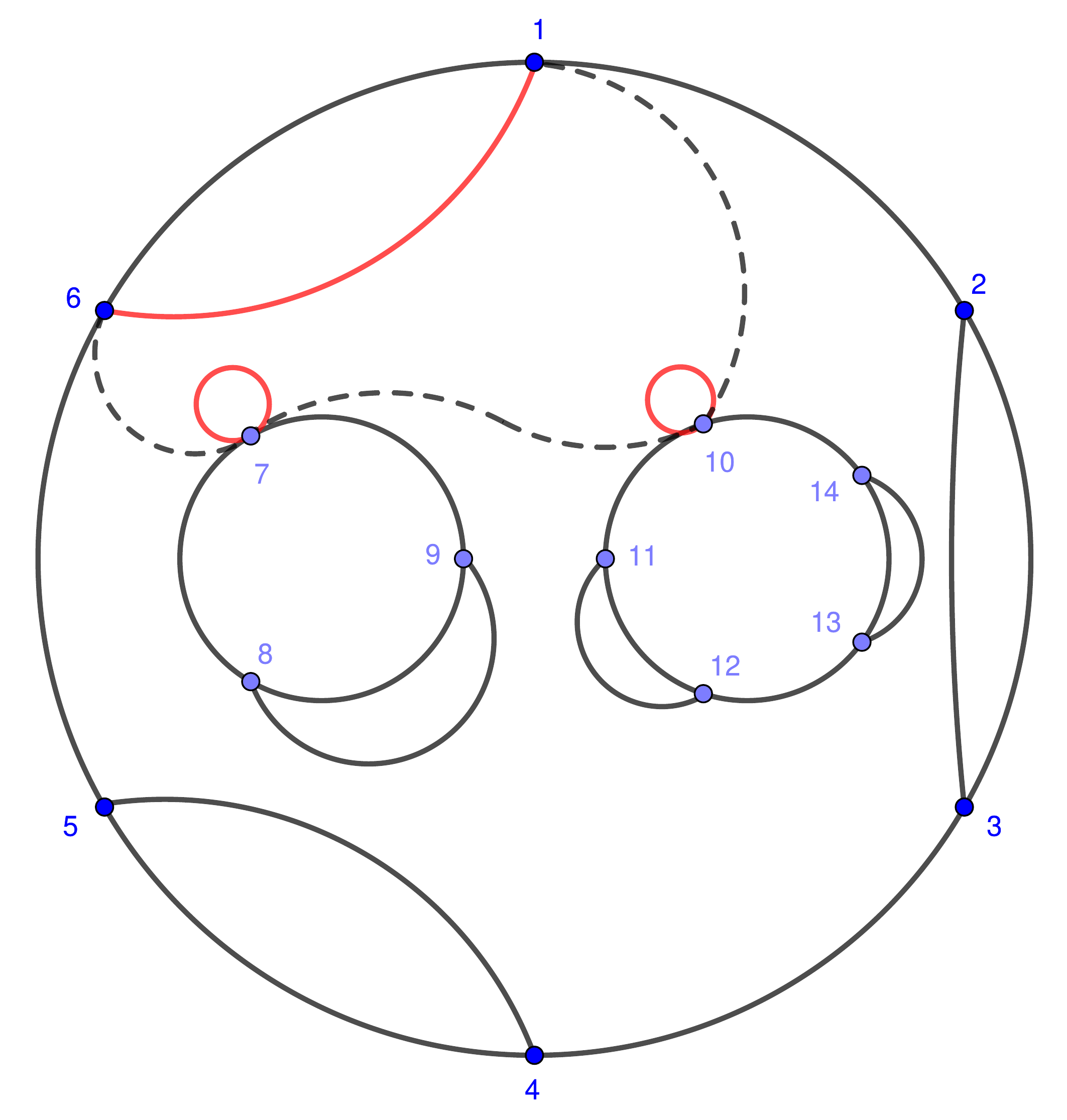} 
\end{center}
\caption{\small\label{Figure:PSG211} A partitioned permutation $(\cV,\pi)$  in the set $\PSSS$ corresponding to \begin{center}
$\pi=(1,6)(2,3)(4,5)(7)(8,9)(10)(11,12)(13,14)$
\end{center} \begin{center}
$\cV=\{\{1,6,7,10\},\{2,3\},\{4,5\},\{8,9\},\{11,12\},\{13,14\}\}.$
\end{center}  Each block of $\cV$ is a cycle of $\pi$ except by one block which is the union of the two cycles of size $1$ of $\pi$ and one cycle of size $2$.}
\end{figure}

\begin{notation}
For a partitioned permutation $$(\cV,\pi)\in \PSSSG,$$ we write $\pi=\pi_1\times\pi_2$ with $\pi_1\in S_{NC}(m_{i_1},m_{i_2})$ and $\pi_2\in NC(m_{i_3})$ for some permutation $(i_1,i_2,i_3)$ of $(1,2,3)$. Each block of $\cV$ is a cycle of $\pi$ except by one block which is the union of two cycles of $\pi$ one from each permutation $\pi_i$.
\begin{enumerate}
\item We denote by $\PSSSSS$ to the set of partitioned permutations $(\cV,\pi)\in \PSSSG$ such that $\pi_1\in NC_2(m_{i_1},m_{i_2})$ and $\pi_3\in NC_2(m_{i_3})$,
\item We denote by $\PSSSSSS$ to the set of partitioned permutations $(\cV,\pi)\in \PSSSSS$ such that $\pi_1$ has exactly two through strings, i.e $\pi_1\in NC_2^{(2)}(m_{i_1},m_{i_2})$.
\item We denote by $\PSSSS$ to the set of partitioned permutations $(\cV,\pi)\in \PSSSSS$ such that $\pi_1$ has exactly one through string, (i.e $\pi_1\in NC_2^{(1)}(m_{i_1},m_{i_2})$) and the block of $\cV$ which is the union of two cycles of $\pi$ consist of the unique through string of $\pi_1$ and any block of $\pi_2$.
\end{enumerate}
Some examples can be seen in Figures \ref{Figure:Pairing partitioned permutations-a}, \ref{Figure:Pairing partitioned permutations-b}, and  \ref{Figure:Pairing partitioned permutations-c}.
\end{notation}

\begin{figure}
\begin{center}
\includegraphics[width=160pt,height=160pt]{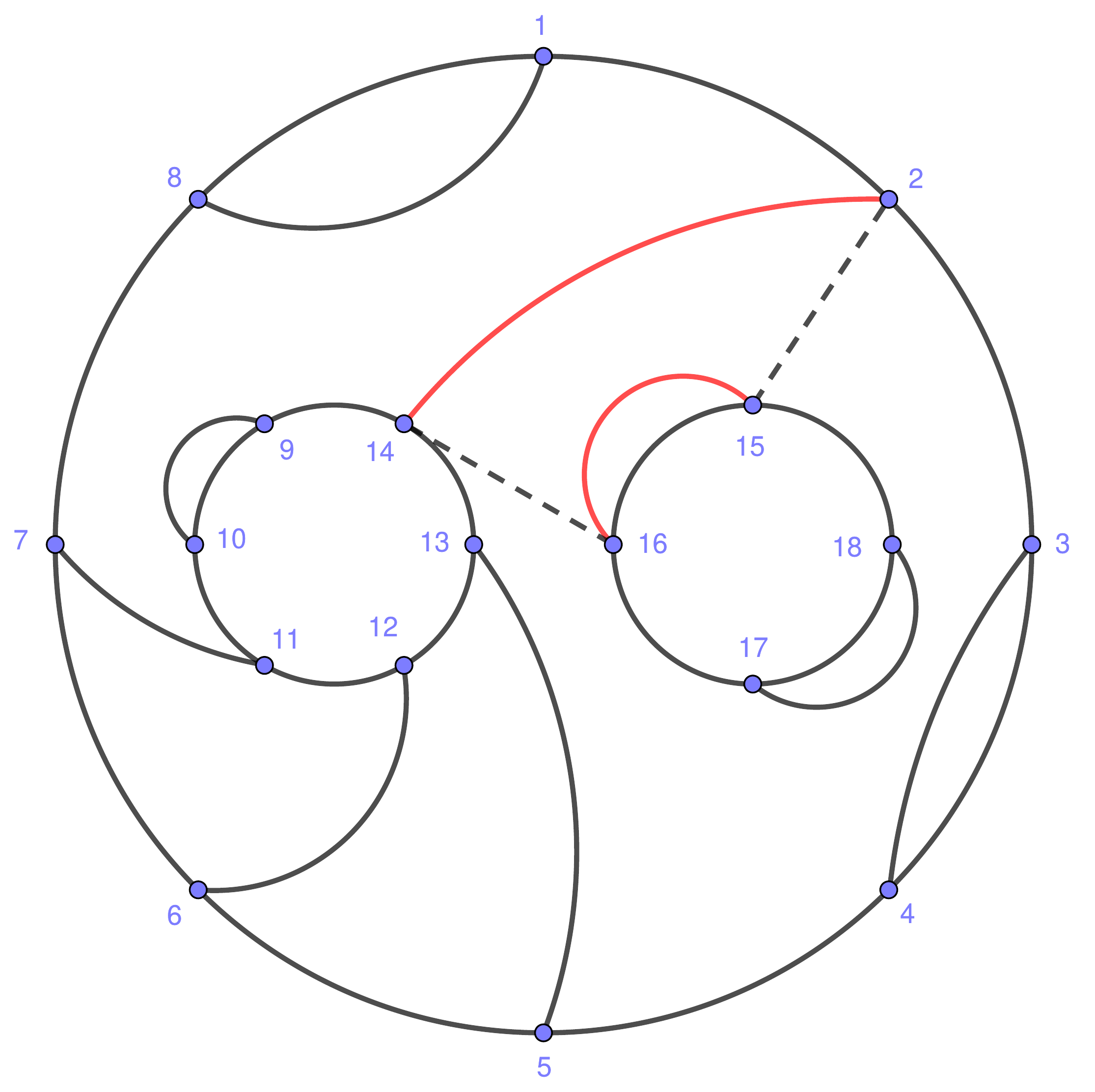}
\end{center}
\caption{\small\label{Figure:Pairing partitioned permutations-a}
$(\cV,\pi)$ in  $\cPS_{NC}^{(1,1)}(9,5,4)$ where  $\pi= 
(1,8)\ab(2,14)(3,4)(5,13)
(6,12)(7, 11)(9,10)(15,16)(17,18)$ and
$\cV=\{\{1,8\},\{2,14,15,16\},\{3,4\},\ab\{5,13\},\{6,12\},\{7,11\},\{9,10\},\{17,18\}\}$. 
}
\end{figure}

\begin{definition}
Let $n\in \mathbb{N}$. We define the set of \textit{block-pairings} on $[n]$, which we denote by $NC_2^{block}(n)$, as the set of pairs $(B,\sigma)$ where $\sigma \in NC_2(n)$ and $B$ is a block of $\sigma$.
\end{definition}

\begin{remark}
The cardinality of $NC_2^{block}(n)$ is given by $\frac{n}{2}|NC_2(n)|$ as for each non-crossing pairing there are $\frac{n}{2}$ blocks to be chosen.
\end{remark}

\begin{definition}
Let $n\in\mathbb{N}$ and $T_n$ be the graph consisting of a unique basic cycle defined as in Section \ref{Section:GraphTheory}. For a partition $\pi\in \cP(n)$, we say that the graph $T_{n}^{\pi}$ is a \textit{double uniloop graph} if $T_{n}^{\pi}$ consists of a graph with two loops over the same vertex and such that removing those loops results in a double tree. We denote the set of double uniloop graphs by $\mathcal{DUL}(n)$.
\end{definition}

For a block-pairing, $(B,\sigma) \in NC_2^{block}(n)$, we may think of the block $B=\{u,v\}$ as a transposition of $S_n$ whose cycle decomposition is $(u,v)$. Under this interpretation let us define the function $\Psi : NC_2^{block}(n) \rightarrow \mathcal{Q}(T_n)$ given by,
$$\Psi(B,\sigma) = T_{n}^{\gamma_n\sigma B},$$
where $\mathcal{Q}(T_n)$ denotes the set of quotient graphs of $T_{n}$ and $\gamma_n=(1,\dots,n)\ab\in S_n$.

\begin{lemma}\label{Lemma:Counting2LoopType}
Let $n\in\mathbb{N}$ and $\Psi : \NC_2^{block}(n) \rightarrow \mathcal{Q}(T_n)$ be defined as above. The image, $Im(\Psi)$, equals to the set of double uniloop graphs; $\mathcal{DUL}(n)$. Moreover, the function,
$$\Psi : \NC_2^{block}(n) \rightarrow \mathcal{DUL}(n)$$
is injective. Therefore,
\begin{eqnarray*}
|\{\pi\in\cP(n): T_n^{\pi} \text{ is a double uniloop graph}\}| & = & |NC_2^{block}(n)| \nonumber \\
& = & \frac{n}{2}|\NC_2(n)|.
\end{eqnarray*}
\end{lemma}
\begin{proof}
Let $(B,\sigma)\in \NC_2^{set}(n)$ and let $\pi^\prime=\gamma_n\sigma$ with $\gamma_n=(1,\dots,n)\in S_n$. Since $\sigma\in \NC_2(n)$ Lemma \ref{Corollary:CountingDoubleTrees} says that $T_n^{\pi^\prime}$ is a double tree and $\overline{\pi^\prime}=\sigma$. In other words, any block, $\{u,v\}$ of $\sigma$, corresponds to the edges $e_u$ and $e_v$ of $T_m^{\pi^\prime}$ connecting the same pair of vertices and with opposite orientation. We let $B=\{u^\prime,v^\prime\}$. Observe that $\pi^\prime$ has two cycles: $A$ and $B$ (which we regard as blocks and vertices of $T_m^{\pi^\prime}$), such that $u^\prime,\gamma_n(v^\prime) \in A$ and $v^\prime,\gamma_n(u^\prime)\in B$. Therefore, $\pi=\pi^\prime B$ have exactly the same cycles of $\pi^\prime$ except by one which is obtained by the union of $A$ and $B$. This means that the quotient graph $T_{n}^{\pi}$ is obtained by identifying the vertices $A$ and $B$ of $T_n^{\pi^\prime}$ which results in $T_{n}^{\pi}$ being a double uniloop graph. This proves $Im(\Psi) \subset \mathcal{DUL}$. 

Conversely, let $T_{n}^{\pi}$ be a graph of double uniloop type. Then all edges of $\overline{T_{n}^{\pi}}$ have multiplicity $2$. We let $\sigma$ be the pairing obtained by $u\overset{\sigma}\sim v$ if $e_u$ and $e_v$ connect the same pair of vertices. Let $\tau = \sigma (u^\prime, v^\prime)$ where $e_{u^\prime}$ and $e_{v^\prime}$ correspond to the loops of $T_{n}^{\pi}$. Then $\#(\tau) = n/2 + 1$. By Theorem \ref{Theorem:ExtractingPairingProperties}, we have that, as partitions, $\gamma_n \sigma \leq \pi$; so $\#(\gamma_n \sigma) \geq \#(\pi) = n/2$. By \cite[2.10]{MN}, we have that $\#(\tau) + \#(\gamma_n \tau) + \#(\gamma_n) \leq n + 2$; so $\#(\gamma_n \tau) \leq n/2$. Thus 
\[
\#(\gamma_n \tau) \leq n/2 \leq \#(\gamma_n \sigma).
\] 
If $u^\prime \sim_{\gamma_n \sigma} v^\prime$, then $\#(\gamma_n \tau) = \#(\gamma_n \sigma (u^\prime, v^\prime)) = \#(\gamma_n \sigma) + 1$, which is impossible. Thus $u^\prime \not\sim_{\gamma_n \sigma} v^\prime$ and hence $\gamma_n \sigma < \pi$. This implies that $\#(\gamma_n \sigma ) \geq n/2 + 1$. Since $\sigma$ is a pairing we must have $\#(\gamma_n \sigma ) = n/2 + 1$, and thus $\sigma\in \NC_2(n)$, i.e. $(\{u^\prime,v^\prime\},\sigma)\in \NC_2^{block}(n)$. Moreover, note that $\#(\pi)= n/2 $ and $\#(\gamma_n\sigma) = n/2 + 1$, since $\gamma_n \sigma \leq \pi$ then $\pi$ is obtained by joining two blocks of $\gamma_n \sigma$, these blocks must correspond to $[u^\prime]_{\gamma_n \sigma}$ and $[v^\prime]_{\gamma_n \sigma}$ as we proved that $u^\prime \not\sim_{\gamma_n \sigma} v^\prime$. On the other hand, the cycles of the permutation $\gamma_n \sigma (u^\prime,v^\prime)$ are the same as the cycles of $\gamma_n \sigma$ except by one which is the union of two cycles of $\gamma_n \sigma$: $[u^\prime]_{\gamma_n \sigma}$ and $[v^\prime]_{\gamma_n \sigma}$. This proves that, as partitions, $\gamma_n \sigma (u^\prime,v^\prime) = \pi$ which proves that $\Psi(\{u^\prime,v^\prime\},\sigma) = T_n^\pi$, thus $Im(\Psi)=\mathcal{DUL}(n)$.

To verify injectivity, let us remember that if $(B,\sigma)\in \NC_2^{block}(n)$ and $\pi=\gamma_n \sigma B$ then $T_n^{\pi}$ pair the edges $e_u$ and $e_v$ whenever $\{u,v\}$ is a block of $\sigma$, i.e. $\overline{\pi}=\sigma$. Moreover, the block $B=\{u^\prime,v^\prime\}$ correspond to the loops $e_{u^\prime}$ and $e_{v^\prime}$ of $T_n^{\pi}$. Let $(B_1,\sigma_1)$ and $(B_2,\sigma_2)$ be such that, as partitions, $\pi_1 = \gamma_n \sigma_1 B_1$ and $\pi_2 = \gamma_n \sigma_2 B_2$ are the same. By the pointed before, $\sigma_1=\overline{\pi_1}=\overline{\pi_2} =\sigma_2$. Finally observe that since $T_n^{\pi_1}$ and $T_n^{\pi_2}$ are the same then, their corresponding loops are the same, which means that the blocks $B_1$ and $B_2$ must be the same. 
\end{proof}

\begin{figure}[t]
\begin{minipage}[b]{\textwidth}\centering
\includegraphics[width=160pt,height=160pt]{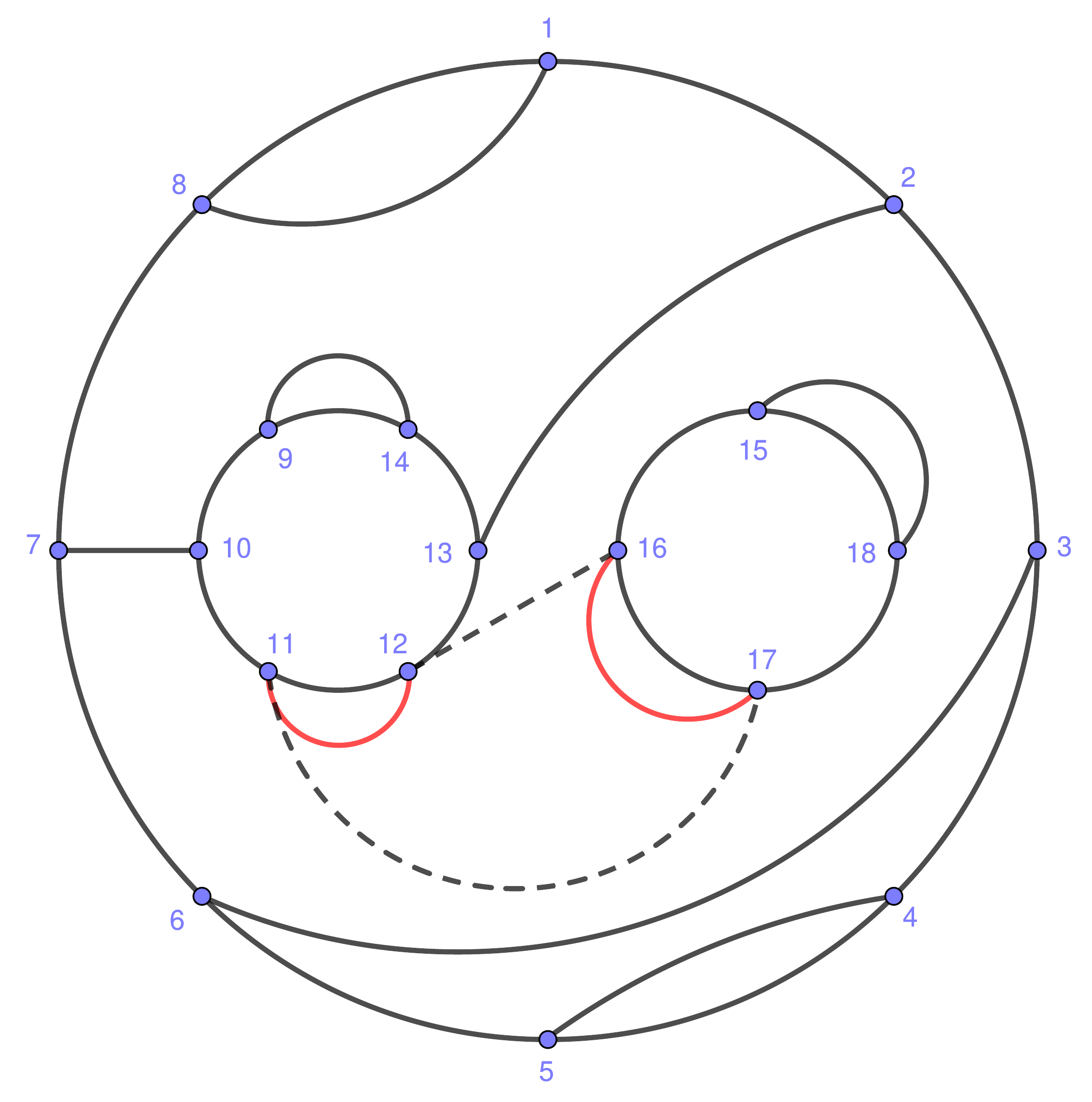} 
\caption
{\protect\raggedright\small%
\label{Figure:Pairing partitioned permutations-b}%
$(\cV,\pi)\in \cPS_\NC^{(1,1)}(9,5,4)$ with $\pi=(1,8)(2,13)\ab(3,6)(4,5)(7,10)(9,14)(11,12)(15,18)(16,17)$ and $\cV=\{\{1,8\},\{2,13\},\{3,6\},\{4,5\},\{7,10\},\{9,14\},\{11,12,16,17\},\ab\{15,18\}\}.$ 
The permutation $\pi$ can be written as $\pi_1\times\pi_2$ with $\pi_1=(1,8)(2,13)(3,6)(4,5)(7,10)(9,14)(11,12)$ and $\pi_2 = (15,18)(16,17)$, the permutation $\pi_1$ has only the two through strings $(2,13)$ and $(7,10)$.}
\end{minipage}\end{figure}

\begin{figure}[t]
\begin{minipage}[b]{\textwidth}\centering
\includegraphics[width=160pt,height=160pt]{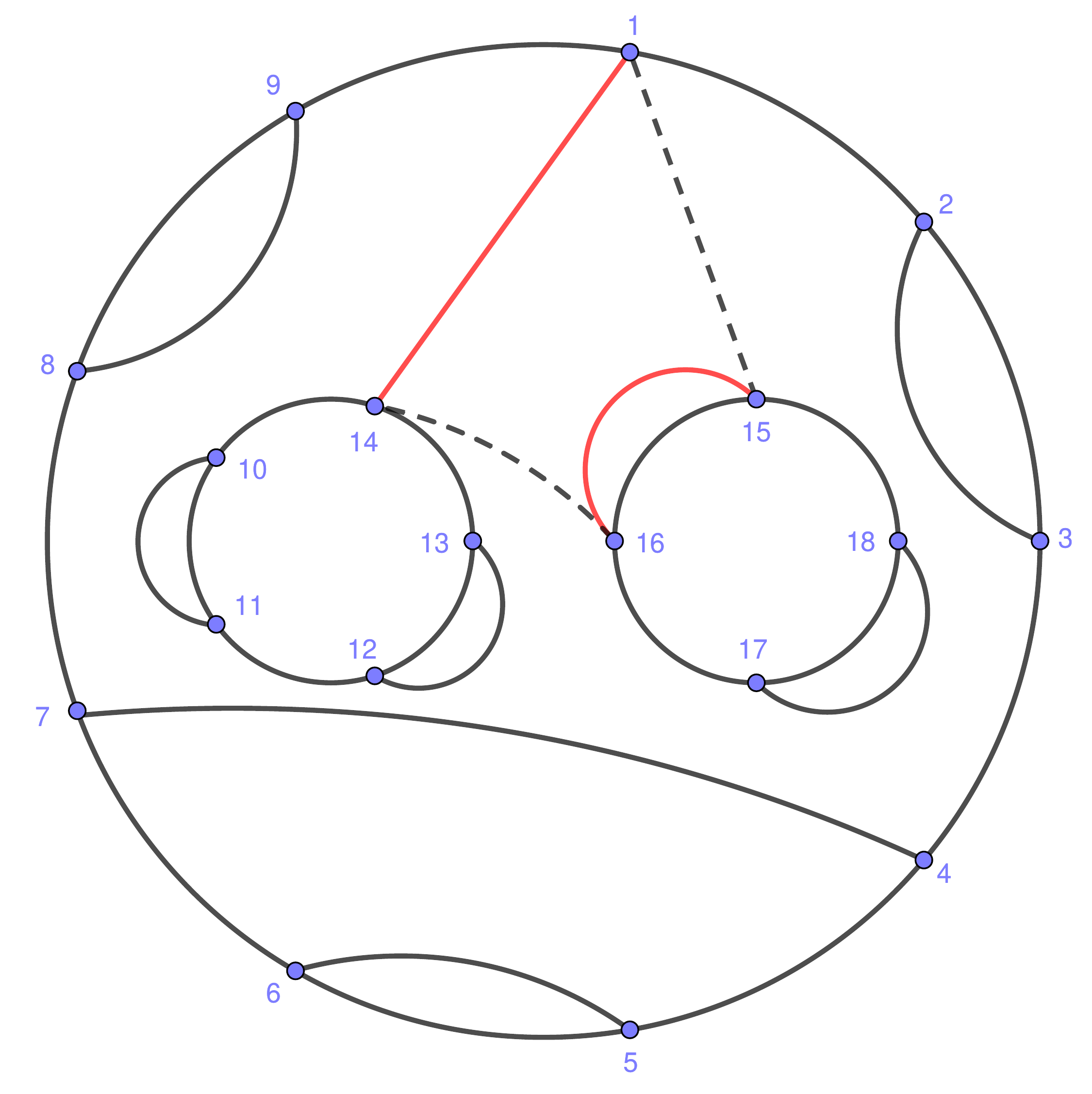}
\caption{\protect\raggedright\small%
\label{Figure:Pairing partitioned permutations-c}%
$(\cV,\pi)\in \cPS_\NC^{(1,1)(t)}(9, 5, 4)$ given by $\pi=(1,14)\ab(2,3)(4,7)(5,6)(8,9)(10,11)(12,13)(15,16)(17,18)$ and $\cV=\{\{1,14,15,16\},\{2,3\},\{4,7\},\{5,6\},\{8,9\},\{10,11\},\{12,13\},\ab\{17,18\}\}$. We have $\pi=\pi_1\times\pi_2\in NC_2(9,5)\times NC_2(4)$ with
$\pi_1$ equal to $\{(1,14)(2,3)(4,7)(5,6)(8,9)(10,11)(12,13)\}$ and $\pi_2=\ab(15,16)(17,18)$. Note that $\pi_1$ has the unique through string, $(1,14)$. The block $\{1,14,15,16\}$ of $\cV$ is obtained by joining the cycle $(1,14)$ of $\pi_1$ to the cycle $(15,16)$ of $\pi_2$.}
\end{minipage}
\end{figure}

\begin{lemma}\label{Lemma:Counting2-4-UL}
$$|\TFULSet|=|\PSSS|=|\PSSSS|$$
\end{lemma}
\begin{proof}
A quotient graph $T_{m_1,m_2,m_3}^{\pi}$ can be of $2$-$4$ uniloop type only when one of $m_1,m_2$ or $m_3$ is even and the other two are odd, similarly $\PSSS$ and $\PSSSS$ are non-empty only in that case, so we may assume $m_1$ is even and $m_2,m_3$ are both odd. The graph $T_{m_1}^{\pi}$ is of double uniloop type and can be chosen in $\frac{m_1}{2}|NC_2(m_1)|$ ways as seen in Lemma \ref{Lemma:Counting2LoopType}. The loop of $T_{m_2}^{\pi}$ can be chosen in $m_2$ ways. Then we quotient the rest to get a double tree. This is basically doing the quotient of a basic cycle of length $m_2-1$ which produces $|NC_2(m_2-1)|$ distinct double trees, so $T_{m_2}^{\pi}$ is chosen in $m_2|NC_2(m_2-1)|$ ways. Similarly $T_{m_3}^{\pi}$ is chosen in $m_3|NC_2(m_3-1)|$ ways, and since the union of the graphs $T_{m_1}^{\pi}$, $T_{m_2}^{\pi}$ and $T_{m_2}^{\pi}$ is already determined by the loop then the number of ways of choosing the graph $T_{m_1,m_2,m_3}^{\pi}$ is,
$$\frac{m_1m_2m_3}{2}|NC_2(m_1)||NC_2(m_2-1)||NC_2(m_3-1)|$$
which is the cardinality of $\PSSS$. The cardinality of $\PSSSS$ can be computed easily. We choose a non-crossing pairing on the $(m_2,m_3)$-annulus with a single through string, By \cite[Lemma 13]{MSAW} this can be done in $\binom{m_2}{m_2-1/2}\binom{m_3}{m_3-1/2}$ ways. The last expression can be rewritten as,
$$m_2m_3|NC_2(m_2-1)||NC_2(m_3-1)|,$$
by using the well know property $|NC_2(n)|=Cat_{n/2} \vcentcolon = \frac{1}{n/2+1}\binom{n}{n/2}$. On the other hand, the non-crossing pairing on $[m_1]$ points can be chosen in $|NC_2(m_1)|$ ways, and we select a block of that pairing which can be chosen in $\frac{m_1}{2}$ ways. Therefore the cardinality of $\PSSSS$ is,
$$\frac{m_1m_2m_3}{2}|NC_2(m_1)||NC_2(m_2-1)||NC_2(m_3-1)|,$$
as required.
\end{proof}

\begin{lemma}\label{Lemma:Counting2-4-U}
\begin{multline*}
|\TFUCSet|= 2|\PSSSSS|-2|\PSSSSSS| \\
-2|\PSSSS|
\end{multline*}
\end{lemma}
\begin{proof}
We will count all graphs $T_{m_1,m_2,m_3}^{\pi}\in \TFUCSet$. If $T_{m_1,m_2,m_3}^{\pi}\in \TFUCSet$ then one of the graphs, $T_{m_{i_1}}^{\pi}$, is a double tree, $T_{m_{i_2},m_{i_3}}^{\pi}$ is a double unicircuit graph with $(i_1,i_2,i_3)$ a permutation of $(1,2,3)$ and $T_{m_1,m_2,m_3}^{\pi}$ results of joining $T_{m_{i_1}}^{\pi}$ and $T_{m_{i_2},m_{i_3}}^{\pi}$ along some edge.

We count firstly the case where the unique \cycles of $\overline{T_{m_{i_2},m_{i_3}}^{\pi}}$ is not a loop. Lemma \ref{Corollary:CountingDoubleTrees} says that $T_{m_{i_1}}^{\pi}$ can be chosen in $|\NC_2(m_{i_1})|$ ways. Similarly, Corollary \ref{Corollary:CountingDU} says that $T_{m_{i_2},m_{i_3}}^{\pi}$ can be chosen in 
\[
|\NC_2(m_{i_2},m_{i_3})|\ab -|\NC_2^{(1)}(m_{i_2},m_{i_3})|-|\NC_2^{(2)}(m_{i_2},m_{i_3})|
\] 
ways. Then we choose an edge from each graph $T_{m_{i_1}}^{\pi}$ and $T_{m_{i_2},m_{i_3}}^{\pi}$. In the first graph there are $\frac{m_{i_1}}{2}$ choices, and in the second there are $\frac{m_{i_2}+m_{i_3}}{2}$ choices. Once selected the edges we make the union of the graphs $T_{m_{i_1}}^{\pi}$ and $T_{m_{i_2},m_{i_3}}^{\pi}$ along these edges which can be done in $2$ ways depending the orientation of the edges. Therefore, the total number of graphs is given by,
\begin{multline*}
2\frac{m_{i_1}(m_{i_2}+m_{i_3})}{4}|\NC_2(m_{i_1})| \times\\
\left[|\NC_2(m_{i_2},m_{i_3})|-|\NC_2^{(1)}(m_{i_2},m_{i_3})|-|\NC_2^{(2)}(m_{i_2},m_{i_3})|\right].
\end{multline*}
The last expression is twice the number of partitioned permutations $(\cV,\pi_1\times\pi_2)\in \PSDUC$ with $\pi_2\in \NC_2(m_{i_1})$ and, $$\pi_1\in \NC_2(m_{i_2},m_{i_3})\setminus (\NC_2^{(1)}(m_{i_2},m_{i_3})\cup \NC_2^{(2)}(m_2,m_3)),$$ and we choose a cycle from each $\pi_1$ and $\pi_2$ and join them together to make a block of $\cV$.

Similarly we count the case where the unique \cycles of $\overline{T_{m_{i_2},m_{i_3}}^{\pi}}$ is a loop. $T_{m_{i_1}}^{\pi}$ can be chosen in $|\NC_2(m_{i_1})|$ ways, and $T_{m_{i_2},m_{i_3}}^{\pi}$ can be chosen in $|\NC_2^{(1)}(m_{i_2},m_{i_3})|$ ways. Then we choose an edge from each graph $T_{m_{i_1}}^{\pi}$ and $T_{m_{i_2},m_{i_3}}^{\pi}$. In the first graph there are $\frac{m_{i_1}}{2}$ choices, and in the second there are $\frac{m_{i_2}+m_{i_3}-2}{2}$ choices because the \cycles cannot be chosen as it is a loop. Once selected the edges we have two possible orientations for the union along the edges. Thus the total number of graphs is given by,
$$2\frac{m_{i_1}(m_{i_2}+m_{i_3}-2)}{4}|\NC_2(m_{i_1})||\NC_2^{(1)}(m_{i_2},m_{i_3})|.$$
The last expression is twice the number of partitioned permutations $(\cV,\pi_1\times\pi_2)\in \PSDUC$ whit $\pi_2\in \NC_2(m_{i_1})$, $\pi_1\in \NC_2^{(1)}(m_{i_2},\ab m_{i_3})$, and we choose a cycle from each $\pi_1$ and $\pi_2$ and join them together to make a block of $\cV$ with the restriction that the selected cycle of $\pi_1$ cannot be the unique through string of $\pi_1$.
Adding both cases up says that the number of $2$-$4$ unicircuit graphs, $T_{m_1,m_2,m_3}^{\pi}$, equals twice the number of partitioned permutations $(\cV,\pi_1\times\pi_2)\in \PSDUC$ with $\pi_1\in \NC_2(m_{i_2},m_{i_3})$, $\pi_2\in \NC_2(m_{i_1})$ and such that $\pi_1$ doesn't have two through strings and if $\pi_1$ has a single through string then this through string is never joined to another cycle of $\pi_2$ to make a block of $\cV$, i.e. $(\cV,\pi_1\times\pi_2)$ belongs to,
$$\PSDUC \setminus (\PSDUCTwoThroughStrings \cup\PSDUL).$$
Thus the number of all possible graphs $T_{m_1,m_2,m_3}^{\pi}\in \TFUCSet$ is twice the cardinality of the set,
$$\PSDUC \setminus (\PSDUCTwoThroughStrings \cup\PSDUL),$$
as desired.
\end{proof}

\section{The third order moments and cumulants}
\label{sec:free-cumulants}
We are ready to provide a proof to the main theorem, let us recall Corollary \ref{Corollary:FirstExpressionOfAlpha},
\begin{multline*}
\alpha_{m_1,m_2,m_3} = \sum_{\substack{\pi\in \cP(m) \\ T_{m_1,m_2,m_3}^{\pi}\in \TSTTSet}}(\C_6+6\C_4+2)
+ \sum_{\substack{\pi\in \cP(m) \\ T_{m_1,m_2,m_3}^{\pi}\in \TFFTTSet}}(\C_4+1)^2 \\
+ \sum_{\substack{\pi\in \cP(m) \\ T_{m_1,m_2,m_3}^{\pi}\in \TFULSet}}(\mathring{\C_4}+2)
+ \sum_{\substack{\pi\in \cP(m) \\ T_{m_1,m_2,m_3}^{\pi}\in \TFUCSet}}(\C_4+1)
+ \sum_{\substack{\pi\in \cP(m) \\ T_{m_1,m_2,m_3}^{\pi}\in \DBSet}}1. \\
\end{multline*}
On the other hand, Corollary \ref{Corollary:CountingDB} says,
$$|\DBSet|=|NC_2(m_1,m_2,m_3)|-|\TFUCSet|-|\TFFTTSet|-2|\TFULSet|-2|\TSTTSet|.$$
Combining these two expressions we get the following simpler expression,
\begin{multline}\label{Equation:SecondExpressionOfAlpha}
\alpha_{m_1,m_2,m_3} = \sum_{NC_2(m_1,m_2,m_3)}1 +\sum_{\substack{\pi\in \cP(m) \\ T_{m_1,m_2,m_3}^{\pi}\in \TSTTSet}}(\C_6+6\C_4) \\
+ \sum_{\substack{\pi\in \cP(m) \\ T_{m_1,m_2,m_3}^{\pi}\in \TFFTTSet}}(\C_4^2+2\C_4)
+ \sum_{\substack{\pi\in \cP(m) \\ T_{m_1,m_2,m_3}^{\pi}\in \TFULSet}}\mathring{\C}_4
+ \sum_{\substack{\pi\in \cP(m) \\ T_{m_1,m_2,m_3}^{\pi}\in \TFUCSet}}\C_4
\end{multline}

\begin{theorem}\label{Corollary:ThirdExpressionOfAlpha}
Let $m_1,m_2,m_3\in\mathbb{N}$. Then,
\begin{multline*}
\alpha_{m_1,m_2,m_3} = |NC_2(m_1,m_2,m_3)|+4\C_6|\PS| \\
+4\C_4^2|\PSS|+2\C_4|\PSSSSS| \\
+(\mathring{\C}_4-2\C_4)|\PSSS|
\end{multline*}
\end{theorem}
\begin{proof}
Lemmas \ref{Lemma:Counting2-6TT}, \ref{Lemma:Counting2-4-4TT}, \ref{Lemma:Counting2-4-U} and \ref{Lemma:Counting2-4-UL} let us count the sets $\TSTTSet, \ab \TFFTTSet, \ab \TFULSet$ and $\TFUCSet$; combining this with Equation (\ref{Equation:SecondExpressionOfAlpha}) gives,
\begin{multline}\label{Equation_Number5}
\alpha_{m_1,m_2,m_3} = |NC_2(m_1,m_2,m_3)|+4\C_6|\PS| \\
+4\C_4^2|\PSS|+2\C_4|\PSSSSS| \\
+(\mathring{\C}_4-2\C_4)|\PSSSS|+2\C_4 R,
\end{multline}
where $R$ is given by,
\begin{multline*}
12|\PS|+4|\PSS|\\
\mbox{} -
|\PSSSSSS|.
\end{multline*}
We will prove that $R=0$. We know that,
\begin{multline*}
12|\PS| \\
=
12\frac{m_1}{2}\frac{m_2}{2}\frac{m_3}{2}|NC_2(m_1)||NC_2(m_2)||NC_2(m_3)|,
\end{multline*}
and,
\begin{multline*}
4|\PSS|= \\
4\frac{m_1}{2}\frac{m_2}{2}\frac{m_3}{2}(\frac{m}{2}-3)|NC_2(m_1)||NC_2(m_2)||NC_2(m_3)|,
\end{multline*}
whit $m=m_1+m_2+m_3$. Thus,
\begin{multline}\label{Equation_number7}
12|\PS|+4|\PSS|= \\
\frac{mm_1m_2m_3}{4}|NC_2(m_1)||NC_2(m_2)||NC_2(m_3)|
\end{multline}
Now we compute $|\PSSSSSS|$. Each element $(\cV,\pi_1\times\pi_2)\in \PSSSSSS$ is such that we have $\pi_1\in NC_2^{(2)}(m_{i_1},m_{i_1})$ and 
$\pi_2  \in \ab  NC_2(m_{i_3})$ 
for some permutation $(i_1,i_2,i_3)$ of $(1,2,3)$. Assume $\pi_1\in NC_2^{(2)}(m_1,m_2)$ and $\pi_2\in NC_2(m_3)$, the partition $\cV$ is such that each block of $\cV$ is a cycle of $\pi$ except by one block which is the union of two cycles of $\pi$ one from each $\pi_1$ and $\pi_2$. Those cycles can be chosen in $\frac{m_1+m_2}{2}\frac{m_3}{2}$ ways. Thus the number of those partitioned permutation is,
$$\frac{(m_1+m_2)m_3}{4}|NC_2(m_3)||NC_2^{(2)}(m_1,m_2)|.$$
The set $NC_2^{(2)}(m_1,m_2)$ can be counted easily. On each circle we choose two points corresponding to the through strings. By \cite[Lemma 13]{MSAW} that can be done in $\binom{m_1}{\frac{m_1}{2}-1}\binom{m_2}{\frac{m_2}{2}-1}$ ways. Then we join the points to make the two through strings, which can be done in two ways. Thus,
\begin{eqnarray*}
|NC_2^{(2)}(m_1,m_2)|&=& 2\binom{m_1}{\frac{m_1}{2}-1}\binom{m_2}{\frac{m_2}{2}-1} \\
&=&2\frac{m_1}{2}\frac{m_2}{2}\frac{1}{m_1/2}\binom{m_1}{m_1/2}\frac{1}{m_2/2}\binom{m_2}{m_2/2} \\ 
&=&\frac{m_1m_2}{2}|NC_2(m_1)||NC_2(m_2)|,
\end{eqnarray*}
so the total number of elements $(\cV,\pi_1\times\pi_2)\in \PSSSSSS$ with $\pi_1\in NC_2^{(2)}(m_1,m_2)$ and $\pi_2\in NC_2(m_3)$ is given by,
$$\frac{m_1m_2m_3(m_1+m_2)}{8}|NC_2(m_1)||NC_2(m_2)||NC_2(m_3)|.$$
In the same way we count the other two cases corresponding to $\pi_1\in NC_2^{(2)}(m_1,m_3)$ and $\pi_2\in NC_2(m_2)$ and $\pi_1\in NC_2^{(2)}(m_2,m_3)$ and $\pi_2\in NC_2(m_1)$. Adding all together up gives,
\begin{eqnarray*}
&&|\PSSSSSS| \\
&=& (2m_1+2m_2+2m_3)\frac{m_1m_2m_3}{8}|NC_2(m_1)|\, |NC_2(m_2)|\, |NC_2(m_3)| \\
&=&\frac{mm_1m_2m_3}{4} |NC_2(m_1)|\, |NC_2(m_2)|\, |NC_2(m_3)|.
\end{eqnarray*}
So Equation (\ref{Equation_number7}) implies $R=0$. This turns Equation (\ref{Equation_Number5}) into,
\begin{multline*}
\alpha_{m_1,m_2,m_3} = |NC_2(m_1,m_2,m_3)|+4\C_6|\PS| \\
+4\C_4^2|\PSS|+2\C_4|\PSSSSS| \\
+(\mathring{\C}_4-2\C_4)|\PSSSS|.
\end{multline*}
We conclude the proof by replacing $|\PSSSS|$ by \\
$|\PSSS|$, as we proved in Lemma \ref{Lemma:Counting2-4-UL} they are equal.
\end{proof}

\noindent
\textbf{Proof of main theorem.} As we have seen the third order fluctuation moments $\alpha_{m_1, m_2, m_3}$ are somewhat involved. However the higher order cumulants are very simple.

\begin{theorem}[Main Theorem]\label{Chapter4_Theorem_Third order cumulants}
The third order cumulants, $(\K_{p,q,r})_{p,q,r}$, of a Wigner Ensemble, $X$, are given by,
$$\K_{p,q,r}= \left\{ \begin{array}{lcc}
             4\C_6 &   if  & p=q=r=2 \\
             \mathring{\C}_4-2\C_4 &   if  & \{p,q,r\}=\{2,1,1\} \\
             0 &  otherwise 
             \end{array}
   \right.$$
\end{theorem}
\begin{proof}
Let us recall that, up to order two, the free cumulants of $X$ are given by; $\K_2=1, \K_{2,2}=2\C_4$ and $0$ otherwise. Let $(\K^\prime_n)_n,(\K^\prime_{p,q})_{p,q}$ and $(\K^\prime_{p,q,r})_{p,q,r}$ be the sequences defined by $\K_n^\prime=\K_n$ for all $n$, $\K^\prime_{p,q}=\K_{p,q}$ for all $p,q$ and $\K^\prime_{2,2,2}=4\C_6, \K^\prime_{2,1,1}=\mathring{\C}_4-2\C_4$ and $0$ otherwise. By definition $\K^\prime_{n}$ and $\K^\prime_{p,q}$ coincide with the free cumulants of first and second order $\K_n$ and $\K_{p,q}$. Therefore these sequences satisfy the moment-cumulant relations:
\begin{equation}\label{Eq8}
\alpha_{m}=\sum_{(\cV,\pi)\in \mathcal{PS}_{NC}(m)} \K^\prime_{(\cV,\pi)}
\end{equation}
\begin{equation}\label{Eq9}
\alpha_{m_1,m_2}=\sum_{(\cV,\pi)\in \mathcal{PS}_{NC}(m_1,m_2)} \K^\prime_{(\cV,\pi)}
\end{equation}
for all $m,m_1,m_2$. For a non-crossing partitioned permutation $$(\cV,\pi)\in \mathcal{PS}_{NC}(m_1,m_2,m_3),$$ let $\K^\prime_{(\cV,\pi)}$ be the multiplicative extension of $(\K^\prime_n)_n,(\K^\prime_{p,q})_{p,q}$ and $(\K^\prime_{p,q,r})_{p,q,r}$, i.e;
$$\K_{(\cV,\pi)}=\prod_{\substack{B\text{ block of }\cV \\ V_1,\dots,V_i \text{ cycles of } \pi \text{ with }V_i\subset B}}\K^\prime_{|V_1|,\dots,|V_i|}.$$
Observe that,
\begin{multline*}
\sum_{(\cV,\pi)\in \mathcal{PS}_{NC}(m_1,m_2,m_3)}\K^\prime_{(\cV,\pi)} = |\NC_2(m_1,m_2,m_3)| \\
+4\C_6|\PS| +4\C_4^2|\PSS| \\
+2\C_4|\PSDUC| +(\mathring{\C}_4-2\C_4)|\PSSS|.
\end{multline*}
According to Theorem \ref{Corollary:ThirdExpressionOfAlpha} last expression equals $\alpha_{m_1,m_2,m_3}$, so the sequences $(\K^\prime_n)_n,(\K^\prime_{p,q})_{p,q}$ and $(\K^\prime_{p,q,r})_{p,q,r}$ satisfy the moment-cumulant relation of order three, namely,
\begin{equation}\label{Eq10}
\alpha_{m_1,m_2,m_3}=\sum_{(\cV,\pi)\in \mathcal{PS}_{NC}(m_1,m_2,m_3)} \K^\prime_{(\cV,\pi)}.
\end{equation}
However the free cumulants $(\K_n)_n,(\K_{p,q})_{p,q}$ and $(\K_{p,q,r})_{p,q,r}$ are the unique sequences satisfying Equations (\ref{Eq8}), (\ref{Eq9}) and (\ref{Eq10}), so it must be $\K^\prime_{p,q,r}=\K_{p,q,r}$ as desired.
\end{proof}

\section*{Acknowledgements}
We would like to thank Roland Speicher for his comments and fruitful discussions while preparing this paper.

\thebottomline
\end{document}